%% file: main.tex
\documentclass{article}
\usepackage[english]{babel}
\usepackage{float}
\usepackage{amsmath, amssymb, amsthm, epsfig, color, mathtools, bbm, dsfont, mathrsfs, bm}
\usepackage[latin1]{inputenc}
\usepackage[T1]{fontenc}
\usepackage{indentfirst}
\usepackage[shortlabels]{enumitem}
\usepackage[title]{appendix}
\usepackage{lmodern}       
\usepackage{comment}  
\usepackage{textcomp, upgreek}
\usepackage{setspace}
\usepackage{soul}
\usepackage{accents}
\usepackage{hyperref}
\usepackage{geometry}
\usepackage{xcolor}
\allowdisplaybreaks[4]

\usepackage{xr}   
\usepackage{graphicx} 
\usepackage[export]{adjustbox}
\usepackage{subcaption} 
\usepackage{mathtools} 
\usepackage{siunitx}
\usepackage{bm} 
\usepackage{comment}

\usepackage{addfont}

\usepackage{tikz, tkz-euclide}
\usetikzlibrary{patterns}
\usetikzlibrary{arrows.meta}
\usetikzlibrary{calc}
\usetikzlibrary{decorations.markings}
\usepackage{fancybox}

\usepackage{graphicx}
\graphicspath{ {./images/} }
\newcommand{\Z}{\mathbb{Z}}

\newcommand{\lab}{\mathrm{lab}}

\newcommand{\I}{\mathrm{I}}
\newcommand{\Sp}{\mathrm{sp}}

\newcommand{\C}{\mathscr{C}}

\newcommand{\vareps}{\varepsilon}
\newcommand{\cA}{\mathcal{A}}
\newcommand{\B}{\mathrm{B}}
\newcommand{\Int}{\mathrm{Int}}
\def\fInt{\mathfrak{Int}}

\def\d{\mathop{\textrm{\rm d}}\nolimits}                  
\def\sign{\mathop{\textrm{\rm sign}}\nolimits}                  
\def\P{\mathbb{P}}
\def\ter{\textcolor{red}}
\def\teb{\textcolor{blue}}

\newcommand{\be}{\begin{equation}}
\newcommand{\ee}{\end{equation}}

\numberwithin{equation}{section}

  \newcounter{dummy} \numberwithin{dummy}{section}
  \theoremstyle{plain}
  \newtheorem*{theorem*}        {Theorem}
	\newtheorem*{conjecture*}   {Conjecture}
  \newtheorem{theorem}[dummy]          {Theorem}
  \newtheorem{lemma}[dummy]              {Lemma}
  \newtheorem*{lemma*}          {Lemma}
    \newtheorem{claim}[dummy]         {Claim}
  \newtheorem{corollary}[dummy]           {Corollary}
  \newtheorem{proposition}[dummy]       {Proposition}
   
  \newtheorem{remark}[dummy]           {Remark}
  
  \theoremstyle{remark}
  \theoremstyle{definition}
   \newtheorem{definition}[dummy]          {Definition}
\makeatletter

\newcommand\longleftrightarrowfill@{%
  \arrowfill@\leftarrow\relbar\rightarrow}
\makeatother

\definecolor{Red}{cmyk}{0,1,1,0}

\definecolor{Blue}{cmyk}{1,1,0,0}

\definecolor{DarkBlue}{rgb}{0.1,0.1,0.5}
\definecolor{Red}{rgb}{0.9,0.0,0.1}
\definecolor{DarkGreen}{rgb}{0.10,0.50,0.10}
\definecolor{DarkRed}{rgb}{0.50,0.10,0.10}
\definecolor{bleu}{RGB}{0,140,189}%

\DeclarePairedDelimiter\ceil{\lceil}{\rceil} 

\begin{document}

\begin{center}
{\LARGE Phase transitions in low-dimensional long-range random field Ising models}
\vskip.5cm
Jian Ding$^{1}$, Fenglin Huang$^{1}$, Jo{\~a}o Maia$^{2}$
\vskip.3cm
\begin{footnotesize}
$^{1}$ School of Mathematical Sciences, Peking University, China\\
$^{2}$ Beijing International Center for Mathematical Research, Peking University, China\\
\end{footnotesize}
\vskip.1cm
\end{center}

\begin{abstract}
We consider the long-range random field Ising model in dimension $d = 1, 2$, whereas the long-range interaction is of the form $J_{xy} = |x-y|^{-\alpha}$ with $1< \alpha < 3/2$ for $d=1$ and with $2 < \alpha \leq 3$ for $d = 2$. Our main results establish phase transitions in these regimes. In one dimension, we employ a Peierls argument with some novel modification, suitable for dealing with the randomness coming from the external field; in two dimensions, our proof follows that of Affonso, Bissacot, and Maia (2023) with some adaptations, but new ideas are required in the critical case of $\alpha=3$. 
\end{abstract}

\section{Introduction}

The Ising model is one of the most studied models in statistical mechanics. In this paper, we work with the combination of two variations of the Ising model, the random field and the long-range models. The long-range variation, introduced first by Kac and Thompson \cite{Kac_Thompson_69}, considers a model on the lattice $\Z^d$ with the same configuration space $\Omega \coloneqq \{-1, +1\}^{\Z^d}$ (as in the nearest-neighbor case), but with a long-range interaction  $\{J_{xy}\}_{x,y\in\Z^d}$ given by 
\begin{equation}\label{Long-Range Interaction}
  J_{xy} = \begin{cases}
                 \frac{1}{|x-y|^\alpha} &\text{ if }x\neq y,\\
                 0                &\text{otherwise,} 
            \end{cases}
\end{equation}
with $\alpha>d$ ensuring that the model is regular. A second variation, the random field Ising model (RFIM), is the usual Ising model with the addition of a random external field. Our model of interest is the long-range random field Ising model, which has a \textit{local Hamiltonian} given by
\begin{equation}\label{Eq: Hamiltonian}
  H_{\Lambda; \varepsilon h}^\eta(\sigma) \coloneqq -\sum_{x,y\in\Lambda} J_{xy}\sigma_x\sigma_y - \sum_{x\in \Lambda, y\in\Lambda^c} J_{xy}\sigma_x\eta_y - \sum_{x\in\Lambda} \varepsilon h_x\sigma_x,
\end{equation}
where $\Lambda\subset\Z^d$ is finite, $\eta\in\Omega$ is the \textit{boundary condition}, $\{h_x\}_{x\in\Z^d}$ is a collection of i.i.d. Gaussian random variables in a probability space denoted as $(\widetilde{\Omega}, \mathcal{A}, \mathbb{P})$ and the parameter $\varepsilon >0$ controls the variance of the external field. 
As in the usual Ising model, the local Gibbs measure is a probability measure given by \begin{equation*}
  \mu_{\Lambda;\beta, \varepsilon h}^\eta(\sigma) \coloneqq \mathbbm{1}_{\Omega_\Lambda^\eta}(\sigma)\frac{e^{-\beta H_{\Lambda; \varepsilon h}^{\eta}(\sigma)}}{Z_{\Lambda; \beta, \varepsilon}^{\eta}(h)},
\end{equation*}
where $Z_{\Lambda; \beta, \varepsilon}^{\eta}(h)$ is a normalizing constant and $\Omega_\Lambda^\eta \coloneqq \{\omega\in\Omega : \omega_x=\eta_x, \forall x\in\Lambda^c\}$ is the collection of the configurations matching $\eta$ outside $\Lambda$. For the extremal boundary conditions $\eta \equiv +1$ ($\eta_{x}= +1, \forall x \in \mathbb{Z}^{d})$, and analogously, $\eta \equiv -1$, the weak*-limits
\begin{equation*}
    \mu_{\beta,\varepsilon h}^{\pm}[\sigma] \coloneqq \lim_{N\to\infty} \mu_{[-N,N]^d;\beta, \varepsilon h}^{\pm}[\sigma],
\end{equation*}
are well-defined (see \cite[Theorem 7.2.2]{Bovier.06})  and are our main objects of study. When the limiting measures are distinct, i.e., $\mu^+_{\beta, \epsilon h} \neq \mu^-_{\beta, \varepsilon h}$ $\mathbb{P}$-a.s. for some $\beta>0$, we say there is \textit{phase transition}, also called \textit{long-range order}.    

Both the long-range model and the RFIM have a rich history in statistical physics. Determining for which dimensions $d$ the RFIM presents a phase transition was a topic of heated debate in the physics community until it was settled in the seminal works of Aizenmann and Wehr \cite{Aizenman.Wehr.90} and Bricmont and Kupiainen \cite{Bricmont.Kupiainen.88}. Aizenmann and Wehr showed that the model has only one Gibbs measure in dimensions $d=1,2$, while Bricmont and Kupiainen proved phase transition for $d\geq 3$. Several interesting phenomena occur in the RFIM, and interesting results were obtained in recent years: we mention for example \cite{Aizenman_Harel_Peled_20, AP19, Cha18, ding2023phase, Ding_Wirth_23, Ding_Xia_21} for results in two dimensions, and \cite{Ding_Liu_Xia_24, Ding2021} for three and higher dimensions.

For the long-range model in one dimension, Dobrushin and Ruelle proved that there is no phase transition if $\alpha>2$ \cite{Dobrushin69, Ruelle68}. On the other hand, Kac and Thompson \cite{Kac_Thompson_69} conjectured that the long-range interaction should induce a phase transition, different from the Ising model. This was proved by Dyson \cite{Dyson.69} for $1<\alpha<2$; for this reason, the model is also known as the \textit{Dyson model}. For the critical exponent $\alpha=2$, phase transition was proved in the seminal work of Fr\"ohlich and Spencer \cite{Frohlich.Spencer.82}, where the authors introduced a notion of long-range one-dimensional contours.  This contour system was extended to the multidimensional setting by Affonso, Bissacot, Endo and Handa \cite{affonso_24}, where the authors proved phase transition at low temperatures, even in the presence of a decaying field.  

For the long-range random field Ising model, very little was known until very recently. In low dimensions, the same work of Aizenman and Wehr \cite{Aizenman.Wehr.90} proved uniqueness for $d=1,2$ and $3d/2<\alpha$, and in one dimension they also proved the uniqueness for the critical case of $\alpha=3/2$. Phase transition for the nearest-neighbor Ising model with no external field implies phase transition for the long-range model, again with no external field, by standard correlation inequalities. However, a different argument is required for models with a random field, since some correlation inequalities are lost in this setting.  A recent work of Affonso, Bissacot, and Maia \cite{Affonso_Bissacot_Maia_23} proved phase transition for dimensions $d\geq 3$ in the entire region $\alpha>d$. In one dimension, the only known result prior to this work was proved by Cassandro, Orlandi, and Picco \cite{Cassandro.Picco.09}, with the extra restrictions that $3-\frac{\ln 3}{\ln 2} < \alpha < 3/2$ and $J(1)$ large enough, where $J(1)$ is the nearest-neighbor interaction.

A natural follow-up question is if the restriction of $J(1)$ being large can be lifted since this is not required in the long-range model with no field. Another question of interest, as pointed out by \cite{Aizenman_Greenblatt_Lebowitz_2012}, is if $\alpha = 3$ is indeed the critical exponent for phase transition in two dimensions in the presence of a random field. In this paper we answer both these questions, completing the picture of phase transition. 

For the long-range random field Ising model, we prove phase transition in one dimension for $1<\alpha<3/2$, without the assumption of $J(1)$ being large. We also show that there is a phase transition in two dimensions for all $2<\alpha\leq 3$, including the critical exponent $\alpha = 3$. Notice that this is distinct from the one-dimensional critical case $\alpha = 3/2$, where there is uniqueness at all temperatures. Our results can be summarized in the next two theorems. 

\begin{theorem}\label{thm-main} Fix $d=1$ and $1<\alpha<\frac{3}{2}$. For any constant $c_1>0$,
there exist constants $\vareps_1\coloneqq \vareps_1(\alpha,c_1)>0$ and $\beta_1\coloneqq \beta_1(\alpha, c_1)>0$ such that for any $\varepsilon < \varepsilon_1$ and any $\beta>\beta_1$, the following holds with $\mathbb P$-probability at least $1-c_1$:
  \begin{equation*}
    \mu^+_{\beta, \varepsilon h}(\sigma_0=-1)<c_1.
  \end{equation*}
\end{theorem}

\begin{theorem}\label{thm-main-2d} Fix $d=2$ and $2<\alpha\le3$.  For any constant $c_2>0$,
there exist constants $\vareps_2\coloneqq \vareps_2(\alpha,c_2)>0$ and $\beta_2\coloneqq \beta_2(\alpha, c_2)>0$ such that for any $\varepsilon < \varepsilon_2$ and any $\beta>\beta_2$, the following holds with $\mathbb P$-probability at least $1-c_2$: 
  \begin{equation*}
    \mu^+_{\beta, \varepsilon h}(\sigma_0=-1)<c_2.
  \end{equation*}
\end{theorem}

Both previous works on the long-range RFIM employed a notion of contour suitable for the long-range Ising model with no field and then made the necessary adaptations to control the external field. Despite the notions of contours in \cite{Cassandro.Picco.09} and \cite{Affonso_Bissacot_Maia_23} being different, both of them were inspired by the contours of Fr\"ohlich and Spencer \cite{Frohlich.Spencer.82}.

In one dimension, it does not seem possible to use the same contours as in \cite{Cassandro.Picco.09} and remove the condition on $J(1)$. Removing the condition on $J(1)$ was achieved in \cite{Bissacot_Endo_Enter_Kimura_18} for the model with no field, but an ever stronger restriction on $\alpha$ where $3/2<\alpha^*<\alpha<2$ is required. On the other hand, one could try to use the same notion of contour as in \cite{Frohlich.Spencer.82}, since they were recently extended to the entire region $1<\alpha\leq2$ in \cite{Affonso_Bissacot_Corsini_Welsch_24}, again for the model with no random field. This however, would require extending the entropy arguments of \cite{Affonso_Bissacot_Maia_23} to one dimension, which seems possible but not trivial, and then applying our improved coarse-graining procedure as presented in Section~\ref{sec: 2d phase transition}.

We opt to use a different approach, introducing a new Peierls-type argument, tailored for long-range interactions with random field. Our construction of the Peierls map resembles that of \cite{Frohlich.Spencer.82}, and it is also done in a multiscale fashion. A noticeable difference is that we do not rely on spin flips, only on the local densities of the plus and minus spins in the configuration. Using our construction, the entropy arguments are simpler, and a coarse-graining argument similar to that in \cite{FFS84} is enough to control the field. In addition, we can also show phase transition for $\alpha\in (1,2)$ when there is no external field, which gives an alternative proof to \cite{Affonso_Bissacot_Corsini_Welsch_24} (see Remark~\ref{rmk: extension to pure 1d model}).

\subsection{Proof strategy in one dimension}
Our argument in one dimension follows the general principle of a Peierls argument: denoting by $\Lambda\Subset \Z$ a finite subset of $\Z$, for each configuration $\sigma\in\Omega^+\coloneqq  \cup_{\Lambda\Subset \Z}\Omega_\Lambda^+$ with $\sigma_0=-1$, we map $\sigma$ to a configuration $\omega$ with $\omega_0=1$ and compare their energies. Our first intuition for the Peierls map is to flip all the minuses in a suitable interval containing the origin. However, two difficulties arise:
\begin{enumerate}
    \item Spins now interact with all the others, so it is hard to define what is a suitable interval.
    \item If we flip all the minuses in an interval, the minus spins surrounded by many pluses result in a huge entropy that we can not control.
\end{enumerate}
 
To overcome the difficulties above, especially the second one, in Section~\ref{sec: balancing procedure} we define a balancing procedure. The intuition is that minus spins surrounded by a vast majority of plus spins are, in a sense, `isolated', so they should be treated as plus spins and vice-versa. We will continue this procedure until no such spin appears. In the end, we will flip all the minus spins left in a suitable interval to gain energy for which we can control the entropy.

In Lemma~\ref{Lemma: interacton_2}, we establish the key energy bound, which essentially states that for a configuration without `isolated' spins, there will be a significant number of plus spins as long as the configuration contains at least one plus spin. Lemma~\ref{Lemma: Sequence_01} plays a crucial role in the proof of Lemma~\ref{Lemma: interacton_2}, and its validity heavily depends on the one-dimensional structure.

Then in  Proposition \ref{prop: entropy argument} we provide the necessary entropy estimations to do a Peierls argument. To control the random field we use the approach of \cite{Ding2021}, so phase transition follows after we control the probability of some bad events (see \eqref{bad_event_1d}). This is done in Proposition \ref{prop: free energy control}, using the coarse-graining argument of \cite{FFS84} together with the previously established entropy bounds. 

\subsection{Proof strategy in two dimensions}
The argument of \cite{Affonso_Bissacot_Maia_23} applies the strategy of \cite{Ding2021} to the long-range model, and the main complication is to extend the coarse-graining argument of \cite{FFS84} to long-range contours, which are in general not connected. Their approach is robust, which allows us to consider the same contour system in two dimensions. For $2<\alpha<3$, a simple modification of their argument suffices to show phase transition. Such adaption, however, does not suffice to prove the critical case $\alpha=3$. A rough intuition underlying the existence of phase transition is the following Imry-Ma argument \cite{Imry.Ma.75}: the contribution of the external field from any $A\subset \Z^2$ should be of order $\sqrt{|A|}$, so there is phase transition as long as the energy contribution of flipping the spins is larger than $\sqrt{|A|}$. For the box $A = [-N,N]^2$, this contribution is at least $N^{4-\alpha}$, which beats $\sqrt{|A|}=N$ only when $\alpha<3$. However, when   $\alpha = 3$ the interaction between spins inside and outside $A$ is actually of order $N\log(N)$, see \cite{Biskup_Chayes_Kivelson_07}. So, for the critical case, we must introduce some improved energy estimations, which capture the correct order of the interaction, see Lemma \ref{Lemma: Int_boundary_cubes}. 

Even with this improvement, the same coarse-graining argument does not hold. In order to compare the entropy bound with the energy bound we need to introduce a finer coarse-graining procedure. 
Instead of directly using the entropy bound of \cite{Affonso_Bissacot_Maia_23}, we give our entropy bound according to its coarse-graining level. This improvement allows us to apply a union bound over all levels and control the enumeration coming from each level.

The structure of the paper is as follows. Section~\ref{sec: 1d phase transition} contains the proof of Theorem \ref{thm-main}. In Section~\ref{sec: balancing procedure} we define our Peierls map in one dimension. In Section \ref{sec: balancing_properties}, we analyze the Peierls map and in Section~\ref{sec: proof of main thm} we present the outline of our argument and prove Theorem \ref{thm-main} assuming some required estimations on the energy and entropy. Sections~\ref{sec: energy bound} and \ref{sec: entropy bound} are devoted to proving such estimations. Section \ref{sec: 2d phase transition} contains the proof of Theorem \ref{thm-main-2d}, and it is rather independent from Section~\ref{sec: 1d phase transition}. In Section~\ref{sec: 2d not crtitical}, we prove the case of $2<\alpha<3$ and in Section~\ref{sec: 2d crtitical} we prove the critical case $\alpha=3$.

\section{Phase Transition in 1d}\label{sec: 1d phase transition}

In this section, we focus on the case $d=1$ and prove Theorem \ref{thm-main}. In Section \ref{sec: balancing procedure}, we construct the Peierls map, whereas a key ingredient is a balancing procedure; in Section \ref{sec: balancing_properties}, we analyze the Peierls map and prove some properties of the balancing procedure, and in Section~\ref{sec: proof of main thm} we prove Theorem \ref{thm-main} assuming some suitable bounds on the energy and the entropy; in Sections \ref{sec: energy bound} and \ref{sec: entropy bound}, we provide the aforementioned bounds on the energy and the entropy respectively.

Throughout this section, we fix $d=1$, $1<\alpha<\frac{3}{2}$ and $\Lambda\Subset\Z$ containing the origin. We also fix the following constants: $c_1\geq 10$ an arbitrary constant, $\delta\coloneqq \min\{0.001,\frac{1.5-\alpha}{20}\}>0$, 
and $M_0\coloneqq M_0(\alpha, c_1)>2$ large enough, whose value will be chosen later. 
We reserve notations $C_1, \ C_2, \ \dots$ for arbitrary constants, possibly taking different values in different proofs, but depending only on $M_0$, $\delta$, or $\alpha$. In addition, we use $b_1,\ c_1,\ b_2, \ c_2, \ \dots$ and $\overline{c}_1, \ \overline{c}_2, \ \dots$ for constants whose values are fixed (throughout the paper) upon their first occurrences.

\subsection{The construction of the Peierls map}\label{sec: balancing procedure}

Given a configuration $\sigma\in\Omega^+_\Lambda$ such that $\sigma_0=-1$, we aim to determine a set of sites $A\subset \Lambda$ such that our Peierls map will obtain a configuration $\tau_A(\sigma)$ from $\sigma$ by flipping spins in $A$. Such a set has to be carefully chosen so that on the one hand we can compare the energy change from $\sigma$ to $\tau_A(\sigma)$, and on the other hand, the multiplicity of the Peierls map (i.e., the entropy loss) can be controlled. To this end, we will employ a balancing procedure  (see Definition \ref{def: Peierls map} below). 

As a notation convention, when writing an interval $[a,b)$ we always mean its restriction to $\Z$, $[a,b)\cap \Z$. At each scale $\ell\geq 0$, we consider intervals $\I_\ell(x) = [2^{\ell-4}x-2^{\ell-1}, 2^{\ell - 4}x+2^{\ell-1})$, with $x\in\Z$ and side length $2^\ell$. Each such interval will be called an $\ell$-interval, and we write $\I_\ell$ (i.e., omitting the center $x$) to denote an arbitrary $\ell$-interval. The family of all such intervals is denoted as $\mathcal{I}\coloneqq \{\I_\ell(x)\}_{x\in\Z, \ell\geq 0}$, and the restriction to $\ell$-intervals is $\mathcal{I}_\ell\coloneqq \{\I_\ell(x)\}_{x\in\Z}$. We can further split each $\mathcal{I}_\ell$ into 16 sub-collections \begin{equation}\label{eq: sub-collection of disjoint coverings}
    \mathcal{I}_\ell^i\coloneqq\{[2^{\ell-4}(16x+i), 2^{\ell-4}(16x+16+i))\}_{x\in\Z},~i=0,1, \dots 15,
\end{equation} whereas each sub-collection forms a partition of $\Z$. 

For any configuration $\sigma\in\Omega$ and any interval $\I$, we split $\I$ into its positive and negative sites $\I^\pm(\sigma) \coloneqq \{x\in \I : \sigma_x = \pm1\}$. We say that $\I_\ell$ is \textit{minus dense} with respect to a configuration $\sigma$ when $\frac{|\I^-_\ell(\sigma)|}{|\I_\ell|}> 1-\frac{1}{M_\ell}$, where $M_\ell \coloneqq M_02^{\delta \ell}$. When the minus density satisfies  $\frac{|\I^-_\ell(\sigma)|}{|\I_\ell|}\le \frac{1}{M_\ell}$, there are very few minuses in the interval, so we say $\I_\ell$ is {\textit{minus vacant}}. If the density satisfies $\frac{|\I^-_\ell(\sigma)|}{|\I_\ell|}> \frac{1}{M_\ell}$, we say $\I_\ell$ is \textit{minus occupied}. We remark that a minus dense interval is also a minus occupied interval. In a completely analogous fashion, we define the \textit{plus dense}, \textit{plus vacant}, and \textit{plus occupied} intervals. Notice that the notions of occupied are not exclusionary, that is, an interval $\I_\ell$ can be both minus occupied and plus occupied since $\frac{1}{M_\ell}\leq \frac{1}{2}$. However, the dense and vacant properties are complementary: a plus vacant interval has minus density larger than $(1 - \frac{1}{M_\ell})$, and it is therefore minus dense.

In the Peierls argument, if $\I_\ell\owns 0$ is a minus dense interval, the natural attempt is to compare $\sigma$ with a configuration after we flip all minus spins in $\I_\ell$. The long-range interaction, however, makes this very costly since there can be many other minus dense intervals near $\I_\ell$. We therefore need a notion to identify when $\I_\ell$ is isolated. 

\begin{definition}\label{def: favored interval}
An interval $\I_\ell(x)$ is \textit{plus favored} with respect to $\sigma$ if it satisfies the following two conditions:
\begin{enumerate}
    \item[\textbf{(I)}] $\sigma_y=1$ whenever $1\leq d(y, \I_\ell(x))\le 2^{\ell - 1}$ (this will be used in e.g. Lemma \ref{Lemma: interacton_close_to_I} later). 
    \item[\textbf{(II)}] All the $2M_\ell$ closest neighboring $\ell$-intervals are plus dense, that is, for all $1\leq k\leq M_\ell$,  $\I_\ell(x
    {+16k})$ and $\I_\ell(x{-16k})$ are plus dense. 
\end{enumerate}
We may also consider the following weaker notion of dense and favored:  $\I_\ell$ is \textit{weakly} plus dense (with respect to $\sigma$) when $|\I^+_\ell(\sigma)|> 2^\ell\left(1-\frac{c_1}{M_\ell}\right)$, and $\I_\ell$ is \textit{weakly} \textit{plus favored} (with respect to $\sigma$) if it satisfies \textbf{(I)} and \textbf{(II)} after we replace $M_{\ell}$ by  $M^\prime_\ell \coloneqq 2\left\lfloor\frac{M_\ell}{2c_1}\right\rfloor$ (so $M^\prime_\ell$ is even, which is only for notational convenience later when divided by $2$) and replace plus dense by weakly plus dense, where $c_1\geq 10$ is a constant fixed at the beginning of this section. See Figure~\ref{Fig: plus_favored} for an illustration of plus favored interval. Similarly, we can define the \textit{minus favored}, \textit{weakly minus dense}, and \textit{weakly} \textit{minus favored} intervals. 
\end{definition}

\begin{figure}[ht]
    \centering
    \input{Fig_plus-favored}
    \caption{The figure depicts a configuration $\sigma$ where the red cubes denote plus spins and blue cubes denote minus spins. Taking $M_0=1$ and $\delta=\frac{1}{3}$, the interval $\I_3(x)$ is plus favored, since its $2^{3-1} = 4$ closest neighbors in both directions are positive, and the number of minuses in each interval of the form $\I_3(x \pm 16)$ and $\I_3(x\pm 32)$ is smaller than $\frac{2^3}{2^{3\cdot\frac{1}{3}}} = 4$. Moreover, as there are minuses in $\I_3(x)$ and it is plus favored, it is also isolated. }
    \label{Fig: plus_favored}
\end{figure}
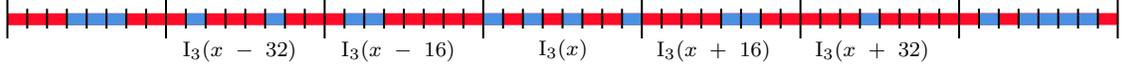

\begin{definition}
    Given any $r\geq 1$  and any interval $\I\subset \Z$, we define $\rho_r(\I)\coloneqq \{x\in\Z: d(x, \I)\leq (r-1)|\I|\}$, the expansion of $\I$ by a $2r-1$ factor. 
\end{definition}

\begin{definition}\label{def: isolated interval}
    An interval $\I_\ell$  is \textit{isolated} with respect to $\sigma$ when it is plus favored and $\I^-_\ell(\sigma)\neq \emptyset$, or minus favored and $\I_\ell^+(\sigma)\neq \emptyset$. Similarly, we define the \textit{weakly isolated} intervals. We say $\sigma$ is \textit{balanced} in $\I \subset \Z$ 
    if for every $\ell\geq 0$,  every $\ell$-interval $\I^\prime\in\mathcal{I}_\ell$ with $\rho_{M_\ell}(\I^\prime)\subset\I$ is not isolated with respect to $\sigma$.
    A plus favored isolated interval will be called a \textit{plus isolated} interval.
    Similarly, we define the \textit{minus isolated} interval.
    
We denote by $\C_\ell^\pm(\sigma)$ the plus/minus isolated intervals in the $\ell$-th scale. We also denote $\C^\pm(\sigma) \coloneqq \cup_{\ell\geq 0}\C_\ell^\pm(\sigma)$ and  $\C(\sigma)\coloneqq \C^+(\sigma)\cup \C^-(\sigma)$. In addition, let  $\C^{+,0}(\sigma)\coloneqq \{ \I \in \C^+(\sigma) :  0\in\I \}$ be the set of all plus isolated intervals in $\sigma$ containing the origin.
\end{definition} 
 To determine the spins of $\sigma$ to be flipped, we transform it into a balanced configuration by erasing all isolated intervals, one by one. Next, we precisely describe this erasing process (as we will see below, we erase a plus/minus isolated interval by flipping all minuses/pluses therein).

For any $\sigma\in\Omega$ and $A\subset \Z$, let us denote by 
$\tau_A(\sigma)$ the configuration $\sigma$ after we flip all spins in $A$, so $\tau_A(\sigma) \coloneqq ( (-1)^{\mathbbm{1}_{A}(x)}\sigma_x)_{x\in\Z}$.
Given any $\sigma\in\Omega^{+}$, at each step $s\geq 0$, starting with the configuration $\sigma^0\coloneqq \sigma$, we employ the following \emph{balancing procedure}.
\begin{enumerate}
    \item Take $\ell\geq 0$ the smallest integer such that there exists an isolated $\ell$-interval $\I_\ell$ with $\I_\ell\in \C(\sigma^{s})\setminus\C^{+,0}(\sigma^{s})$ (that is, if $\I_\ell$ is a plus isolated interval then it cannot contain the origin; this is because we do not want to flip the origin). If there is more than one isolated $\ell$-interval in $ \C(\sigma^{s})\setminus\C^{+,0}(\sigma^{s})$, pick the leftmost one and denote it by $\B_s$; if there is no such isolated interval, we then stop our procedure.
    \item Flip all minuses in $\B_s$ if it is a plus isolated interval, and flip all pluses if it is minus isolated. That is, we take  $\sigma^{s+1} \coloneqq \tau_{\B_s^+(\sigma^{s})}(\sigma^{s})\mathbbm{1}_{\{\B_s\in\C^-\}} + \tau_{\B_s^-(\sigma^{s})}(\sigma^{s})\mathbbm{1}_{\{\B_s\in\C^+\}}$ (recall that $\B^\pm(\sigma) = \{v\in \B: \sigma_v = \pm 1\}$ for any $\B \subset \Z$). 
\end{enumerate}

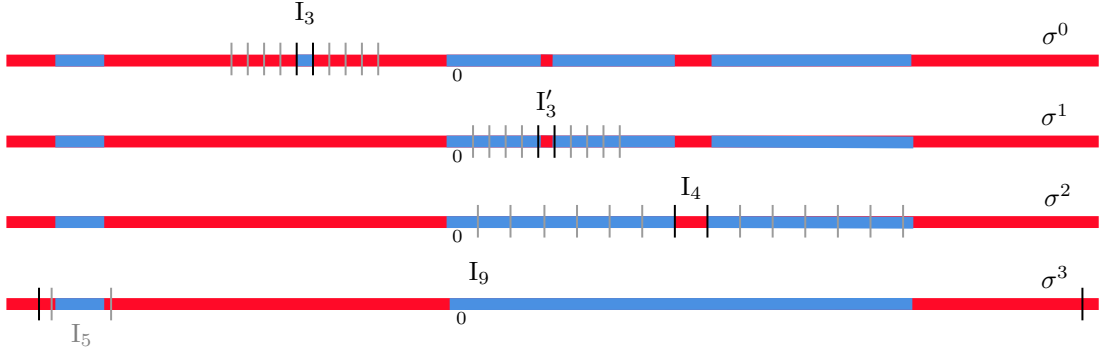
\begin{figure}[ht]
    \centering
    \input{Fig_balancing_procedure}
    \vspace{-1cm}
    \caption{An illustration of the balancing procedure. As in Figure \ref{Fig: plus_favored}, red regions denote plus spins, while blue regions denote minus spins. We consider $M_0=1$ and $\delta = \frac{2}{3}$. Starting from configuration $\sigma^0$, the first isolated interval selected is $\I_3$. As it is plus favored, it is flipped to plus. Next, $\I_3^\prime$ is selected and flipped to minus for an analogous reason. Then, $\I_4$ is selected and flipped to minus. Notice that $\I_4$ is isolated with respect to $\sigma^2$ but not in $\sigma^1$, since $M_4 = 2^{\frac{8}{3}} \approx 6.34$ and all its first $6$ nearest-neighboring intervals are minus dense with respect to $\sigma^2$ but not $\sigma^1$. Notice that, in $\sigma^3$, $\I_5$ is not plus favored, since $M_5 = 2^{\frac{10}{3}}\approx 10.07$, but not all of its 10 closest neighbors are plus dense. The procedure stops at $\sigma^3$, since the only isolated interval ($\I_9$) is plus favored but contains 0.}
    \label{Fig: balancing_procedure}
\end{figure}

At each step, the balancing procedure selects one isolated interval and erases it. See Figure \ref{Fig: balancing_procedure} for an illustration. Lemma \ref{lem: procedure stops} below assures that our procedure will stop in a finite time. 
\begin{definition}\label{def: Peierls map}
    We let $S\coloneqq S(\sigma)$ be the last step of the balancing procedure of $\sigma$ and let $\I_{\sigma}$ be the smallest isolated interval with respect to $\sigma^S$ containing the origin (if there is more than one, pick the leftmost one; we let $\I_\sigma = \emptyset$ if such an interval does not exist). The spins to be flipped in the Peierls argument are $A_{\sigma}\coloneqq \I_\sigma^-(\sigma^S)$. 
\end{definition}
\subsection{The analysis of the Peierls map}\label{sec: balancing_properties}
In this subsection, we analyze the Peierls map constructed in Section \ref{sec: balancing procedure}. To this end, we first show that the balancing procedure eventually stops, for which the following lemma will be useful.
For any interval $\B$ selected in the balancing procedure, let $s_{\B}$ denote the first step in which it was selected. The following lemma shows that $\B$ will be selected only once and thus $s_{\B}$ is the only step in which $\B$ was selected.
\begin{lemma}\label{Lemma: aux_1}
    Let $\I$ be an interval selected at step $s_{\I}$ (in the balancing procedure of $\sigma$). Then for any $k\ge 1$, $\sigma^{s_\I+k}\vert_{\rho_{{\frac{3}{2}}}{(\I)}}$ is constant (here $\sigma|_{\cdot}$ means the restriction of $\sigma$).
    Furthermore, for all intervals $\B$ selected in step $s_\B>s_\I$, we have $\B\nsubseteq \rho_{{\frac{3}{2}}}{(\I)}$ or $\B=\rho_{{\frac{3}{2}}}{(\I)}$.
\end{lemma}
\begin{proof}
By symmetry, we assume that $\I$ is flipped to plus at step $s_\I.$
   Note that $\sigma^{s_\I+1}\vert_{\rho_{{\frac{3}{2}}}{(\I)}}=+1$ for the following reasons: for $x\in \rho_{\frac{3}{2}}(\I)\setminus\I$ we have $\sigma^{s_\I} = 1$ by Definition \ref{def: favored interval} \textbf{(I)} (this obviously holds at step $s_\I + 1$); for $x\in \I$, we have $\sigma^{s_\I+1}_x = 1$ by the definition of the balancing procedure. Suppose now there is a step $s>s_\I$ such that $\sigma^{s}\vert_{\rho_{{\frac{3}{2}}}{(\I)}}=u$ is constant and $\sigma^{s+1}\vert_{\rho_{{\frac{3}{2}}}{(\I)}}$ is not constant. Let $\B$ be the interval selected at step $s$ in the balancing procedure. Since we flipped $\B$ to $-u$, the neighboring sites of $\B$ must have sign $-u$ in $\sigma^s$, by the definition of our balancing procedure. As $\sigma^{s}\vert_{\rho_{{\frac{3}{2}}}{(\I)}}=u$, these neighboring sites of $\B$ cannot be contained in $\rho_{{\frac{3}{2}}}{(\I)}$. Hence we have either $\rho_{\frac{3}{2}}(\I)\cap \B=\emptyset$ or $\rho_{\frac{3}{2}}(\I)\subset \B$. This contradicts the fact that $\sigma^{s}\vert_{\rho_{{\frac{3}{2}}}{(\I)}}$ is constant and $\sigma^{s+1}\vert_{\rho_{{\frac{3}{2}}}{(\I)}}$ is not constant.
   
   In addition, if we flipped some interval $\B$ to $u\in \{-1,1\}$, then $\B$ must have two neighboring $u$ sites and there must be a $-u$ spin in $\B$. Since $\sigma^{s_\I+k}\vert_{\rho_{{\frac{3}{2}}}{(\I)}}$ is always constant for $k\ge 1$, we get that $\B\nsubseteq \rho_{{\frac{3}{2}}}{(\I)}$ or $\B=\rho_{{\frac{3}{2}}}{(\I)}$.
   Altogether, this completes the proof of the lemma.
\end{proof}
\begin{lemma}\label{lem: procedure stops}
    For any finite interval $\Lambda\subset\Z$ containing the origin and any $\sigma\in \Omega_{ \Lambda}^+$, the balancing procedure starting from $\sigma$ will end in finite steps. 
    In addition, if  $\sigma_0=-1$, then $\I_\sigma \neq \emptyset$ and $ A_\sigma\subset\Lambda$. In this case, the final configuration $\sigma^S$ is balanced in $\rho_{\frac{3}{2}}(\I_{\sigma})$. 
\end{lemma}
\begin{proof}
We first prove that the balancing procedure will end in finite steps.
By Lemma~\ref{Lemma: aux_1}, any interval $\B\in \mathcal{I}$ will be selected in the procedure at most once. Thus, it suffices to show that the number of intervals that may be selected is finite. We first claim that spins outside $\Lambda$ are never flipped to minus. Assuming otherwise, we can take $\B$ as the first minus isolated interval selected in the balancing procedure with $\B\cap \Lambda^c\neq \emptyset$. Supposing that $\B$ was selected in step $s_\B$, since $\B$ is minus isolated with respect to $\sigma^{s_\B}$, we get that the nearest neighbor of $\B$ outside $\Lambda$ must take value $-1$ in $\sigma^{s_\B}$. This contradicts our assumption that $\B$ was the first minus isolated interval intersecting $\Lambda^c$ selected in the procedure.

From the above claim,  it follows in particular that 
\begin{equation}\label{eq: no flip outside}
    \sigma^S\vert_{\Lambda^c}=1.
\end{equation} We also get that any interval $\B$ with $\B\cap \Lambda^c\neq \emptyset$ selected  in the balancing procedure can only be flipped to plus. By the definition of plus isolated intervals, $\B$ must contain at least one minus spin and thus $\B\cap \Lambda \neq \emptyset$. As a result,
we get that no interval $\B$ with $\B\cap \Lambda = \emptyset$ can be selected during the balancing procedure. Moreover, the number of intervals $\B\subset \Lambda$ is clearly finite, so it suffices to show that the number of intervals $\B$ selected in the balancing procedure satisfying $\B\cap \Lambda\neq\emptyset$ and $\B\cap \Lambda^c\neq\emptyset$ is finite. 

To this end, we note that every such interval $\B$ must be plus favored, as argued in the first paragraph of this proof.
Recall that $\B\in \mathcal{I}$, so if $\B$ has length $2^n$, we must have $\min_{y\in \B}|y|\ge 2^{n-4}$. This is because each $n$-interval in $\mathcal{I}$ can be written as $[2^{n-4}x,2^{n-4}(x+16))$ for some $x\in \Z$, and in addition we have $x>0$ or $x\leq -16$ since $0\not\in \B$. Combining with the fact $\B\cap\Lambda\neq \emptyset$, we get that the length of $\B$ is smaller than $32 |\Lambda|$ and therefore the number of intervals in this case is also finite. 

Next we prove that $\I_\sigma\neq \emptyset$ and $A_\sigma\subset \Lambda$ whenever $\sigma_0=-1$.  By the definition of the balancing procedure, a plus favored interval containing $0$ has never been selected through the balancing procedure. Thus, if $\sigma_0=-1$, the spin at the origin has never been flipped to $+1$ and hence we have $\sigma^S_0=-1$. From this, it follows that $0\in \B^-(\sigma)$ for every plus favored interval $\B\in \mathcal{I}$ containing the origin. From \eqref{eq: no flip outside}, we get that there exists at least one plus favored interval containing the origin (any interval $\B\in \mathcal{I}$ with $\Lambda\subset \B$ is plus isolated), and thus $\I_\sigma$ exists and from the previous observations we get that $0\in \I_\sigma^-(\sigma^S)$. Moreover, $A_\sigma = \I_\sigma^-(\sigma^S)\subset \Lambda$, since by \eqref{eq: no flip outside} all spins outside must be plus.

Finally, we prove that $\sigma^S$ is balanced in $\rho_{\frac{3}{2}}(\I_{\sigma})$. Otherwise there would exist an $\ell\geq 0$ and an $\ell$-interval  $\B\subset \rho_{\frac{3}{2}}(\I_\sigma)$ isolated with respect to $\sigma^S$. By Definition \ref{def: isolated interval}, we get that $\rho_{M_\ell}(\B)\subset \rho_{\frac{3}{2}}(\I_\sigma)$ and thus $\B$ is smaller than $\I_\sigma$.
Moreover, since $\sigma^S$ is the final configuration of the balancing procedure, $\B$ must be plus favored and contain the origin. This, however, contradicts (the `smallest' in) the definition of $\I_\sigma$. 
\end{proof}

One desirable feature of our balancing procedure, as shown in Lemma~\ref{Lemma: not_much_flips} below, is that the final configuration inherits structural properties of the original configuration: under some suitable conditions, dense regions of pluses/minuses in the original configuration remain to be weakly dense regions of pluses/minuses in the final configuration. We introduce some auxiliary definitions and lemmas before proving Lemma~\ref{Lemma: not_much_flips}. 

\begin{definition}\label{def: tame}
    An interval $\I$ is \textit{tame} with respect to $\sigma$ up to step $T$ if  $|\I\cap \B_{t}^c|\geq |\I|/16$ for all steps $t<T$, where $\B_{t}$ is the interval selected at step $t$ in the balancing procedure of $\sigma$. 
\end{definition}

\begin{lemma}\label{lem: plus density near a favored interval}
    Suppose that $\B=[b,b+2^\ell)$ is a plus favored interval with respect to $\sigma$. Then for any integer $0<p\le M_\ell2^\ell$, we have \begin{align}
        |\{x\in [b+2^\ell,b+2^\ell+p):\sigma_x=1\}|&\ge (1-\frac{2}{M_\ell})p,\label{eq: plus density near a favored interval 1}\\
        |\{x\in [b-p,b):\sigma_x=1\}|&\ge (1-\frac{2}{M_\ell})p.\nonumber\label{eq: plus density near a favored interval 2}
    \end{align} 
\end{lemma}
\begin{proof}
    By symmetry, it suffices to prove \eqref{eq: plus density near a favored interval 1}. We consider three cases according to the value of $p$.\begin{enumerate}
        \item If $p\le 2^{\ell-1}$, then all spins in $[b+2^\ell,b+2^\ell+p)$ must be plus, and thus \eqref{eq: plus density near a favored interval 1} holds immediately.
        \item If $2^{\ell-1}<p\le 2^{\ell}$, as the number of minuses in  $[b+2^\ell,b+2^\ell+p)$ is at most $\frac{2^{\ell}}{M_\ell}$, we get  \eqref{eq: plus density near a favored interval 1} immediately since $\frac{2^{\ell}}{M_\ell}\le \frac{2p}{M_\ell}$.
        \item If $p>2^\ell$, let $k=\lfloor\frac{p}{2^\ell}\rfloor$, so $p\ge k\cdot 2^\ell$. Since $p\le M_\ell2^\ell$, the $\ell$-interval $[b+i\cdot2^\ell,b+(i+1)\cdot2^\ell)$ is minus vacant for any integer $0<i\leq k$, implying that the number of minuses in $[b+2^\ell,b+2^\ell+p)$ is at most $\frac{(k+1)2^{\ell}}{M_\ell}$. Again, the trivial bound  $\frac{(k+1)2^{\ell}}{M_\ell}\le \frac{2k\cdot 2^\ell}{M_\ell}\le \frac{2p}{M_\ell}$ yields the desired \eqref{eq: plus density near a favored interval 1}.\qedhere
    \end{enumerate}
\end{proof}

\begin{lemma}\label{Lemma: not_much_flips}
   Let $\I$ be an interval with length $2^{n}\leq |\I| < 2^{n+1}$. Let $\sigma\in\Omega_{\Lambda}^+$ and let $T\geq 1$ be a step such that $\I$ is tame with respect to $\sigma$ up to step $T$. If $|\I^+(\sigma)|\leq \frac{|\I|}{\sqrt{M_n}}$, then \begin{equation}\label{eq: not_much_flip}
       |\cup_{t=0}^T\I^+(\sigma^t)|\leq c_1|\I^+(\sigma)|.
   \end{equation}
\end{lemma}

\begin{proof}
    Our proof of Lemma~\ref{Lemma: not_much_flips} is by induction on $n$.
    If $n=1$ then $|\I|=1$. Thus, no interval intersecting $\I$ is allowed to be selected, implying the claim immediately. Now we assume the lemma holds for any $n\leq N-1$, any starting configuration $\sigma$ and any $T\geq 1$. 
    We next prove by induction the case when $n = N$. To this end, let $T^\prime+1$ be the first step before $T$ such that an $\ell$-interval $\B$ with $M_\ell2^\ell> 2^n$ and $\B\cap \I\neq \emptyset$ was selected and flipped to plus (if no such interval exists before step $T$, then we take $T^\prime=T$).
    We claim that it suffices to prove that 
    \begin{equation}\label{eq: not much flips before T'}
        |\cup_{t=0}^{T^\prime}\I^+(\sigma^t)|\leq c_1|\I^+(\sigma)|,
    \end{equation} 
    i.e., the desired bound holds at step ${T^\prime}$. To this end, we next prove that in fact, we must have $T^\prime = T$ assuming \eqref{eq: not much flips before T'} holds. Otherwise, if $T^\prime < T$, as $\B$ was selected at step $T^\prime+1$ and $\I$ was assumed to be tame up to step $T$ (recall Definition \ref{def: tame}), we must have $|\I\cap \B^c|\geq \frac{|\I|}{16}$. Since $2^n<M_\ell 2^\ell$ and $|\I|<2^{n+1}$, we have $\frac{|\I\setminus \B|}{2}\le M_\ell 2^\ell$. Thus, by Lemma \ref{lem: plus density near a favored interval} the number of pluses of $\sigma^{T^\prime}$ in $\I\setminus \B$ is at least $(1 - \frac{2}{M_\ell})\frac{|\I\setminus \B|}{2}$. 
    Hence we have $|\I^+(\sigma^{T^\prime})|\ge (\frac{1}{2}-\frac{1}{M_\ell})|\I\setminus \B|\ge \frac{|\I|}{64}$, which contradicts \eqref{eq: not much flips before T'}. This then proves that $T=T^\prime$ and the desired result follows. The rest of the proof is dedicated to proving \eqref{eq: not much flips before T'}. 
    
    From now on, we fix a starting configuration $\sigma$. In order to keep track of the change during the balancing procedure, for any $x\in\I$ we record the first selected interval $D^x \ni x$ which results in flipping $x$ to plus ($x$ may or may not be flipped again, but that is not relevant to us). If $x$ was never flipped to plus we just take $\mathrm{D}^x = \emptyset$. Let $\mathcal{F}_\ell \coloneqq \{\mathrm{D}^x \subset \Z: |\mathrm{D}^x|=2^{\ell}, x\in\I\}$ be the collection of $\ell$-intervals  we have recorded. For any $\B\in \mathcal{F}_\ell$, let $s_\B$ be the step in which $\B$ was flipped and let $P_\ell(\B)=|\cap_{t=0}^{s_\B}\B^-(\sigma^{t})|$ be the number of new pluses created at the flip.  We say $\B\in\mathcal{F}_\ell$ is \textit{negligible} when $P_\ell(\B)\leq \frac{2^\ell}{M_\ell}$, i.e., not many new pluses were created when flipping minuses in $\B$; otherwise we say $\B$ is \textit{non-negligible}. We further split $\mathcal{F}_\ell$ into the union of $\mathcal{F}_\ell^j\coloneqq \mathcal{F}_\ell\cap \{[2^{\ell}(x+\frac{j}{16}), 2^{\ell}(x+\frac{j}{16}+1))\}_{x\in\Z}$ for $j=0,\dots, 15$. Let $\ell_n\coloneqq \lfloor\frac{n-\log_2(M_0)}{1+\delta}\rfloor$ be the maximal integer $\ell$ such that $M_\ell2^\ell\leq 2^n$. 
\begin{claim}\label{lem: negligibe_contibuiton}
    For negligible intervals and $0\le j\le 15$, if $\ell\le\ell_n$
    we have
    \begin{equation}\label{Eq: negligibe_contibuiton}
         \sum_{\substack{\B\in \mathcal{F}^j_\ell \\ \B \text{ negligible}}}P_\ell(\B)\leq \frac{8|\I^+(\sigma)|}{\sqrt{M_\ell}}. \\
    \end{equation}
\end{claim}
\begin{claim}\label{lem: non-negligibe_contibuiton}
    For non-negligible intervals and $0\le j\le 15$, if $\ell\le\ell_n$ 
    we have
    \begin{equation}\label{Eq: non-negligible_contribuition}
       \sum_{\substack{\B\in \mathcal{F}^j_\ell \\ \B \text{ non-negligible}}}P_{\ell}(\B)\leq\frac{2|\I^+(\sigma)|}{M_\ell^{1/3}}.
    \end{equation}
\end{claim}
For smooth presentation, we will first complete the proof of the lemma (i.e., proving \eqref{eq: not much flips before T'}) assuming Claims \ref{lem: negligibe_contibuiton} and \ref{lem: non-negligibe_contibuiton}, and provide their proofs afterwards. We avoid stating these claims as independent lemmas since we are under the induction hypothesis of the present lemma. 

    Note that $\cup_{t=0}^{T^\prime}\I^+(\sigma^t) \subset \I^+(\sigma) \cup\left( \cup_{t=1}^{T^\prime}\I^+(\sigma^t)\cap \I^-(\sigma)\right)$. 
    Recall the definition of $\mathcal{F}_\ell$ and $P_\ell(\B)$. Every $x \in \cup_{t=0}^{T^\prime}\I^+(\sigma^t)\cap \I^-(\sigma)$ must be in an interval that was selected in the balancing procedure and flipped to plus (i.e., this selected interval was a plus isolated interval). 
    Then we have
    \begin{equation*}
         \left|\cup_{t=0}^{T^\prime}\I^+(\sigma^t)\cap \I^-(\sigma)\right|\leq \sum_{\ell= 0}^{\ell_n} \sum_{\B\in \mathcal{F}_\ell}P_\ell(\B).
    \end{equation*} 
    Combined with \eqref{Eq: negligibe_contibuiton} and \eqref{Eq: non-negligible_contribuition}, as the upper bounds do not depend on $j$, it yields
   \begin{align}
       |\cup_{t=1}^{T^\prime}\I^+(\sigma^t) \cap \I^-(\sigma)|&\leq \sum_{\ell = 0}^{\ell_n} \sum_{\B\in \mathcal{F}_\ell}P_\ell(B) \nonumber \\  &\leq \sum_{\ell = 0}^{\ell_n}16{\times\left(\frac{2|\I^+(\sigma)|}{M_\ell^{1/3}}+\frac{8|\I^+(\sigma)|}{\sqrt{M_\ell}}\right)}
       \leq (c_1-1)|\I^+(\sigma)|,\nonumber\label{Eq: control_flips}
      \end{align}
      where the last inequality comes from choosing $M_0>0$ large enough. This completes the proof of \eqref{eq: not much flips before T'}.
\end{proof}

\begin{proof}[\textbf{Proof of Claim \ref{lem: negligibe_contibuiton}}] Recall that for any negligible interval $\B\in\mathcal{F}_\ell^{ j}$, $s_\B$ is the step that $\B$ was flipped. Let $\mathrm{C}= \rho_{\frac{3}{2}}(\B)\cap\I$. We will prove that $\mathrm{C}$ is tame with respect to $\sigma$ up to step $s_{\B}$ (and then we may apply our induction hypothesis to $\mathrm{C}$, i.e., apply \eqref{eq: not_much_flip} to $\mathrm{C}$; note that in the context of induction we have $|\mathrm C| \leq 2^{\ell+1}< 2^{n})$.  If an interval $\I^\prime$ with $|(\I^\prime)^c\cap \mathrm{C}| \le |\mathrm{C}|/16$ was selected in the balancing procedure before step $s_{\B}$, the length of $\I^\prime$ must be at least $\frac{15}{16}\cdot |\mathrm{C}|$ and thus $\mathrm{C}\subset \rho_{\frac{3}{2}}(\I^\prime)$. By Lemma~\ref{Lemma: aux_1}, the configuration restricted to $\rho_{\frac{3}{2}}(\I^\prime)$ remains constant at all steps after $s_{\I^\prime}$.
In addition, since $\B$ is plus favored at step $s_\B$, 
$\rho_{\frac{3}{2}}(\I^\prime)$ must be all plus at step $s_\B$ and this contradicts the fact that at least one vertex in $\B\cap\I$ was flipped to plus at step $s_\B$. Therefore, only intervals $\I^\prime$ with $|(\I^\prime)^c\cap \mathrm{C}| > |\mathrm{C}|/16$ can be selected in the balancing procedure before step $s_\B$. This implies that $\mathrm{C}$ satisfies the claimed tame condition.

    Next, we want to control the number of negligible $\ell$-intervals in $\mathcal{F}_\ell^j$. For any negligible interval $\B\in\mathcal{F}_\ell^j$,  $\rho_{\frac{3}{2}}(\B)\setminus\B$ is the union of two intervals, and one of them must be contained in $\I$ since $\ell < n$.
    In addition, since $\B$ is plus favored at step $s_\B$ we have $\sigma^{s_\B}\vert_{\rho_{\frac{3}{2}}(\B)\setminus\B} = 1$, and thus $|\cup_{t=0}^{s_\B}\mathrm{C}^+(\sigma^t)|\ge 2^{\ell-1}$. Applying the induction hypothesis combined with the fact that $|\mathrm{C}|\le 2^{\ell+1} $ and $\frac{1}{c_1}\cdot 2^{\ell-1}>\frac{2^{\ell-1}}{\sqrt{M_{\ell+1}}}$, we get $|\mathrm{C}^+(\sigma)|\geq \frac{2^{\ell-1}}{\sqrt{M_{\ell+1}}}$.

    Note that for any non-neighboring intervals $\B,\B^\prime\in\mathcal{F}_\ell^j$ we have $ \rho_{\frac{3}{2}}(\B)\cap \rho_{\frac{3}{2}}(\B^\prime)=\emptyset$, and thus $\mathrm{C}\cap\mathrm{C}^\prime=\emptyset$, where $\mathrm{C}^\prime= \rho_{\frac{3}{2}}(\B^\prime)\cap \I$. This implies that, for any $x\in\Z$, there are at most two $\B\in\mathcal{F}_\ell^j$ with $x\in \rho_{\frac{3}{2}}(\B)~$. 
    As there are $|\I^+(\sigma)|$ pluses in $\I$ at the start of the balancing procedure, we conclude that the number of negligible intervals in $\mathcal{F}_\ell^j$ is at most $2|\I^+(\sigma)|\frac{\sqrt{M_{\ell+1}}}{2^{\ell-1}}$. Combining with the fact that $P_\ell(\B)\leq\frac{2^\ell}{M_\ell}$ for any negligible interval $\B\in\mathcal{F}_\ell^j$, we get that 
    \begin{equation*}
         \sum_{\substack{\B\in \mathcal{F}^j_\ell \\ \B_i \text{ negligible}}}P_\ell(\B)\leq 2|\I^+(\sigma)|\frac{\sqrt{M_{\ell+1}}}{2^{\ell-1}}\cdot \frac{2^\ell}{M_\ell}\leq  \frac{8|\I^+(\sigma)|}{\sqrt{M_\ell}},
    \end{equation*} concluding the proof of Claim~\ref{lem: negligibe_contibuiton}.
\end{proof}
\begin{proof}[\textbf{Proof of Claim \ref{lem: non-negligibe_contibuiton}}] We list the non-negligible intervals in $\mathcal{F}_\ell^j$ as $\{\B_1,\dots,\B_K\}$. For $1\leq i\leq K$, we write $s_i\coloneqq s_{\B_i}$.
    We want to find an interval $\mathrm{C}_i$ neighboring $\B_i$ such that $\mathrm{C}_i$ has a large density of pluses. The set $\{x\in\Z: 1\leq d(x, \B_i)\leq \frac{M_\ell-1}{2}2^\ell\}$ is the union of two intervals, and one of them (which we denote as $\mathrm{C}_i$) must be contained in $\I$, since $M_\ell2^\ell\leq 2^n$. {In addition, we have $|\mathrm{C}_i|<2^{n}.$} Similar to the proof of Claim~\ref{lem: negligibe_contibuiton}, we want to apply the induction hypothesis to get a lower bound on the number of pluses in $\mathrm{C}_i$ at the start of the balancing procedure, and then give an upper bound on the number of non-negligible $\ell$-intervals. 
    
    As in the proof of Claim~\ref{lem: negligibe_contibuiton}, we now show that $\mathrm{C}_i$ is tame with respect to $\sigma$ up to step $s_i$. Suppose this is not the case. So, there is an interval $\I^\prime$ with $|(\I^\prime)^c\cap \mathrm{C}_i| \le |\mathrm{C}_i|/16$ selected in the balancing procedure at some step $s_i^\prime<s_i$. Then by Lemma~\ref{Lemma: aux_1}, we get that $\rho_{3/2}(\I^\prime)$ is constant in $\sigma^{s_i^\prime +k}$ for any $k\ge 1$. Note that $|\I^\prime|\ge\frac{15|\mathrm{C}_i|}{16}$ and thus $\frac{|\I^\prime|}{2}\ge \frac{15|\mathrm{C}_i|}{32}>\frac{|\mathrm{C}_i|}{8}+|\B_i|$ (see Figure~\ref{Fig: I_prime} for an illustration), so we have $\B_i\subset\rho_{3/2}(\I^\prime)$.
    \begin{figure}
        \centering
        \input{Fig_I_prime}
        \caption{The red region denotes the interval $\B_i$; the black line denotes the interval $\mathrm{C}_i$, and is divided into $16$ parts. The black box delimits the interval $\I^\prime$ that covers more than a $\frac{15}{16}$ fraction of $\mathrm{C}_i$, and the gray box represents $\rho_{\frac{3}{2}}(\I^\prime)$. The dotted line splits $\I^\prime$ in half.}
        \label{Fig: I_prime}
    \end{figure}
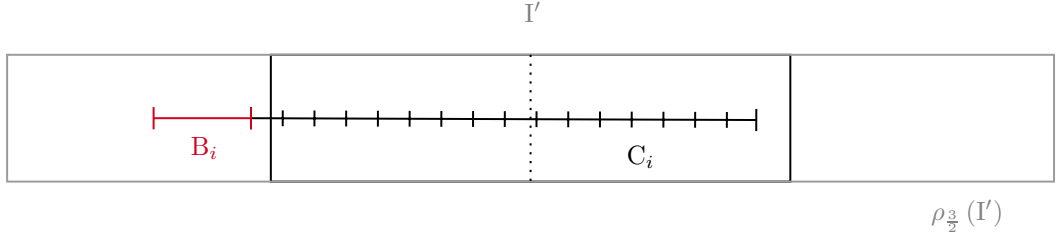
    Since $\B_i$ is plus isolated with respect to $\sigma^{s_i}$, this implies that $\sigma^{s_i}\vert_{\B_i}$ contains at least one minus spin. Thus  $\sigma^{s_i}\vert_{\rho_{\frac{3}{2}}(\I^\prime)}=-1$. However, this contradicts the fact that $\B_i$ is plus favored with respect to $\sigma^{s_i}$.
    This then verifies the desired tame condition for $\mathrm{C}_i$.

    Next, we want to control the number of non-negligible $\ell$-intervals in $\mathcal{F}_\ell^j$. Since $\B_i$ is plus favored at step $s_i$, we get that $|\mathrm{C}_i^+(\sigma^{s_i})|\geq \frac{M_\ell-1}{2}\left(1-\frac{1}{M_\ell}\right)2^\ell$. Take $k<n$ satisfying $2^k\leq \frac{M_\ell-1}{2}2^\ell \leq 2^{k+1}$. As in the proof of Claim~\ref{lem: negligibe_contibuiton}, we apply the induction hypothesis and combine it with the fact that $\frac{1}{c_1}\left( \frac{M_\ell-1}{2}\left(1-\frac{1}{M_\ell}\right)2^\ell\right)>\frac{1}{\sqrt{M_k}}\frac{M_\ell-1}{2}2^\ell$ to get
    \begin{equation}
        |\mathrm{C}_i^+(\sigma)|\geq \frac{1}{\sqrt{M_k}}\frac{M_\ell-1}{2}2^\ell\ge \frac{M_\ell^{1/3}2^{\ell}}{2},\label{eq: non-negligible initial plus control}
    \end{equation}
    where the last inequality holds since $M_0>0$ is big enough and $\delta<\frac{1}{20}$.

    To count the number of $\mathrm{C}_i$'s, we first show that they are disjoint by proving that for any $\B_i, \B_j$, we have 
    \begin{equation}\label{Eq: non-negligible_are_distant}
        d(\B_i, \B_j)> (M_\ell-1)2^\ell.
    \end{equation} 
    To this end, we can assume without loss of generality that $s_i<s_j$, i.e.,  $\B_i$ was flipped first. If $d(\B_i, \B_j)\leq (M_\ell-1)2^\ell$, since $\B_i$ is plus isolated with respect to $\sigma^{s_i}$, $\B_j$ must be minus vacant with respect to $\sigma^{s_i}$. Thus, $|\B_j^-(\sigma^{s_i})|{\le\frac{2^\ell}{M_\ell}}$, which contradicts to the fact that $\B_j$ is non-negligible.

    Note that the choice of $\mathrm{C}_i$ ensures that $\mathrm{C}_i\subset\I$.  Combining with the fact that $\mathrm{C}_i$'s do not intersect, we get from \eqref{eq: non-negligible initial plus control} that the maximum number of non-negligible intervals is at most $|\I^+(\sigma)|\frac{2}{M_\ell^{1/3}2^\ell}$, and therefore 
     \begin{equation*}
       \sum_{i=1}^K P_{\ell}(\B_i)\leq |\I^+(\sigma)|\frac{2}{M_\ell^{1/3}2^\ell}\cdot 2^\ell \leq \frac{2|\I^+(\sigma)|}{M_\ell^{1/3}}.
    \end{equation*}Thus we complete the proof of Claim~\ref{lem: non-negligibe_contibuiton}.
\end{proof}

Suppose at some step of the balancing procedure we selected a plus favored interval $\I$ and flipped all minus spins in $\I$. In the following steps of the balancing procedure, we may have created minuses close to $\I$, so that it is no longer plus favored with respect to the final configuration. We can however show that $\I$ will still be weakly plus favored with respect to the final configuration, as incorporated in the next proposition.

\begin{proposition}\label{Prop: balanced_intervals}
    Let $\sigma^T$ be the configuration obtained via the balancing procedure starting from $\sigma$ after $T$ steps. Given $x\in\Z$ with $\sigma^s_x\neq \sigma_x^{T}$ for some $0\le s<T$, let $\I$ be the last interval selected in the balancing procedure before step $T$ that flips $x$. Then $\I$ is weakly plus (resp. minus) favored with respect to $\sigma^{T}$ if $\sigma^T_x = 1$ (resp. $\sigma^T_x = -1$). 
\end{proposition}

\begin{proof}
    Let us assume that $\sigma^{s}_x=-\sigma^T_x=-1$. Let $s_\I$ be the step in the balancing procedure in which $\I$ was flipped. First notice that all spins in $\rho_{\frac{3}{2}}(\I)$ are plus in $\sigma^{s_\I+1}$. Since $\I$ is the last interval selected before step $T$ that flips $x$, we see that no interval containing $\I$ can be selected after step $s_\I$ and flipped to minus. Combined with Lemma \ref{Lemma: aux_1}, it implies that all spins of $\sigma^T$ in $\rho_{\frac{3}{2}}(\I)$ are still plus. 
    
    Now, let $\ell\geq 0$ and $x\in\Z$ be such that $\I=\I_\ell(x)$. For any $1\leq |k|\leq M_\ell/c_1$ and any interval $\B_{k}=\I_\ell(x+
 {16k})$, we have $|\B_{k}^-(\sigma^{s_\I})| \leq \frac{2^\ell}{M_\ell}$ (since $\I$ is plus favored with respect to $\sigma^{s_\I}$). Without loss of generality, we only consider $k\ge 0$. We want to apply induction on $k$ to show that \begin{equation}\label{eq: induction hypothesis in weak favor proof}
     \left|\cup_{t=s_\I}^T\B_k^-(\sigma^t)\right|\le \frac{c_1 2^\ell}{M_\ell}.
 \end{equation} The case $k=0$ holds since $\I$ remains all plus after step $s_\I$. Now we suppose \eqref{eq: induction hypothesis in weak favor proof} holds for $0\leq k\leq m-1$, and wish to prove \eqref{eq: induction hypothesis in weak favor proof} for $k = m$. To this end, we claim that there does not exist an interval $\mathrm{C}$ flipped to minus after step $s_\I$ such that $|\B_m\cap \mathrm{C}^c|\leq \frac{|\B_m|}{16}$.
 If such an interval $\mathrm{C}$ exists, 
 then $|\mathrm{C}|\ge 2^\ell$ and hence $|\rho_{\frac{3}{2}}(\mathrm{C})\cap \B_{m-1}|\ge \frac{2^\ell}{2}-\frac{2^\ell}{16}\ge 7\cdot 2^{\ell-4}$. By the definition of minus favored (recall that $\mathrm{C}$ is flipped to minus and recall the balancing procedure), we get that $\rho_{\frac{3}{2}}(\mathrm{C})$ is all minus in $\sigma^{s_\mathrm{C}+1}$. Thus $\B_{m-1}^-(\sigma^{s_\mathrm{C}+1})\ge 7\cdot 2^{\ell-4}$ which contradicts the induction hypothesis \eqref{eq: induction hypothesis in weak favor proof} for $k=m-1$, completing the verification of the claim.

We now prove \eqref{eq: induction hypothesis in weak favor proof} using the above claim and Lemma~\ref{Lemma: not_much_flips}. Let $T_m+1$ be the first step that an interval $\mathrm{C}$ with $|\B_m\cap \mathrm{C}^c|\leq \frac{|\B_m|}{16}$ was selected in the balancing procedure (if no such interval exists, then take $T_m=T$). Then $\B_m$ is tame with respect to $\sigma^{s_\I}$ up to step $T_m$. Thus, by Lemma~\ref{Lemma: not_much_flips},  \begin{equation}\label{eq: eq: induction hypothesis in weak favor proof 1}
     \left|\cup_{t=s_\I}^{T_m}\B_m^-(\sigma^t)\right|\le \frac{c_1 2^\ell}{M_\ell}.
 \end{equation}
Furthermore, letting $\mathrm{C}$ be the interval selected at step $T_m+1$, we see from our aforementioned claim that $\mathrm{C}$ is flipped to plus. Thus, we get by Lemma~\ref{Lemma: aux_1} that $\sigma^{T_m+t}\vert_{\rho_{\frac{3}{2}}(\mathrm{C})}$ is constant for all $t\ge 1$. Moreover, since no interval $\mathrm{D}$ with $|\B_m\cap \mathrm{D}^c|\leq \frac{|\B_m|}{16}$ (in particular, no interval containing $\mathrm{C}$) was selected and flipped to $-1$, we must have $\sigma^{T_m+t}\vert_{\rho_{\frac{3}{2}}(\mathrm{C})} = +1$ for all $t\geq 1$.   Combined with the fact that $\B_{m}\subset\rho_{\frac{3}{2}}(\mathrm{C})$ and \eqref{eq: eq: induction hypothesis in weak favor proof 1}, this
completes the proof of \eqref{eq: induction hypothesis in weak favor proof} for $k=m$. 
\end{proof}
\subsubsection{Proof of Theorem~\ref{thm-main}}\label{sec: proof of main thm}

In this subsection, we prove Theorem \ref{thm-main} assuming Propositions~\ref{Prop: Energy_bound_1}, \ref{Prop: Energy_bound_2} and Propositions \ref{prop: free energy control}, \ref{prop: entropy argument}, whose proofs are postponed until Sections \ref{sec: energy bound} and \ref{sec: entropy bound} respectively. Fix \begin{equation}\label{eq: def of theta}
    \theta \coloneqq \min\{2 - \alpha -10\delta, \log_{3.9}(2)\}> \frac{1}{2}.
\end{equation}
Moreover, for all $A,B \subset \Z$ we define $J(A,B) = \sum_{x\in A, y\in B}J_{xy}$.

\begin{proposition}\label{Prop: Energy_bound_1}
    Recall Definition \ref{def: Peierls map}. There is a constant $c_2>0$ depending only on $\alpha$ such that for any configuration $\sigma\in\Omega^+_{\Lambda}$, 
    \begin{equation*}
       J(A_\sigma, A_\sigma^c) \geq c_2 |\I_{\sigma}|^\theta.
    \end{equation*}
\end{proposition}

\begin{proposition}\label{Prop: Energy_bound_2}
     For any configuration $\sigma\in\Omega^+_\Lambda$, 
     \begin{equation}\label{eq: hamiltonian bound}
     H^+_{\Lambda;0}(\sigma) - H^+_{\Lambda;0}(\tau_{A_\sigma}(\sigma)) \geq J(A_\sigma, A^c_\sigma).
     \end{equation}
\end{proposition}

Proposition \ref{Prop: Energy_bound_2} shows that flipping spins in $A_\sigma$ will gain energy in the interaction term. In order to control the term from the external field in the Hamiltonian, we want to flip the external field on $A_\sigma$ as in \cite{Ding2021}. Similar to $\sigma$, for any set $A\subset \Z$ we define $\tau_A(h) \coloneqq ( (-1)^{\mathbbm{1}_{A}(x)}h_x)_{x\in\Z}$. We define the error function 
\begin{equation}\label{Eq: Delta}
\Delta_A(h) \coloneqq -\frac{1}{\beta}\log{\frac{Z_{\Lambda; \beta, \varepsilon}^{+}(h)}{Z_{\Lambda; \beta, \varepsilon}^{+}(\tau_{A}(h))}},
\end{equation}
where $Z_{\Lambda; \beta, \varepsilon}^{\eta}(h)$, for $\eta \in \{-1, 1\}^\Z$,  is the \textit{partition function} (note that $H^\eta_{\Lambda, \epsilon h}(\cdot)$ can be naturally regarded as a function of $\sigma \in \{-1, 1\}^\Lambda$)
\begin{equation*}
    Z_{\Lambda; \beta, \varepsilon}^{\eta}(h)\coloneqq \sum_{\sigma\in\{-1, 1\}^\Lambda} e^{-\beta H_{\Lambda; \varepsilon h}^{\eta}(\sigma)}.
\end{equation*}

The key property required for the error functions as in \eqref{Eq: Delta} is their subgaussian behavior, proved in \cite{Ding2021}. \begin{lemma}\cite[Lemma 3.1]{Ding2021}\label{Lemma: Concentration.for.Delta.General}
    For any $A, A^\prime \Subset \mathbb{Z}$ and $\lambda>0$, we have 
\begin{equation}\label{Eq: Tail.of.the.diff.of.Deltas}
     \mathbb{P}(|\Delta_{A}(h) - \Delta_{A^\prime}(h)|>\lambda \mid h_{(A \cup A^\prime)^c}) \leq  2e^{-\frac{{\lambda^2}}{{8\varepsilon^2|A \Delta A^\prime|}}},
\end{equation}
where $A\Delta A^\prime$ is the symmetric difference and $h_{(A \cup A^\prime)^c}$ denotes the external field on ${(A \cup A^\prime)^c}$.
\end{lemma}

\begin{definition}\label{Def: bad event}
    Let $\mathcal{A}(\I)$ denote the collection of subsets $A\subset \I$ such that there exists a balanced configuration $\sigma$ on $\I$ satisfying $A=\I^-(\sigma)$. In addition, for any $Q>0$ we define $\mathcal{A}_Q(\I)$ to be the collection of subsets $A\in \mathcal{A}(\I)$ such that $J(A,A^c)\in [Q,2Q)$. 
\end{definition}
\begin{definition}
Suppose that $\I_n$ is an $n$-interval containing the origin. Let $\hat{\cA}(\I_n)$ be the collection of subsets $A\in \mathcal{A}(\I_n)$ such that $J(A,A^c)\ge c_2\cdot 2^{\theta n}$.
    We define the good event to be
\begin{equation}\label{bad_event_1d}
    \mathcal{E}(\I_n)\coloneqq \left\{\Delta_{A}(h) \leq \frac{J(A,A^c)}{10}, \text{ for all }A\in \hat{\cA}(\I_n)\right\}.
\end{equation}
\end{definition}
The next proposition implies that the probability of the bad event  $\mathcal{E}(\I_n)^c$ is small.
\begin{proposition}\label{prop: free energy control}    
    Suppose that $\I_n$ is an $n$-interval containing the origin. Then there exists $c_3\coloneqq c_3(\alpha)$ such that $$\mathbb{P}(\mathcal{E}(\I_n)^c)\leq c_3^{-1}\exp(-c_3\frac{2^{(2\theta-1-2\delta)n}}{\varepsilon^2}).$$ 
\end{proposition}

Now, it suffices to control the entropy, as incorporated in Proposition~\ref{prop: entropy argument} below.
Recall that we have a partition of $\mathbb Z$ into $\ell$-intervals as $\mathcal{I}_\ell^{0}=\big\{ [x\cdot 2^\ell,(x+1)\cdot 2^\ell)\big\}_{x\in \mathbb Z}$. We start with some definitions.

\begin{definition}\label{Def: Definition of Psi}
Suppose that $\I_n$ is an $n$-interval.    For any $A\in \mathcal{A}(\I_n)$ and $0\le  \ell\le n$, we define $\Psi_{\ell}(A):\mathcal{I}_\ell^{0}\rightarrow\{-1,0,1\}$ as:   \begin{equation}\label{eq: defi of psi}
    \Psi_{\ell}(A,\I_\ell(x))=\left\{\begin{aligned}
        -1~~~~&\text{if } \I_\ell(x)\subset A;\\
        1~~~~&\text{if } \I_\ell(x)\cap A=\emptyset;\\
        0~~~~&\text{otherwise}.
    \end{aligned}\right.
    \end{equation}
\end{definition}
See Figure~\ref{Fig: Psi} for an illustration of $\Psi(A)$. 
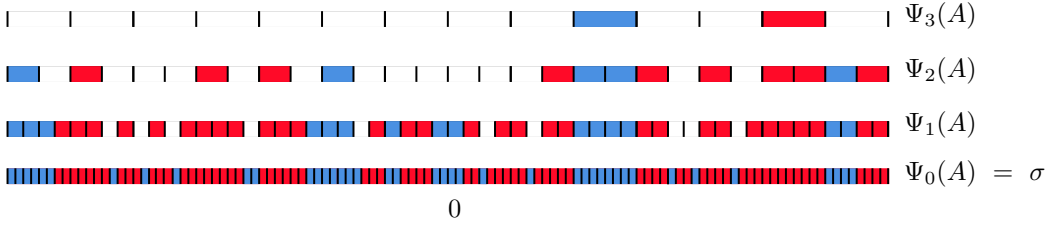
\begin{figure}[ht]
    \centering
    \input{Fig_Psi}
    \caption{Considering a configuration $\sigma$ in an interval $\I$, and taking $A=\I^-(\sigma)$, the picture depicts, from bottom to top, $\Psi_0(A), \Psi_1(A), \Psi_2(A)$ and $\Psi_3(A)$. For every $\ell=0,1,2,3$, intervals $\I_\ell$ are painted \ter{red} if $\Psi_\ell(A, \I_\ell) = 1$, painted \teb{blue} if $\Psi_\ell(A, \I_\ell) = -1$ and painted white otherwise. We write $\sigma = \Psi_0(A)$ since both functions attribute the same value to each site.}
    \label{Fig: Psi}
\end{figure}
\begin{proposition}\label{prop: entropy argument}
Suppose that $\I_n$ is an $n$-interval containing the origin.
   Let $E(\ell,Q,\I_n) \coloneqq |\{\Psi_\ell(A): A\in \mathcal{A}_Q(\I_n)\}|$ denote the number of images  in $\Psi_\ell\left(\mathcal{A}_Q(\I_n)\right)$ (here $\Psi_\ell(A)$ is treated as a map), and recall our definition of $\theta$ from \eqref{eq: def of theta}. Then there exists a constant $c_4>0$ such that the following inequality holds for any  $M\ge 1$ and $\ell\le n-4$:
    \begin{equation}\label{eq: entropy bound from l+1 to l}
        E(\ell,Q,\I_n)\le E(\ell+1,Q,\I_n)\times \exp\big(\frac{c_4Q}{2^{\ell \theta}}\big).
    \end{equation}
\end{proposition}
\begin{corollary}\label{cor: entropy argument whole level}
Suppose that $\I_n$ is an $n$-interval containing the origin. There exists a constant $c_5>0$ such that the following inequality holds for any  $M\ge 1$ and $\ell\le n-3$:
    \begin{equation}\label{eq: entropy bound}
        E(\ell,Q,\I_n)\le  c_5\exp\big(\frac{c_5Q}{2^{\ell \theta}}\big).
    \end{equation}
\end{corollary}
\begin{proof}
    Since $\I_n$ is an $n$-interval containing the origin, we get that $A\subset [-2^n,2^n)$ for any $A \in \mathcal{A}(\I_n)$. Recall Definition~\ref{Def: Definition of Psi}. Thus we get that $E(n-3,Q,\I_n)\le 3^{16}.$ Hence applying \eqref{eq: entropy bound from l+1 to l} iteratively, we get that for $\ell\le n-4$ 
    \begin{equation*}
        E(\ell,Q,\I_n)\le3^{16}\cdot \prod_{k=\ell}^{n-4}\exp\big(\frac{c_4Q}{2^{k \theta}}\big)\le  c_5\exp\big(\frac{c_5Q}{2^{\ell \theta}}\big)
    \end{equation*} with some $c_5$ depending on $c_4$ and $\theta$.
\end{proof}
Now we are ready to prove Theorem~\ref{thm-main}.
\begin{proof}[\textbf{Proof of Theorem~\ref{thm-main}}]
Let $\mathcal{J}_n$ denote the collection of $n$-intervals $\I$ such that $0\in \I$, and take $\mathcal{H}_n \coloneqq \cap_{\I\in\mathcal{J}_n}\mathcal{E}(\I)$. Letting $\mathcal{H}=\cap_{n=0}^{\infty}\mathcal{H}_n$, we derive from Proposition \ref{prop: free energy control} and a simple union bound that (note that $|\mathcal J_n| \leq 16$)
\begin{equation}\label{eq: external field control for all levels}
    \P(\mathcal{H})\ge 1-16\sum_{n=0}^{\infty}c_3^{-1}\exp(-c_3\frac{2^{(2\theta-1-2\delta)n}}{\varepsilon^2})\ge 1-C_1\exp(\frac{1}{C_1\vareps^2}).
\end{equation}
Recall Definition \ref{def: Peierls map}. For any integer $n\ge 0$ and $Q>0$, let $\Sigma_{Q,n}$ denote the collection of configurations such that \begin{enumerate}
    \item $J(A_\sigma,A_\sigma^c)\in [Q,2Q)$.
    \item $\I_\sigma$ is an $n$-interval.
\end{enumerate}
By \eqref{eq: hamiltonian bound}, \eqref{Eq: Delta} and \eqref{bad_event_1d} (as well as the definition of $\mathcal H$), we get that for any $\sigma\in \Sigma_{Q,n}$ and $h\in \mathcal{H}$ \begin{equation}\label{eq: probability bound after flipping}
    \mu_{\Lambda;\beta,\vareps h}^+(\sigma)\le \mu_{\Lambda;\beta,\vareps \tau_{A_\sigma}(h)}^+(\tau_{A_\sigma}(\sigma))\times \exp(-\beta\cdot 0.9Q).
\end{equation}

By Definition~\ref{Def: Definition of Psi}, $\Psi_0(A,\I_0(x)) = 1-2\cdot \mathbbm{1}_{A}(x)$, so $A \rightarrow \Psi_0(A)$ is clearly a bijection. In addition, since $A_\sigma=\I_\sigma^-(\sigma^S)$ and $\sigma^S$ is balanced in $\I_\sigma$ (by Lemma \ref{lem: procedure stops}), we get that $A_\sigma\in \mathcal{A}(\I_\sigma)$. 
Hence the number of choices for $A_{\sigma}$ with $\sigma\in \Sigma_{Q,n}$ is at most $\sum_{\I\in\mathcal{J}_n}E(0,Q,\I)$.
By Corollary \ref{cor: entropy argument whole level}, we have  $E(0,Q,\I)\leq c_5\exp(c_{5}Q)$ for any $\I\in \mathcal{J}_n$. Combining with the fact that $|\mathcal{J}_n|\le 16$, we get that the multiplicity of the map $(\sigma,h) \rightarrow \big(\tau_{A_\sigma}(\sigma),\tau_{A_\sigma}(h)\big)$ is at most  $16c_5\exp(c_{5}Q)$ to 1. Hence, summing over $\sigma\in \Sigma_{Q,n}$ in \eqref{eq: probability bound after flipping} we get that, for $h\in \mathcal{H}$, 
\begin{align}\label{eq: probability bound for level M,n}
   \sum_{\sigma\in \Sigma_{Q,n}}\mu_{\Lambda;\beta,\vareps h}^+(\sigma) &\le 16c_5\exp(c_{5}Q)\times \exp(-\beta\cdot 0.9Q).
\end{align} 
By Proposition \ref{Prop: Energy_bound_1}, we have $\Sigma_{Q,n}=\emptyset$ whenever $Q<c_22^{n\theta}$. Letting $Q_{n,\ell}\coloneqq c_22^{n\theta+\ell}$, we have $\{\sigma\in\Omega_\Lambda^+ : |\I_{\sigma}|=2^n \}=\cup_{\ell=0}^{\infty}\Sigma_{Q_{n,\ell},n}$. Therefore, from \eqref{eq: probability bound for level M,n} we get that, for $h\in \mathcal{H}$,
\begin{align}
   \sum_{\substack{\sigma : |\I_{\sigma}|=2^n}}\mu_{\Lambda;\beta,\vareps h}^+(\sigma) &\le \sum_{\ell=0}^{\infty}16c_5\exp(-(0.9\beta-c_{5})Q_{n,\ell})\nonumber \\
        &\le C_1^{-1}\exp(-C_1(0.9\beta-c_{5})2^{n\theta}),\nonumber
\end{align}
where the second inequality holds under the assumption that $0.9\beta - c_5 > 0$. Recalling Lemma \ref{lem: procedure stops} and summing over $n$ we obtain
\begin{align}
\mu_{\Lambda;\beta,\vareps h}^+(\sigma_0=-1)&\le\sum_{n=0}^{\infty}C_1^{-1}\exp(-C_1(0.9\beta-c_{5})2^{n\theta})\nonumber \\
    &\le C_2^{-1}\exp(-C_2(0.9\beta-c_{5})).\nonumber
\end{align} 
Combining this with \eqref{eq: external field control for all levels}, we get the desired result by letting $\beta>0$ large enough.
\end{proof}
\begin{remark}\label{rmk: extension to pure 1d model}
    Notice that Propositions \ref{Prop: Energy_bound_1}, \ref{Prop: Energy_bound_2} and \ref{prop: entropy argument} hold for $1<\alpha<2$ with $\delta=\frac{2-\alpha}{20}$ and $\theta=\min\{2 - \alpha -10\delta, \log_{3.9}(2)\}> 0$. Thus the long-range order without disorder in this regime follows from a standard Peierls argument.
\end{remark}
\begin{remark}
    Notice that our argument does not rely on the state space being $\{-1, 1\}$, so our results can be extended to the $q$-Potts model with $q\ge 3$. One can define the notion of $p$-favored for $p=1,\cdots,q$, repeat the balancing procedure, and get the long-range order for the $q$-Potts model at a low enough temperature.
\end{remark}

\subsection{Energy bounds}\label{sec: energy bound}
This subsection is dedicated to proving the energy estimation, namely, Propositions \ref{Prop: Energy_bound_1} and \ref{Prop: Energy_bound_2}. We state some auxiliary lemmas before proving Proposition \ref{Prop: Energy_bound_1}.

\begin{lemma}\label{Lemma: interacton_1}
There exists a constant $\overline{c}_1\coloneqq \overline{c}_1(\alpha)$ such that, for any interval $\I\subset\Z$ and any configuration $\sigma\in\Omega$, if $\min\{|\I^-(\sigma)|, |\I^+(\sigma)|\} =  m$, then 
\begin{equation*}
    J(\I^-(\sigma),\I^+(\sigma))\geq \overline{c}_1m^{2-\alpha}.
\end{equation*}
\end{lemma}

\begin{proof}
    We can assume without loss of generality that $|\I^-(\sigma)|\leq |\I^+(\sigma)|$. It is enough to show that, for all $x\in \I^-(\sigma)$ 
\begin{equation}\label{eq: Int_x_with_plus}
    J(\{x\},\I^+(\sigma))\geq \overline{c}_1 {m}^{1-\alpha},
\end{equation}
since we can then sum both sides over $x\in\I^-(\sigma)$ to get the desired bound. As we assumed $m=|\I^-(\sigma)|\leq |\I^+(\sigma)|$,
it is obvious that $\{|x-y|: y\in \I^+(\sigma)\}$ is dominated by the set $\{m+1, \ldots, |\I|\}$ (here the domination is in the sense of the domination between the empirical measures of these two sets), implying that
\begin{equation*}
    J(\{x\},\I^+(\sigma))\geq \sum_{k={m}+1}^{|\I|}k^{-\alpha} \geq \sum_{k={m}+1}^{2m}k^{-\alpha} \geq \overline{c}_1 {m}^{1-\alpha}.\qedhere
\end{equation*}
\end{proof}
\begin{corollary}\label{Cor: F_A}
    For any $A\subset \Z$, 
    \begin{equation*}
        J(A, A^c)\geq \overline{c}_1 |A|^{2-\alpha}.
    \end{equation*}
\end{corollary}
\begin{proof}
    Let $\I$ be the smallest interval with $A\subset \I$ and $|\I|\ge 2|A|$. Applying Lemma~\ref{Lemma: interacton_1} for $\sigma$ with $\sigma_x = \mathbbm{1}_{A}(x) - \mathbbm{1}_{A^c}(x)$ yields the desired result (note that $|\I^\pm(\sigma)| \geq |A|$).
\end{proof}

When there are very few minuses or very few pluses in an interval, the lower bound on Lemma~\ref{Lemma: interacton_1} is not very useful. To get a better bound, we need to take advantage of the balanced property. To this end, it will be useful to approximate arbitrary intervals by $\ell$-intervals, as carried out in the next lemma.

\begin{lemma}\label{Lemma: approximate_interval}
    Let $\I$ be any interval. There exists $\ell\geq 0$ and an $\ell$-interval $\I_\ell= [a, b]$  such that $\I\subset \I_\ell$ and $\max{\{d(\I,a), d(\I,b)\}}\leq 0.7|\I|$ (see Figure~\ref{Fig: Lemma_3.3} for an illustration).
\end{lemma}
\begin{proof}
    Take $\ell\geq 0$ satisfying $\frac{15}{8}2^{\ell-2}\leq |\I| \leq \frac{15}{8}2^{\ell-1}$. As $|\I|\leq 2^{\ell}-\frac{2^{\ell}}{16}$, we can find an $\ell$-interval $\I_{\ell}(x)$ covering $\I$. Supposing $x$ is the smallest integer with $\I\subset \I_{\ell}(x)$, we can always find a translation $\I_{\ell}(x+n)=[a,b]$, with $0\leq n\leq 15$, such that $|d(a, \I) - d(b,\I)|\leq \frac{2^{\ell}}{16}$. Thus, for this interval $\I_{\ell}(x+n)$, we have $d(a, \I) \leq \frac{1}{2}\left({2^{\ell} + 2^{\ell-4} - \frac{15}{8}2^{\ell -2}}\right)\leq \frac{7}{10}|\I|$. A similar estimation holds for $d(b, \I)$, concluding the proof.
\end{proof}

\begin{figure}[ht]
    \centering
    \input{Fig_Lemma_3.3}
    \caption{The gray region delimits the interval $\I$. We take $\ell$ satisfying $\frac{15}{8}2^{\ell-2}\leq |\I| \leq \frac{15}{8}2^{\ell-1}$. The black line represents $\Z$ divided into intervals of length $2^{\ell}/16$. The $\ell$-interval $\I_{\ell}$ with endpoints in the red regions satisfies the desired inequality.}
    \label{Fig: Lemma_3.3}
\end{figure}
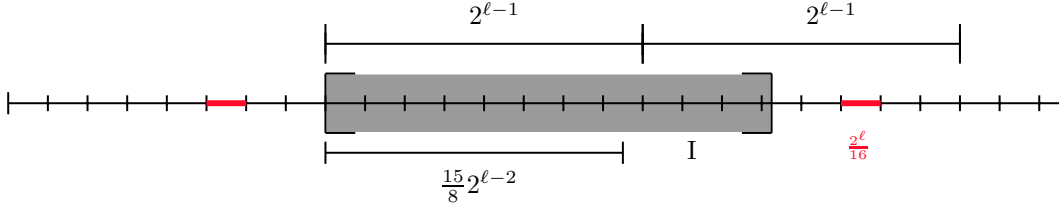

In Lemma \ref{Lemma: interacton_2}, we will show that there is a restriction on the distance between intervals covering the minus spins. To describe that precisely we use the next definition and lemma.

\begin{definition}\label{def: lambda good}
    Given $\lambda>0$, we say that a sequence $\mathscr{P} = (p_i)_{i=1}^N \subset \{0,1\}^N$ is $\lambda$-\textit{good}, if for any interval $\I \subsetneq [1, N]$ with $\max_{i\in \I} p_i = 1$, there exists  $x\in [1,N]\setminus \I$ with $p_x=1$ and $d(x, \I)\leq \lambda|\I|+1$. 
\end{definition}

\begin{lemma}\label{Lemma: Sequence_01}
    Let $N\geq 1$ and $\mathscr{P}=(p_i)_{i=1}^N$ be an $\lambda$-good sequence with $p_1=p_N=1$, Then, $$|\{1\leq i\leq N: p_i = 1\}|\geq N^{\log_{\lambda+2}(2)}.$$
\end{lemma}
\begin{proof}
    Our proof is by induction. If $N=1$ then $p_1=1$ and the desired result holds. Then we assume the lemma holds for any $N \leq M-1$ (for some $M \geq 2$) and we next prove it for $N = M$. Let $\I=[n,m]$ be the largest interval such that $\max_{i\in \I} p_i=0$. We now prove that $\mathscr{P}_{[1,n)}\coloneqq (p_i)_{i=1}^n$ is $\lambda$-good.  Take any interval $\B\subsetneq [1,n)$ with $\max_{i\in \B}p_i = 1$, and take $x\in [1,M]\setminus \B$ with $p_x=1$ such that $d(x, \B) = d(\B, \{y\in [1,M]\setminus \B : p_y=1\})$. When there are two possible choices for $x$, we take the one on the left. As $\mathscr{P}$ is $\lambda$-good, we already know that $d(x, \B)\leq \lambda |\B| +1$, by the definition of $\lambda$-good, so we just need to show that $x\in [1,n)\setminus \B$. Otherwise, we have $x=m+1$. Since $p_1=p_{n-1}=1$ ($p_{n-1}=1$ follows from the maximality of $\I$), by the definition of $x$ and the assumption that $x = m+1$, we can assume that $B=[b,n)$ for some $b>1$. Moreover, as $d(x, B) = |\I|+1$,  we must have $\max_{i\in [b - |\I| -2, b)} p_i = 0$ (by $p_1=1$), contradicting the fact that $|\I|$ is the largest interval containing only $0$'s. The same argument shows that $\mathscr{P}_{(m,M]} \coloneqq (p_i)_{i=m}^M$ is also $\lambda$-good, and therefore, by the induction hypothesis
    \begin{align}\label{Eq: bound_1s}
       |\{x\in [1,M]: p_x=1\}| &=  |\{x\in [1,n): p_x=1\}| + |\{x\in (m,M]: p_x=1\}|\nonumber \\ 
        &\geq (n-1)^{\log_{\lambda +2}(2)} + (M-m-1)^{\log_{\lambda +2}(2)}.
    \end{align}
    Denoting $N_1=n-1$ and $N_2=M-m-1$, we can assume without loss of generality that $N_1\leq N_2$. As $\mathscr{P}$ is $\lambda$-good and $[1,n)\subsetneq [1,N]$ satisfies $\max_{i\in [1,n)} p_i = 1$, we must have $|\I| = d([1,n), (m,M])-1\leq \lambda N_1$. Hence $M\leq N_1(1+\lambda) + N_2$ and $(2+\lambda)N_2\ge M$. Plugging this into \eqref{Eq: bound_1s} we get 
    \begin{align}
        |\{x\in [1,N]: p_x=1\}| \geq  \left(\frac{M-N_2}{\lambda+1}\right)^{\log_{\lambda +2}(2)} + N_2^{\log_{\lambda +2}(2)}&\geq 2\left(\frac{M}{\lambda + 2}\right)^{{\log_{\lambda + 2}(2)}}\nonumber = M^{{\log_{\lambda + 2}(2)}},\nonumber
    \end{align}
    where in the last inequality we used the fact that $\left(\frac{M-N_2}{\lambda+1}\right)^{\log_{\lambda +2}(2)} + N_2^{\log_{\lambda +2}(2)}$ increases in $N_2$ when $N_2\ge \frac{M}{2+\lambda}$.
\end{proof}

\begin{lemma}\label{Lemma: interacton_2}
    Recall our definition of $\theta$ in \eqref{eq: def of theta}. Let $\I$ be an interval, and let $\sigma\in\Omega$ be a balanced configuration in $\rho_{\frac{3}{2}}(\I)$. If $\sigma$ is not constant on $\I$, then there is a constant $\overline{c}_2\coloneqq \overline{c}_2(\alpha)$ such that
    \begin{equation*}
        \sum_{x,y\in\rho_{\frac{3}{2}}{(\I)}}\mathbbm{1}_{\{\sigma_x\neq\sigma_y\}}J_{xy}\geq \overline{c}_2|\I|^\theta.
    \end{equation*}
\end{lemma}
\begin{proof}
    To simplify the notation let us write $\overline{\I}\coloneqq \rho_{\frac{3}{2}}(\I)$. Taking $N$ such that $2^N\le |\I|<2^{N+1}$, we want to show that
    \begin{equation}\label{eq: induction hypothesis in interaction_2}
        J(\overline{\I}^-(\sigma),\overline{\I}^+(\sigma))\geq \overline{c}_22^{N\theta+\theta}.
    \end{equation}
    Our proof is again by induction on $N$. When $N\le 1000\log_2(M_0)+1$, \eqref{eq: induction hypothesis in interaction_2} holds trivially by taking $\overline{c}_2$ sufficiently small depending on $M_0$. Now, given $n>1000\log_2(M_0)+1$, we suppose that \eqref{eq: induction hypothesis in interaction_2} holds for all intervals $\I$ with $2^N\leq |\I| < 2^{N+1}$ for all $N<n$. We will show that for $\I$ with $2^n\le |\I|< 2^{n+1}$ the inequality \eqref{eq: induction hypothesis in interaction_2} also holds with $N=n$. We can assume that $|\I| = 2^n$, since otherwise we take $\I'$ the subinterval of $\I$ with length $2^n$ such that $\sigma$ is not constant on $\I'$ and prove \eqref{eq: induction hypothesis in interaction_2} for $\I'$. Furthermore, we  can assume without loss of generality that $|\overline{\I}^-(\sigma)|\leq |\overline{\I}^+(\sigma)|$.
    
     Considering the scale $m= \lfloor n(1-C_1\delta) - C_1\log_2(M_0)\rfloor$ with $C_1=500$, we have $2^{n-m}\ge M_n^{C_1}$. By the assumption $n>1000\log_2(M_0)+1$, we get that $m>0$. For scales $\ell\geq m$, if there is a minus occupied $\ell$-interval $\I_\ell\subset \I$, as we are assuming $|\overline{\I}^-(\sigma)|\leq |\overline{\I}^+(\sigma)|$, we can apply Lemma \ref{Lemma: interacton_1} to $\I$ and get \eqref{eq: induction hypothesis in interaction_2}. 
     
     We can thus assume all $\ell$-intervals $\I_\ell\subset \I$ with $\ell\geq m$ are minus vacant. Under this assumption, if $\B\subset \I$ is an $\ell$-interval containing a minus with $\ell\geq m$ and $\rho_{M_\ell}(\I_\ell)\subset \I$, then there exists $x\in \overline{\I}^-(\sigma)\cap \B^c$ such that $d(x,\B)\leq 2^{\ell - 1}$; this is because $\B$ cannot be plus favored due to the balanced assumption for $\I$.
    
    Let $\{\B_1,\dots, \B_k\}\subset \mathcal{I}_m^0$ be a covering of $\overline{\I}^-(\sigma)$ by disjoint $m$-intervals, such that each $\B_i$ (for $1\leq i\leq k$) contains at least one minus spin. We want to use Lemma~\ref{Lemma: Sequence_01} to give a lower bound on $k$. Let $\mathcal{B} = \cup_{i=1}^k\B_i$ and let $\I_{\mathcal{B}}$ be the smallest interval that contains $\mathcal{B}$. We claim that $|\I_{\mathcal{B}}|\geq \frac{2^n}{5M_n}$.
    Otherwise, by Lemma \ref{Lemma: approximate_interval}, there exists an interval $\mathrm{D}=[a, b)\in\mathcal{I}$ containing $\I_{\mathcal{B}}$ such that  $|D| \leq d(\I_{\mathcal{B}},a)+ d(\I_{\mathcal{B}},b) + |\I_{\mathcal{B}}|\leq 2.4 |\I_{\mathcal{B}}|\leq 2.4\cdot\frac{2^n}{5M_n}$. Moreover, $\rho_{M_n}(\mathrm{D})\subset\overline{\I}$ (since $M_n|D|<2^{n-1}$) and $\sigma\vert_{\overline{\I}\setminus \I_{\mathcal{B}}}=1$, by the definition of $\mathcal{B}$. Thus $\mathrm{D}$ is a plus isolated interval in $\sigma$ which contradicts the balanced assumption for $\I$, yielding that  $|\I_{\mathcal B}|\geq \frac{2^n}{5M_n}$.
    
    Taking $K\coloneqq |\I_{\mathcal{B}}|/2^m$, we partition $\I_{\mathcal{B}} = \cup_{i=1}^K \mathrm{C}_i$ with $\{\mathrm{C}_1, \dots, \mathrm{C}_K\}\subset \mathcal{I}_m^0$ being a collection of disjoint $m$-intervals. Notice that $\{\B_1,\dots,\B_k\}\subset \{\mathrm{C}_1, \dots, \mathrm{C}_K\}$. We then define $\mathscr{P}_{\mathcal{B}} = (p_i)_{i=1}^K$ by taking $p_i = 1$ when $\mathrm{C}_i\in \{B_1, \dots, B_k\}$ and $p_i=0$ otherwise. We want to show that this sequence is $1.9$-good (recall Definition \ref{def: lambda good}). By the minimality of $\I_{\mathcal{B}}$, we get that $p_1=p_K=1$. Take any interval $[k_1, k_2]\subset [1, K]$ with $\max_{k_1\leq i\leq k_2}p_i = 1$. Let $\mathrm{C} = \cup_{i=k_1}^{k_2}\mathrm{C}_{i}$ be the union  of $m$-intervals associated to $(p_i)_{i=k_1}^{k_2}$. Let $L\geq m$ be such that $2^L\leq |\mathrm{C}|\leq 2^{L+1}$.
    
    By Lemma \ref{Lemma: approximate_interval}, there exists an interval $\mathrm{D}=[a,b]\in\mathcal{I}$,  such that $\max{\{d(a,\mathrm{C}), d(b,\mathrm{C})\}}\leq 0.7|\mathrm{C}|$ and $\mathrm{C}\subset \mathrm{D}$.  
    Since $\mathrm{C}$ contains a minus spin, so does $\mathrm{D}$. In addition, since all $\ell$-intervals in $\I$ with $\ell \geq m$ are assumed to be minus vacant, there must be a minus spin $x\in\Z\setminus \mathrm{D}$ with $d(x, \mathrm{D})\leq \frac{|\mathrm{D}|}{2}$, as otherwise $\mathrm{D}$ would be isolated (so in particular, $x\in \mathcal B \setminus \mathrm{C}$). Thus, by the triangle inequality we have $d(\mathrm{C}, \mathcal{B}\setminus \mathrm{C}) \leq \max{\{d(a,\mathrm{C}), d(b,\mathrm{C})\}} + d(\mathrm{D}, \mathcal{B}\setminus \mathrm{C}) \leq 0.7 |\mathrm{C}| + \frac{|\mathrm{D}|}{2} \leq 1.9|\mathrm{C}|$. This shows in particular that there exists $y\in [1,K]$ with $p_y = 1$ and $d(y, [k_1, k_2])\leq 1.9(k_2 - k_1+1)$, implying that $\mathscr{P}_{\mathcal{B}}$ is a $1.9$-good sequence. Therefore, Lemma~\ref{Lemma: Sequence_01} gives us
    \begin{equation*}
        k = |\{1\leq i\leq K: p_i = 1\}|\geq \left(\frac{|\I_{\mathcal{B}}|}{2^m}\right)^{{\log_{3.9}(2)}} \geq \left(\frac{2^{n-m}}{5M_n}\right)^{{\log_{3.9}(2)}},
    \end{equation*}
    where in the last inequality we used the lower bound $|\I_{\mathcal{B}}|\geq\frac{2^{n}}{5M_n} $. Let $\overline{\B}_i=\rho_{\frac{3}{2}}(\B_i).$ As the lemma holds for each $\B_i$ by the induction hypothesis, we use the fact that $\overline{\B}_i\cap \overline{\B}_{i+2}= \emptyset$ to conclude 
    \begin{align*}
        J(\overline{\I}^-(\sigma),\overline{\I}^+(\sigma))&\ge \sum_{i=1}^{\lceil k/2\rceil}J(\overline{\B}_{2i-1}^-(\sigma),\overline{\B}_{2i-1}^+(\sigma))\\&\ge
        \frac{k}{2}\cdot \overline{c}_22^{m\theta}\geq 
        \frac{\overline{c}_2}{2\cdot {5}^{\log_{3.9}(2)}}\cdot 2^{(n-m)(\log_{3.9}(2)-\theta)}\cdot M_{n}^{-\log_{3.9}(2)} \cdot 2^{n\theta} \\
        &\stackrel{(*)}{\geq} \frac{\overline{c}_2}{2\cdot {5}^{\log_{3.9}(2)}}\cdot M_{n}^{C_1\left(\log_{3.9}(2)-\theta\right)-\log_{3.9}(2)} \cdot 2^{n\theta} \geq \overline{c}_22^{n\theta+\theta},
    \end{align*}
    where $(*)$ comes from the fact that $2^{n-m}\ge M_n^{C_1}$ and in the last inequality we used that $C_1\Big(\log_{3.9}(2)\\ - \theta\Big)= 500(\log_{3.9}(2)-0.5)\ge 1$ and chose $M_0>0$ large enough depending on $\theta$.
\end{proof}

We are ready to prove Proposition \ref{Prop: Energy_bound_1}. 

\begin{proof}[\textbf{Proof of Proposition \ref{Prop: Energy_bound_1}}]
    Since $\I_\sigma$ is plus favored with respect to $\sigma^S$, we see that $\sigma^S$ is all plus in $\rho_{3/2}(\I_\sigma)\setminus \I_\sigma$. If $A_\sigma = \I_\sigma$ (i.e., $\sigma^S$ is all minus in $\I_\sigma$), then we must have $|\{x\in \rho_{\frac{3}{2}}(A_\sigma) : \sigma^S_x=+1\}| = |\{x\in \rho_{\frac{3}{2}}(A_\sigma) : \sigma^S_x=-1\}| = |A_\sigma|$. We can therefore bound
    \begin{equation*}
        J(A_\sigma,A_\sigma^c)\geq \sum_{x,y\in\rho_{\frac{3}{2}}(A_\sigma)}\mathbbm{1}_{\{\sigma^S_x\neq \sigma_y^S\}}J_{xy}\geq c_1^\prime|\I_\sigma|^{2-\alpha},
    \end{equation*}
    where the second inequality is due to Lemma \ref{Lemma: interacton_1}.

    If $A_\sigma\neq\I_\sigma$, then $\sigma^S$ is not constant on $\I_\sigma$. Since in addition $\I_\sigma$ is balanced with respect to $\sigma^S$, we can apply Lemma \ref{Lemma: interacton_2} to get the desired bound. 
\end{proof}

We now want to prove Proposition \ref{Prop: Energy_bound_2}. Again, we first introduce some auxiliary lemmas. 

\begin{lemma}\label{Lemma: First_interaction}
    For any configuration $\sigma\in\Omega^+$, 
    \begin{equation*}
        \sum_{\substack{x\in A_\sigma \\  y\in \I_\sigma^c}}\mathbbm{1}_{\{\sigma_x=\sigma_y^S=-1\}}J_{xy} \leq 0.1J(A_\sigma, A_\sigma^c).
    \end{equation*}
\end{lemma}

\begin{proof}
    To ease the notation we drop the subscript in $A_{\sigma}$ and $\I_\sigma$ throughout the proof. Let $n\geq 0$ and $x\in\Z$ be such that $\I=\I_n(x)$. For every $i\in\Z$, take $\mathrm{E}_i\coloneqq \I_n(x+16i)$.  Recalling that $\I_n(x)=[2^{n-4}x-2^{n-1},2^{n-4}x-2^{n-1})$, we see that $\I_n(x)$ and $\I_n(x\pm16)$ are disjoint and thus $\{\mathrm{E}_i\}_{i\in\Z}$ is a partition of $\Z$. For $1\le i\leq M_n$, as $\I$ is plus favored with respect to $\sigma^S$, we must have $|\mathrm{E}_i^-(\sigma^S)|\leq \frac{2^n}{M_n}$.  Note that $\mathrm{E}_i$ is to the right-hand side of $\I$. Letting $\mathrm{E}_i=[x_i, x_{i+1})$, one can get that $\min_{x\in \mathrm{E}_i^-(\sigma^S)}d(x,A)\ge d(x_i,A)\ge\max_{x\in \mathrm{E}_{i-1}^+(\sigma^S)}d(x,A)$, and combining with the fact that the interaction is decreasing with the distance one can derive that $J(\mathrm{E}_i^-(\sigma^S), A)\leq \frac{2^n}{M_n}J(\{x_i\}, A)$ and $J(\mathrm{E}_{i-1}^+(\sigma^S), A)\geq \left(2^n-\frac{2^n}{M_n}\right)J(\{x_i\}, A)$. Thus we get for all $1<i\leq M_n$, 
    \begin{equation*}
    J(\mathrm{E}_i^-(\sigma^S), A) \leq \frac{1}{M_{n}-1}J(\mathrm{E}_{i-1}^+(\sigma^S), A).
    \end{equation*}
  To get a similar bound for $\mathrm{E}_1(\sigma^S)$, we notice that the first half of $\mathrm{E}_1$ must be all plus in $\sigma^S$, that is, $\mathrm{E}_1^+(\sigma^S)\supset [x_1, \frac{x_1+x_2}{2})$. This is, again, a consequence of $\I$ being plus favored. Using the decreasing property of the interaction again, we get that $J(E_1^-(\sigma^S), A)\leq \frac{2^n}{M_n}J(\frac{x_1+x_2}{2}, A)$ and $J(\mathrm{E}_{1}^+(\sigma^S), A)\geq 2^{n-1}J(\frac{x_1+x_2}{2}, A)$, and therefore $J(\mathrm{E}_1^-(\sigma^S), A) \leq \frac{2}{M_n}J(\mathrm{E}_{1}^+(\sigma^S), A)$. We can thus bound
  \begin{align}\label{Eq: interation_E_i_close}
      \sum_{i= 1}^{M_n}J(\mathrm{E}_i^-(\sigma^S), A) &\leq \frac{2}{M_n}J(\mathrm{E}_1^+(\sigma^S), A) + \sum_{i=1}^{M_n - 1}\frac{1}{M_n - 1} J(\mathrm{E}_i^+(\sigma^S), A)\nonumber \\ 
          &\leq \frac{3}{M_n - 1}J( \bigcup_{i=1}^{M_n}\mathrm{E}_i^+(\sigma^S),A).
  \end{align}
  For the distant intervals $\mathrm{E}_i$ with $i>M_n$, as $d(A, \mathrm{E}_i)>M_02^{n(1+\delta)}$, we can simply bound \begin{align}
      J(A, \bigcup\limits_{i> M_n}\mathrm{E}_i)\leq |A|\sum\limits_{R\geq M_02^{n(1+\delta)}}R^{-\alpha}&\leq CM_0^{1-\alpha}2^{n(1+\delta)(1-\alpha)}|A|\nonumber\\ &\leq CM_{ n}^{1-\alpha}|A|^{2-\alpha},\label{eq: interacion E_i far}
  \end{align} where the last inequality comes from $|A|\le 2^n$. By Corollary~\ref{Cor: F_A}, we have $\overline{c}_1|A|^{2-\alpha}\leq J(A,A^c)$. Combined with \eqref{eq: interacion E_i far}, this yields
  \begin{equation}\label{Eq: interation_E_i_far all}
      J(A, \bigcup\limits_{i> M_n}\mathrm{E}_i)\leq \frac{C}{\overline{c}_1}M_{ n}^{1-\alpha}J(A,A^c). 
  \end{equation}
  By symmetry, we have the same estimations \eqref{Eq: interation_E_i_close} and \eqref{Eq: interation_E_i_far all} for $i<0$, and therefore 
  \begin{equation*}
       \sum_{\substack{x\in A \\  y\in \I^c}}\mathbbm{1}_{\{\sigma_x=\sigma_y^S=-1\}}J_{xy}  \leq  \frac{3}{M_n - 1}J(A,\bigcup_{1\le|i|\le M_n}\mathrm{E}_i^+(\sigma^S)) +  \frac{2C}{\overline{c}_1}M_{n}^{1-\alpha}J(A,A^c).
  \end{equation*}
  The desired bound then comes from choosing $M_0$ large enough.  
\end{proof}

Next, we control the interaction between $A_\sigma^c$ and the sites in $A_\sigma$ that changed signs in the balancing procedure. Let us first introduce some definitions. 
\begin{definition}\label{Def: flipped_sites}
    Let $F(-,\sigma)\coloneqq\{x\in A_\sigma: \sigma_x=1\}$ be the collection of spins flipped to minus in the balancing procedure. For each $x\in F(-,\sigma)$, let $\I^x$ be the last interval containing $x$ which flipped $x$ in the balancing procedure of $\sigma$. Let $\mathscr{F}_\ell(-,\sigma) = \{\I^x: x\in F(-, \sigma), |\I^x| = 2^\ell\}$. Given $\I\in\mathscr{F}_\ell(-,\sigma)$, recall that $s_\I$ denotes the step in which $\I$ was selected. Therefore, we have $F(-,\sigma) \subset \cup_{\ell\geq 0}F_\ell(-,\sigma)$, where $F_\ell(-,\sigma)\coloneqq \cup_{\I\in \mathscr{F}_\ell(-,\sigma)}\I^+(\sigma^{s_\I})$. Similarly, we define $F(+,\sigma) \coloneqq \{x\in A_\sigma^c: \sigma_x=-1, \sigma^S_x=1\}$ to be the collection of spins flipped to plus in the balancing procedure and define $\mathscr{F}_\ell(+,\sigma) = \{\I^x: x\in F(+, \sigma), |\I^x| = 2^\ell\}$.
\end{definition}
Next, we prove a few lemmas and the last two will be used in the proof of Proposition~\ref{Prop: Energy_bound_2}. 

\begin{lemma}\label{lem: fake interval expansion}
Recall Definition \ref{def: Peierls map} and recall that $M^\prime_\ell = 2\left\lfloor\frac{M_\ell}{2c_1}\right\rfloor$.
    Given $\sigma\in\Omega^+$, let $n\geq 0$ be such that $|\I_\sigma| = 2^n$. Then, for all $\ell < n$ and $\I_\ell\in\mathscr{F}_\ell(-,\sigma)$, we have $\rho_{\frac{M^\prime_\ell}{2}}(\I_\ell)\subset \I_\sigma$.
\end{lemma}

\begin{proof}
    Given $\ell < n$ and $\I_\ell\in\mathscr{F}_\ell(-,\sigma)$, let $x\in\Z$ be such that  $\I_{\ell}=\I_\ell(x)$. Let $\B_k\coloneqq \I_\ell(x+16k)$ for all $1\leq |k|\leq M^\prime_\ell$. Supposing there exists $1\leq |k|\leq \frac{M^\prime_\ell}{2}$ such that $\B_k\cap \I_\sigma^c \neq \emptyset$, we next derive a contradiction. Without loss of generality, we can assume that $k$ has the smallest absolute value. Our derivation of the contradiction is divided into the following two cases.
    
    If $\ell<n-1$, since $\B_k\cap \I_\sigma^c \neq \emptyset$ we must have $B_{k^\prime}\subset \rho_{\frac{3}{2}}(\I_\sigma)\setminus \I_\sigma$, where $k^\prime = \sign(k)\cdot (|k|+1)$ and thus $\B_{k^\prime}$ is the neighboring interval of $\B_k$ which is further away from $\I_\sigma$. As $\I_\sigma$ is plus favored with respect to $\sigma^S$, all spins of $\sigma^S$ in $\rho_{\frac{3}{2}}(\I_\sigma)\setminus \I_\sigma$ are plus, and so are all spins in $\B_{k^\prime}$. However, by Proposition~\ref{Prop: balanced_intervals}, $\I_\ell$ is weakly minus favored, and thus $\B_{k^\prime}$ must be weakly minus dense, arriving at a contradiction.

    If $\ell=n-1$, then we have $\rho_{\frac{3}{2}}(\I_\sigma)\setminus \I_\sigma \subset  \B_{-2}\cup \B_{-1}\cup \B_1\cup \B_2$.  Recall that all spins of $\sigma^S$ in $\rho_{\frac{3}{2}}(\I_\sigma)\setminus \I_\sigma$ are plus. Recalling also that $\I_\ell$ is weakly minus favored with respect to $\sigma^S$ (see Proposition \ref{Prop: balanced_intervals}), we see that $\B_{i}$ must be weakly minus dense for any $i=-2,-1,1,2$. Hence the number of pluses of $\sigma^S$ in $\B_2\cup \B_1\cup \B_{-1}\cup \B_{-2}$ is at most $4\cdot \frac{2^{\ell}}{M^\prime_\ell}$. We have again arrived at a contradiction since $4\cdot \frac{2^{\ell}}{M^\prime_\ell}<2^n$ (and $\rho_{3/2}(\I_\sigma)\setminus \I_\sigma$ is all plus).
\end{proof}

Provided with Lemma \ref{lem: fake interval expansion}, we can prove the following lemma.

\begin{lemma}\label{Lemma: Interaction_far_from_I}
    Given $\sigma\in\Omega^+$ and $\I\in\mathscr{F}_\ell(-,\sigma)$, let $\widetilde{\I}\coloneqq \rho_{\frac{M^\prime_\ell}{2}}(\I)$. 
    For any $B\subset \widetilde{\I}^c$ we have
    \begin{equation*}
        J(\I^+(\sigma^{s_\I}), A_\sigma^c\cap B) \leq \frac{4}{M^\prime_\ell}J(A_\sigma \cap \widetilde{\I}, 
        A_\sigma^c\cap B)\frac{|\I^+(\sigma^{s_\I})|}{2^\ell}.
    \end{equation*}
    Analogously, for $\I\in\mathscr{F}_\ell(+,\sigma)$ and $B\subset \widetilde{\I}^c$, we have
    \begin{equation*}
        J(\I^-(\sigma^{s_\I}), A_\sigma\cap B) \leq \frac{4}{M^\prime_\ell}J(A_\sigma^c \cap \widetilde{\I}, A_\sigma\cap B)\frac{|\I^-(\sigma^{s_\I})|}{2^\ell}.
    \end{equation*}
\end{lemma}

\begin{proof} Due to symmetry, we only provide proof for the first inequality. To ease the notation, we drop the subscript in $A_{\sigma}$ in the proof. For every $y\in A^c\cap B$, we wish to lower-bound $J(A\cap\widetilde{\I}, \{y\})$. Take $a,b$ such that $\I=[a,b)$. We assume $y\geq b$, and the argument for $y<a$ follows by symmetry. 
    Let $x\in\Z$ be such that $\I=\I_\ell(x)$. By Proposition~\ref{Prop: balanced_intervals}, $\I$ is weakly minus favored with respect to $\sigma^S$. So, for all $1\leq k\leq M^\prime_\ell$, we have $\I_\ell(x+ 16k)$ is weakly minus dense. Moreover, by Lemma~\ref{lem: fake interval expansion} and Definition \ref{def: Peierls map}, we get that $\{w\in \I_\ell(x+16k) : \sigma^S_w=-1\}\subset A$ for every $1\leq |k|\leq \frac{M^\prime_\ell}{2}$. Therefore we have
   \begin{align*}
      |\{z\in A\cap \widetilde{\I}: z\geq b\}|&\geq \sum_{1\leq k\leq\frac{M^\prime_\ell}{2}}|\{w\in \I_\ell(x+16k) : \sigma^{ S}_w=-1\}| \\ 
      &\geq \left(1 - \frac{1}{M^\prime_\ell}\right)\frac{M^\prime_\ell}{2}\cdot 2^\ell \geq \frac{M^\prime_\ell2^\ell}{4}.
   \end{align*}
   
   In addition, since the interaction decreases with respect to the distance, we have $J(A\cap \widetilde{\I}, \{y\})\geq (y-b+1)^{-\alpha}\frac{M^\prime_\ell2^\ell}{4}$. Similarly we can bound $J(\I^+(\sigma^{s_\I}),\{y\})\leq (y-b+1)^{-\alpha}|\I^+(\sigma^{s_\I})|$ and therefore
   \begin{equation*}
       J(A\cap \widetilde{\I}, \{y\})\geq \frac{J(\I^+(\sigma^{s_\I}),\{y\})}{|\I^+(\sigma^{s_\I})|} \frac{M^\prime_\ell2^\ell}{4}.
   \end{equation*}
   Summing over $y\in A^c\cap B$, we get the desired bound. 
\end{proof}

For the interaction within $\rho_{\frac{M_\ell^\prime}{2}}(\I)$, we need a tighter control. Given $\I=\I_\ell(x)\in \mathscr{F}_\ell(-,\sigma)$, for each $1\leq |k|\leq \frac{M^\prime_\ell}{2}$, let us denote $\I^k\coloneqq \I_\ell(x+16k)$. 

\begin{lemma}\label{Lemma: interacton_close_to_I}
    There exists a constant $\overline{c}_3 \coloneqq \overline{c}_3(\alpha)>0$ such that, for any $\sigma$, $\I\in \mathscr{F}_\ell(-,\sigma)$ and $1\leq |k|\leq \frac{M^\prime_\ell}{2}$, 
    \begin{equation*}
        J(\I^+(\sigma^{s_\I}), \I^k \cap A_\sigma^c)\leq \overline{c}_3 (M^\prime_\ell)^{1-\alpha}J(\I^k\cap A_\sigma , \I^k\cap A^c_\sigma )\frac{|\I^+(\sigma^{s_\I})|}{2^\ell}.
    \end{equation*}
    Analogously, for $\I\in\mathscr{F}_\ell(+,\sigma)$, we have
    \begin{equation*}
        J(\I^-(\sigma^{s_\I}), \I^k \cap A_\sigma)\leq \overline{c}_3 (M^\prime_\ell)^{1-\alpha}J(\I^k\cap A_\sigma , \I^k\cap A^c_\sigma )\frac{|\I^-(\sigma^{s_\I})|}{2^\ell}.
    \end{equation*}
\end{lemma}
\begin{proof} Due to symmetry, we only provide proof for the first inequality. For notation clarity, in this proof, we drop the subscript in $A_{\sigma}$. As $\I$ is weakly minus favored with respect to $\sigma^S$ (by Proposition~\ref{Prop: balanced_intervals}), all its first $2^{\ell-1}$ neighbors are minus, hence $d(\I, \I^k\cap A^c_\sigma )\geq 2^{\ell-1}$. Thus, on the one hand, we have
    \begin{equation}\label{Eq: interacton_close_to_I_1}
        J(\I^+(\sigma^{s_\I}),  \I^k\cap A^c) \leq (2^{\ell-1})^{-\alpha}|\I^+(\sigma^{s_\I})||\I^k\cap A^c|. 
    \end{equation}
    On the other hand, as $\I^k$ is weakly plus vacant with respect to $\sigma^{S}$ (recalling that $\I$ is weakly minus favored), we have 
    \begin{equation}\label{Eq: I_k_c}
        |\I^{ k}\cap A^c| = |\{x\in \I^k: \sigma^S_x=1\}|\leq\frac{1}{M^\prime_\ell}2^\ell.
    \end{equation} As this is smaller than $|\I^k|/2$, we can apply Lemma~\ref{Lemma: interacton_1} to get $J(\I^{ k}\cap A^c, \I^{k}\cap A) \geq \overline{c}_1|\I^{k}\cap A^c|^{2-\alpha}$. Combined with \eqref{Eq: interacton_close_to_I_1}, it yields that
    \begin{align*}
         J(\I^+(\sigma^{s_\I}),  \I^k\cap A^c) & \stackrel{\eqref{Eq: interacton_close_to_I_1}}{\leq} (2^{\ell-1})^{-\alpha}|\I^+(\sigma^{s_\I})||\I^k\cap A^c|\frac{J(\I^{k}\cap A^c, \I^{k}\cap A)}{\overline{c}_1|\I^{k}\cap A^c|^{2-\alpha}} \\
         &= \frac{2^\alpha}{\overline{c}_1}J(\I^{k}\cap A^c, \I^{k}\cap A)\frac{|\I^+(\sigma^{s_\I})|}{2^\ell}\left(\frac{|\I^{k}\cap A^c|}{2^{\ell}}\right)^{\alpha-1}\\
         &\stackrel{\eqref{Eq: I_k_c}}{\leq}  \frac{2^\alpha}{\overline{c}_1}J(\I^{k}\cap A^c, \I^{k}\cap A)\frac{|\I^+(\sigma^{s_\I})|}{2^\ell}(M^\prime_\ell)^{1-\alpha}.\qedhere
    \end{align*}
\end{proof}

We are ready to prove Proposition~\ref{Prop: Energy_bound_2}.

\begin{proof}[\textbf{Proof of Proposition~\ref{Prop: Energy_bound_2}}]
Using the identity $\sigma_x\sigma_y = 2\mathbbm{1}_{\{\sigma_x= \sigma_y\}} -1$ for all $x,y\in \Z$, we can write
\begin{equation}\label{eq: energy_difference 0}
     H^+_\Lambda(\sigma) - H^+_\Lambda(\tau_{A_\sigma}(\sigma)) = - \sum_{x\in A_\sigma,y\in A_\sigma^c}2J_{xy}\sigma_x\sigma_y =  2J(A_\sigma,A_\sigma^c)-4\sum_{x\in A_\sigma,y\in A_\sigma^c}J_{xy}\mathbbm{1}_{\{\sigma_x=\sigma_y\}}.
\end{equation}
We then split $\mathbbm{1}_{\{\sigma_x= \sigma_y\}}$ according to the signs of the original configuration and the final configuration $\sigma^S$, getting that
\begin{align}\label{eq: indicator of interaction}
    \mathbbm{1}_{\{\sigma_x= \sigma_y\}} &=  \mathbbm{1}_{\{\sigma_x = \sigma_y= \sigma_y^S = - 1\}}+\mathbbm{1}_{\{\sigma_x = \sigma_y= -\sigma_y^S = -1\}}  +\mathbbm{1}_{\{\sigma_x = \sigma_y= 1\}} \nonumber \\
    & \leq \mathbbm{1}_{\{\sigma_x = \sigma_y= \sigma_y^S = - 1\}} +  \mathbbm{1}_{\{\sigma_x = \sigma_y= -\sigma_y^S = -1\}} + \mathbbm{1}_{\{\sigma_x = 1\}} .
\end{align}
Recall that $\I_\sigma$ is the $n$-interval such that $A_\sigma = \I_\sigma^-(\sigma^S)$. Recalling also Definition~\ref{Def: flipped_sites}, we have $\{y\in A_\sigma^c:-\sigma_y=\sigma^S_y=1\}=F(+,\sigma)$ and $\{x\in A_\sigma:-\sigma_x =\sigma^S_x =1\}=F(-,\sigma)$. With this, we can plug \eqref{eq: indicator of interaction} into \eqref{eq: energy_difference 0} and get that

\begin{multline}\label{eq: energy_difference}
    H^+_\Lambda(\sigma) - H^+_\Lambda(\tau_{A_\sigma}(\sigma)) \ge 2J(A_\sigma,A_\sigma^c) - 4 \sum_{\substack{x\in A_\sigma \\ y\in \I_\sigma^c}} \mathbbm{1}_{\{\sigma_x=\sigma_y=\sigma^S_y=-1\}}J_{xy} \\- 4J(A_\sigma\setminus F(-,\sigma), F(+,\sigma)) - 4J(F(-,\sigma), A_\sigma^c) . 
\end{multline}
The sum on the right-hand side of \eqref{eq: energy_difference} is only over $y\in\I_\sigma^c$ since all spins in $\I_\sigma \setminus A_\sigma$ are plus with respect to $\sigma^S$. We claim that it suffices to prove the following:
    \begin{align}
        &J(F(-,\sigma), A_\sigma^c) \leq \frac{C}{M_0}J(A_\sigma, A^c_\sigma),\label{Eq: Interation_F}\\
        &J(F(+,\sigma), A_\sigma) \leq \frac{C}{M_0}J(A_\sigma, A^c_\sigma)\label{Eq: Interation_F+},
    \end{align}
    for an appropriate constant $C>0$ depending only on $\alpha$ and $\delta$. Indeed, combining \eqref{Eq: Interation_F} and \eqref{Eq: Interation_F+} with Lemma \ref{Lemma: First_interaction} (which bounds the second term on the right-hand side of \eqref{eq: energy_difference}), we get that $$H^+_\Lambda(\sigma) - H^+_\Lambda(\tau_{A_\sigma}(\sigma)) \geq (2 - 4/10 - \frac{2C}{M_0})J(A_\sigma, A^c_\sigma),$$ and the desired result follows by taking $M_0$ large enough.
    
    We now prove \eqref{Eq: Interation_F}. Recall Definition \ref{Def: flipped_sites}. Moreover, we can restrict each $\mathscr{F}_\ell(-,\sigma)$ to disjoint intervals by taking $\mathscr{F}_{\ell,i}(-,\sigma) \coloneqq \mathscr{F}_\ell(-,\sigma)\cap \mathcal{I}_\ell^i$, for $i=0,\dots, 15$, where $\mathcal{I}_\ell^i$ was defined in \eqref{eq: sub-collection of disjoint coverings}. We can also write $F_\ell(-,\sigma) \subset \cup_{i=0}^{15} F_{\ell, i}(-,\sigma)$ where $F_{\ell,i} (-,\sigma)\coloneqq \cup_{\I\in \mathscr{F}_{\ell,i}(-,\sigma)}\I^+(\sigma^{s_\I})$, for $i=0,\dots, 15$.
    Then, 
    \begin{equation}\label{Eq: Split_F}
        J(F(-,\sigma), A_\sigma^c) \leq \sum_{\ell= 0}^n\sum_{i=0}^{15} J(F_{\ell, i}(-,\sigma), A_\sigma^c) \leq \sum_{\ell= 0}^n\sum_{i=0}^{15} \sum_{\B\in\mathcal{I}_\ell^i}J(F_{\ell, i}(-,\sigma), A_\sigma^c\cap \B). 
    \end{equation}
    
    Fixing $\B\in\mathcal{I}_\ell^i$, we now upper-bound $J(F_{\ell, i}(-,\sigma), A_\sigma^c\cap \B)$. Recall that for any $\I\in\mathscr{F}_\ell(-,\sigma)$, we denote $\widetilde{\I}= \rho_{\frac{M^\prime_\ell}{2}}(\I)$. As $\mathcal{I}_\ell^i$ forms a disjoint partition of $\Z$, for all $\I_\ell\in \mathscr{F}_{\ell, i}(-,\sigma)$ either $\B\subset \widetilde{\I}_\ell$ or $\B \cap \widetilde{\I}_\ell = \emptyset$. 
    
    We first consider $\I_\ell\in \mathscr{F}_{\ell, i}(-,\sigma)$ with $\B \cap \tilde \I_\ell = \emptyset$. As in the proof of Lemma~\ref{Lemma: not_much_flips}, we say $\I_\ell\in \mathscr{F}_{\ell, i}(-,\sigma)$ is \textit{negligible} when $|\I_\ell^{+}(\sigma^{s_{\I_\ell}})|\leq \frac{2^\ell}{M^\prime_{\ell}}$, i.e., not many spins were flipped to minus. Otherwise, we say $\I_\ell$ is \textit{non-negligible}. Similarly to \eqref{Eq: non-negligible_are_distant}, we want to lower-bound the distance between two non-negligible intervals. We cannot simply use \eqref{Eq: non-negligible_are_distant} since it was proved for intervals in  $\mathcal{F}_\ell$, not in $\mathscr{F}_{\ell, i}(-,\sigma)$ as we need. Next we show that two non-negligible intervals $\I_\ell, \I_\ell^\prime$ must satisfy 
    \begin{equation}\label{eq: non-negligible distance}
        d(\I_\ell,\I_\ell^\prime)\geq (M^\prime_\ell-1)2^\ell.
    \end{equation} 
    Indeed, for two non-negligible intervals $\I_\ell,\I_\ell^\prime$, let $s,s^\prime$ be the step in which they were flipped respectively. We can assume without loss of generality that $s<s^\prime$, i.e., $\I_\ell$ was flipped first. We get from Proposition~\ref{Prop: balanced_intervals} that  $\I_\ell$ is weakly minus favored with respect to $\sigma^{s^\prime}$.  If $d(\I_\ell, \I_\ell^\prime)\leq (M^\prime_\ell-1)2^\ell$,  then $|(\I_\ell^\prime)^-(\sigma^{s^\prime})|\ge 2^\ell-\frac{2^\ell}{M^\prime_\ell}$, which contradicts to the fact that $\I^\prime_\ell$ is non-negligible.
    This proves \eqref{eq: non-negligible distance} and implies, in particular, that $\widetilde{\I}_\ell\cap\widetilde{\I}_\ell^\prime=\emptyset$. By Lemma~\ref{Lemma: Interaction_far_from_I}, when $\B\cap\widetilde{\I}_\ell=\emptyset$ we have $J(\I_\ell^+(\sigma^{s_{\I_\ell}}), A_\sigma^c\cap \B)\leq \frac{4}{M^\prime_\ell}J(\widetilde{\I}_\ell\cap A_\sigma,  A_\sigma^c\cap \B)$. Summing over non-negligible intervals we get
    \begin{equation}\label{Eq: Int_A_1}
        \sum_{\substack{\I_\ell\in\mathscr{F}_{\ell, i}(-,\sigma), \ \B\cap\tilde{\I}_\ell = \emptyset \\ \I_\ell \text{ non-negligible}}}J(\I_\ell^+(\sigma^{s_{\I_\ell}}), A_\sigma^c \cap \B)\leq \frac{4}{M^\prime_\ell}J(A_\sigma , A_\sigma^c\cap \B).
    \end{equation}
    For negligible $\I_\ell\in\mathscr{F}_{\ell, i}^-$, by Lemma~\ref{Lemma: Interaction_far_from_I} we have 
    \begin{equation*}
        J(\I_\ell^+(\sigma^{s_{\I_\ell}}),  A_\sigma^c\cap \B)\leq  \frac{4}{M^\prime_\ell}J(\widetilde{\I}_\ell\cap A_\sigma,  A_\sigma^c \cap \B)\frac{|\I_\ell^+(\sigma^{s_{\I_\ell}})|}{2^\ell} \leq \frac{4}{M^\prime_\ell}J(\widetilde{\I}_\ell\cap A_\sigma,  A_\sigma^c \cap \B)\frac{1}{M^\prime_\ell}.
    \end{equation*} In this case, not all $\widetilde{\I}_\ell$'s are disjoint, but every $x\in\widetilde{\I}_\ell$ can be contained in at most $M^\prime_\ell$ other intervals $\widetilde{\I}_\ell^\prime$ with $\I_\ell^\prime\in\mathscr{F}_\ell(-,\sigma)$. Therefore,
    
    \begin{equation}\label{Eq: Int_A_2}
         \sum_{\substack{\I_\ell\in\mathscr{F}_{\ell, i}(-,\sigma), \ \B\cap\tilde{\I}_\ell = \emptyset \\ \I_\ell \text{ negligible}}}J(\I_\ell^+(\sigma^{s_{\I_\ell}}), A_\sigma^c \cap \B)\leq \frac{4}{M^\prime_\ell}J(A_\sigma, A_\sigma^c\cap \B).
    \end{equation}

    We now consider the interval $\I_\ell\in\mathscr{F}_{\ell, i}(-,\sigma)$ with $\B\subset \widetilde{\I}_\ell$. If $\I_\ell$ is negligible, Lemma~\ref{Lemma: interacton_close_to_I} yields
    $J(\I_\ell^+(\sigma^{s_{\I_\ell}}), A_\sigma^c \cap \B)\leq \overline{c}_3(M^\prime_\ell)^{-\alpha}J(\B\cap A_\sigma, \B\cap A^c_\sigma)$. Summing over all $M^\prime_\ell+1$ possible such intervals we get
    \begin{equation}\label{Eq: Int_A_3}
         \sum_{\substack{\I_\ell\in\mathscr{F}_{\ell, i}(-,\sigma), \ \B\subset\tilde{\I}_\ell \\ \I_\ell \text{ negligible}}}J(\I_\ell^+(\sigma^{s_{\I_\ell}}), A_\sigma^c \cap \B)\leq 2\overline{c}_3(M^\prime_\ell)^{1-\alpha}J(\B\cap A_\sigma, \B\cap A^c_\sigma).
    \end{equation}

   For a non-negligible interval $\I_\ell\in\mathscr{F}_{\ell, i}(-,\sigma)$ with $\B\subset\widetilde{\I}_\ell$, Lemma~\ref{Lemma: interacton_close_to_I} again yields $$ J(\I_\ell^+(\sigma^{s_{\I_\ell}}), A_\sigma^c \cap \B) \leq \overline{c}_3(M^\prime_\ell)^{1-\alpha}J(\B\cap A_\sigma, \B\cap A^c_\sigma).$$ Furthermore, there are at most two intervals $\I_\ell\in\mathscr{F}_{\ell, i}(-,\sigma)$ with $\B\subset\widetilde{\I}_\ell$ can be non-negligible since all non-negligible intervals are at least $(M^\prime_\ell - 1)2^\ell$ distant away from each other (see \eqref{eq: non-negligible distance}). Thus we get that  \begin{equation}\label{Eq: Int_A_4}
         \sum_{\substack{\I_\ell\in\mathscr{F}_{\ell, i}(-,\sigma), \ \B\subset\tilde{\I}_\ell \\ \I_\ell \text{ non-negligible}}}J(\I_\ell^+(\sigma^{s_{\I_\ell}}), A_\sigma^c \cap \B)\leq 2\overline{c}_3(M^\prime_\ell)^{1-\alpha}J(\B\cap A_\sigma, \B\cap A^c_\sigma).
    \end{equation} Together with \eqref{Eq: Int_A_1}, \eqref{Eq: Int_A_2} and \eqref{Eq: Int_A_3}, this yields that
    \begin{equation*}
        \sum_{\B\in\mathcal{I}_\ell^i}J(F_{\ell, i}(-,\sigma), A_\sigma^c\cap \B) \leq  (64+4\overline{c}_3)(M^\prime_\ell)^{1-\alpha} J(A_\sigma, A_\sigma^c).
    \end{equation*}
    Plugging the preceding inequality into \eqref{Eq: Split_F} we get $\eqref{Eq: Interation_F}$. 

    To prove \eqref{Eq: Interation_F+} we proceed similarly. Defining $\mathscr{F}_{\ell,i}(+,\sigma) \coloneqq \mathscr{F}_\ell(+,\sigma)\cap \mathcal{I}_\ell^i$, for $i=0,\dots, 15$,  we have $F_\ell(+,\sigma) \subset \cup_{i=0}^{15} F_{\ell, i}(+,\sigma)$ where $F_{\ell,i} (+,\sigma)\coloneqq \cup_{\I\in \mathscr{F}_{\ell,i}(+,\sigma)}\I^-(\sigma^{s_\I})$. As in  \eqref{Eq: Split_F}, we can write
    \begin{equation}\label{Eq: Split_F plus}
         J(F(+,\sigma), A_\sigma) \leq \sum_{\ell= 0}^n\sum_{i=0}^{15} J(F_{\ell, i}(+,\sigma), A_\sigma) \leq \sum_{\ell= 0}^n\sum_{i=0}^{15} \sum_{\B\in\mathcal{I}_\ell^i}J(F_{\ell, i}(+,\sigma), A_\sigma\cap \B). 
    \end{equation}
    As Lemmas \ref{Lemma: interacton_close_to_I} and \ref{Lemma: Interaction_far_from_I} also apply to intervals $\I\in\mathscr{F}_\ell(+, \sigma)$, by a completely symmetrical argument we can show the following inequalities (analogous to \eqref{Eq: Int_A_1}, \eqref{Eq: Int_A_2}, \eqref{Eq: Int_A_3} and \eqref{Eq: Int_A_4} respectively)
    \begin{equation}\label{Eq: Int_A_1_plus}
        \sum_{\substack{\I_\ell\in\mathscr{F}_{\ell, i}(+,\sigma), \ \B\cap\tilde{\I}_\ell = \emptyset \\ \I_\ell \text{ non-negligible}}}J(\I_\ell^-(\sigma^{s_{\I_\ell}}), A_\sigma\cap \B)\leq \frac{4}{M^\prime_\ell}J(A_\sigma^c , A_\sigma\cap \B),
    \end{equation}
      \begin{equation}\label{Eq: Int_A_2_plus}
         \sum_{\substack{\I_\ell\in\mathscr{F}_{\ell, i}(+,\sigma), \ \B\cap\tilde{\I}_\ell = \emptyset \\ \I_\ell \text{ negligible}}}J(\I_\ell^-(\sigma^{s_{\I_\ell}}), A_\sigma\cap \B)\leq \frac{4}{M^\prime_\ell}J(A_\sigma^c, A_\sigma\cap \B),
    \end{equation}
    \begin{equation}\label{Eq: Int_A_3_plus}
         \sum_{\substack{\I_\ell\in\mathscr{F}_{\ell, i}(+,\sigma), \ \B\subset\tilde{\I}_\ell \\ \I_\ell \text{ negligible}}}J(\I_\ell^-(\sigma^{s_{\I_\ell}}), A_\sigma\cap \B)\leq 2\overline{c}_3(M^\prime_\ell)^{1-\alpha}J(\B\cap A_\sigma^c, \B\cap A_\sigma),
    \end{equation}
    \begin{equation}\label{Eq: Int_A_4_plus}
         \sum_{\substack{\I_\ell\in\mathscr{F}_{\ell, i}(+,\sigma), \ \B\subset\tilde{\I}_\ell \\ \I_\ell \text{ non-negligible}}}J(\I_\ell^-(\sigma^{s_{\I_\ell}}), A_\sigma \cap \B)\leq 2\overline{c}_3(M^\prime_\ell)^{1-\alpha}J(\B\cap A^c_\sigma, \B\cap A_\sigma).
    \end{equation}
    Combining \eqref{Eq: Int_A_1_plus}, \eqref{Eq: Int_A_2_plus}, \eqref{Eq: Int_A_3_plus} and \eqref{Eq: Int_A_4_plus} we get
    \begin{equation*}
        \sum_{\B\in\mathcal{I}_\ell^i}J(F_{\ell, i}(+,\sigma), A_\sigma^c\cap \B) \leq  (64+4\overline{c}_3)(M^\prime_\ell)^{1-\alpha} J(A_\sigma, A_\sigma^c).
    \end{equation*}
    Plugging this inequality back in \eqref{Eq: Split_F plus} gives us  \eqref{Eq: Interation_F+}.
\end{proof}

\subsection{Entropy bound}\label{sec: entropy bound}

In this subsection, we first prove Proposition \ref{prop: entropy argument} and then use it to show Proposition \ref{prop: free energy control}.
\begin{proof}[\textbf{Proof of Proposition~\ref{prop: entropy argument}}]
    Take any $A\in\mathcal{A}_Q(\I_n)$. Recalling Definition \ref{Def: Definition of Psi}, we can view $\Psi_{\ell}(A)$ as a map with input given by the intervals in $\mathcal{I}_\ell^0$, see Figure \ref{Fig: Psi}. In order to prove $\eqref{eq: entropy bound from l+1 to l}$, we first upper-bound $|\Psi_\ell(A^\prime) : A^\prime\in\mathcal{A}_Q(\I_n), \Psi_{\ell + 1}(A^\prime) = \Psi_{\ell+1}(A)|$, the number of possible maps $\Psi_\ell(A^\prime)$  that satisfy $\Psi_{\ell+1}(A^\prime,\I_{\ell+1})=\Psi_{\ell+1}{(A,\I_{\ell+1})}$ for every interval $\I_{\ell+1}\in\mathcal{I}_{\ell+1}^0$.  For any interval $\B\in \mathcal{I}_\ell^{0}$, let $\hat{\B}$ be the $(\ell+1)$-interval in $\mathcal{I}_{\ell+1}^0$ such that $\B\subset\hat{\B}$. Then if $\Psi_{\ell+1}(A,\hat{\B})\neq 0$, we get from \eqref{eq: defi of psi} that  \begin{equation}\label{eq: information from level l+1}
        \Psi_{\ell}(A,\B)=\Psi_{\ell+1}(A,\hat{\B}).
    \end{equation} 
    
    Letting 
    \begin{equation}\label{eq: cal_I} 
        \mathscr{I}(A)\coloneqq \{\I_{\ell+1}\in\mathcal{I}_{\ell+1}^{0} : \Psi_{\ell+1}(A,\I_{\ell+1})= 0\},
    \end{equation}
    by \eqref{eq: information from level l+1}, we only need to choose the values for the $\ell$-intervals $\B$ such that $\hat{\B}\in \mathscr{I}(A)$.  Since there are $2|\mathscr{I}(A)|$ such intervals and each has three possible values, the number of possible choices is upper-bounded by $3^{2|\mathscr{I}(A)|}$, i.e., 
    \begin{equation}\label{eq: Psi_A_1}
     |\Psi_\ell(A^\prime) : A^\prime\in\mathcal{A}_Q(\I_n), \Psi_{\ell + 1}(A^\prime) = \Psi_{\ell+1}(A)| \leq 3^{2|\mathscr{I}(A)|}.
    \end{equation}
     By Lemma~\ref{lem: procedure stops}, $\sigma^S$ is balanced in $\rho_\frac{3}{2}(\I_\sigma)$. As every $\hat{\B}\in \mathscr{I}(A)$ is contained in $\rho_\frac{3}{2}(\I_\sigma)$, we get that $\sigma^S$ is balanced in $\hat{\B}$. Lemma~\ref{Lemma: interacton_2} then yields that $\overline{c}_2 2^{(\ell+1)\theta} \leq J(A\cap \rho_{\frac{3}{2}}(\hat{\B}),A^c\cap \rho_{\frac{3}{2}}(\hat{\B})).$
    Summing over $\hat{\B}\in \mathscr{I}(A)$ we get that
    \begin{equation}\label{eq: non constant interval number bound 1}
        |\mathscr{I}(A)|\overline{c}_2\cdot 2^{(\ell+1)\theta}\le \sum_{\hat{\B}\in \mathscr{I}(A)}J(A\cap \rho_{\frac{3}{2}}(\hat{\B}),A^c\cap \rho_{\frac{3}{2}}(\hat{\B}))\le 2J(A,A^c),
    \end{equation}
    where the last inequality comes from the fact that each $J(x, y)$ with $x\in A, y\in A^c$ is accounted for at most twice in total when summing $J(A\cap \rho_{\frac{3}{2}}(\hat{B}), A^c\cap \rho_{\frac{3}{2}}(\hat{\B}))$ over $\hat{\B}\in \mathscr{I}(A)$. From \eqref{eq: non constant interval number bound 1} we get that 
    \begin{equation}\label{eq: non constant interval number bound 2}
        |\mathscr{I}(A)|\le \frac{2J(A,A^c)}{\overline{c}_22^{(\ell+1)\theta}}\le \frac{{4Q}}{\overline{c}_22^{(\ell+1)\theta}}.
    \end{equation} 
     Plugging \eqref{eq: non constant interval number bound 2} into \eqref{eq: Psi_A_1} we get $|\Psi_\ell(A^\prime) : A^\prime\in\mathcal{A}_Q(\I_n), \Psi_{\ell + 1}(A^\prime) = \Psi_{\ell+1}(A)|\leq \exp{\left(\frac{{8\ln(3)Q}}{\overline{c}_22^{(\ell+1)\theta}}\right)}$. As this upper bound does not depend on $\Psi_{\ell+1}(A)$, we can sum over all possible choices of $\Psi_{\ell+1}(A)$ in $\{\Psi_{\ell+1}(A): A\in\mathcal{A}_Q(\I_n)\}$ and get
     \begin{equation*}
        E(\ell,Q,\I_n)\le E(\ell+1,Q, \I_n)\times \exp{\left(\frac{{8\ln(3)Q}}{\overline{c}_22^{(\ell+1)\theta}}\right)},
    \end{equation*} and the desired result follows from taking $c_4 \coloneqq \frac{{8\ln(3)}}{\overline{c}_22^{\theta}}$. 
\end{proof}

\begin{proof}[\textbf{Proof of Proposition~\ref{prop: free energy control}}]
    Using Proposition \ref{prop: entropy argument}, our proof follows \cite{FFS84} closely (see also \cite{Bovier.06} for a more detailed version). Recall Definition \ref{Def: bad event}. Given $A\in \hat{\mathcal{A}}(\I_n)$,  for each $1\leq \ell \leq n-3$, let $A_\ell$ be the union of $\ell$-intervals $\I_\ell\in\mathcal{I}_\ell^{0}$ such that $A\cap \I_\ell(x)\neq \emptyset$, that is, $A_{\ell}\coloneqq \displaystyle \bigcup_{\substack{\I_\ell\in\mathcal{I}_\ell^{0} : A\cap \I_\ell(x)\neq \emptyset}}\I_\ell(x)$. By the definition of $\Psi(A)$, we can also write $A_\ell= \displaystyle \bigcup_{\substack{\I_\ell\in\mathcal{I}_\ell^{0} : \Psi(A, \I_\ell) \neq 1}}\I_\ell$. Recalling \eqref{eq: hamiltonian bound}, we write \begin{equation}\label{eq: external field expansion into levels}
       \Delta_A(h)=\sum_{\ell=0}^{n-3}\big(\Delta_{A_{\ell}}(h)-\Delta_{A_{\ell+1}}(h)\big),
    \end{equation} 
    with the convention that $A_0=A$ and $h_{A_{n-2}} = h_\emptyset = 0$. By Lemma \ref{Lemma: Concentration.for.Delta.General}, we have for any $z_\ell>0$
    \begin{equation*}
        \mathbb P\left(|\Delta_{A_{\ell}}(h)-\Delta_{A_\ell+1}(h)|\ge z_\ell\right)\le C_1\exp\left(-\frac{z_\ell^2}{2{\varepsilon^2}|A_{\ell}\Delta A_{\ell+1}|}\right).
    \end{equation*}
    Recall our definition of $\mathscr{I}(A)$ in \eqref{eq: cal_I}, and recall that for any $\B\in \mathcal{I}_\ell^{0}$, $\hat{\B}\in \mathcal{I}_{\ell+1}^0$ is the $(\ell+1)$-interval such that $\B\subset\hat{\B}$.  To upper-bound the symmetric difference, first notice that $A_\ell\subset A_{\ell+1}$, by the definition of the $\Psi_\ell(A)$ map as in \eqref{eq: defi of psi}. Moreover, for any $x\in  A_{\ell+1}\setminus A_\ell$, there exists an $\ell$-interval $\B\in \mathcal{I}_\ell^{0}$ such that $x\in \B$, $\Psi_\ell(A,\B)=1$ and  $\Psi_{\ell+1}(A,\hat{\B})= 0$. Thus we get that
    $|A_{\ell}\Delta A_{\ell+1}|\le 2^{\ell}|\{\B\in \mathcal{I}_\ell^{0} : \Psi_{\ell+1}(A,\hat{\B})=0\}| = {2^{\ell+1}}|\mathscr{I}(A)|$. Using \eqref{eq: non constant interval number bound 2}, we get $|A_{\ell}\Delta A_{\ell+1}|\le \frac{4J(A,A^c)}{\overline{c}_22^{{(\ell+1)}\theta}}\cdot 2^{\ell}$ for $0\le \ell\le n-3$, and therefore
    
    \begin{equation}\label{eq: gaussian inequality at level l}
        \mathbb P(|\Delta_{A_{\ell}}(h)-\Delta_{A_{\ell+1}}(h)|\ge z_\ell) \le C_1\exp{\left(-\frac{z_\ell^2\cdot \overline{c}_22^{{(\ell+1)}\theta}}{{8\varepsilon^2}J(A,A^c){2^{\ell}}}\right)}.
    \end{equation}
   
    In order to give a union bound, we define $\hat{\cA}_{{Q}}(\I_n)=\hat{\cA}(\I_n)\cap\cA_Q(\I_n)$. As the interval $\I_n$ is fixed, we will omit it from the notation of $\hat{\cA}_{{Q}}$ for simplicity. Recalling the definition $E(\ell,Q,\I_n)$ from Proposition~\ref{prop: entropy argument}, since each $\Psi_\ell(A)$ (seen as a map) uniquely determines $A_\ell$, we have $|\{A_\ell : A\in \hat{\mathcal{A}}_Q\}|\leq |\{\Psi_\ell(A) : A\in\mathcal{A}_Q\}| = E(\ell, Q, \I_n)$. Thus, from \eqref{eq: gaussian inequality at level l} we get that for $z_{\ell, Q} \geq 0$, 
    \begin{align}\label{eq: union bound 1 at level l}
        &\mathbb P(\sup_{A\in \hat{\cA}_Q}|\Delta_{A_{\ell}(h)}-\Delta_{A_{\ell+1}}(h)| \ge z_{\ell,Q})\nonumber\\ \le~& E(\ell,Q,\I_n)\cdot E(\ell+1,Q, \I_n)\cdot C_1\exp(-\frac{z_{\ell,Q}^2\cdot \overline{c}_22^{{(\ell+1)}\theta}}{16\varepsilon^2Q2^\ell}).
    \end{align} 
    Applying \eqref{eq: entropy bound} to \eqref{eq: union bound 1 at level l} and letting ${z_{\ell,Q}=\frac{C_2Q}{2^{(\ell+1)\delta}}}$, we get that 
    \begin{align}
        \mathbb P(\sup_{A\in \hat{\cA}_Q}|\Delta_{A_{\ell}}(h)-\Delta_{A_{\ell+1}}(h)|\ge z_{\ell,Q}) \le {\frac{C_1}{C_3}}\exp(\frac{C_3Q}{2^{\ell \theta}}-\frac{{C_4}Q}{\vareps^2\cdot 2^{\ell(1-\theta+2\delta)}}).\nonumber
    \end{align}
    By the definition of $\hat{\mathcal{A}}(\I_n)$, we only need to consider $Q\ge c_2\cdot 2^{\theta n}$, and by our choice of $\delta$ (at the beginning of this section) we have  $\theta>1-\theta+2\delta$. Thus we get that 
    \begin{equation}\label{eq: union bound 2 at level l}
        \mathbb P(\sup_{A\in \hat{\cA}_Q}|\Delta_{A_{\ell}}(h)-\Delta_{A_{\ell+1}}(h)|\ge {\frac{C_2Q}{2^{(\ell+1)}\delta}})\le C_{5}\exp(-\frac{Q}{C_{5}\vareps^2\cdot 2^{\ell(1-\theta+2\delta)}}).
    \end{equation}
    Let $C_2>0$ be small enough depending on $\delta$, we get that $\sum_{\ell=0}^{n-3}\frac{C_2Q}{\vareps\cdot 2^{(\ell+1)\delta}}\le \frac{0.1Q}{\vareps}$. Thus we get from \eqref{eq: union bound 2 at level l} and a union bound that
    \begin{align}
        &\mathbb P(\sup_{A\in \hat{\cA}_Q}|\Delta_A(h)|\ge {{0.1}J(A,A^c)})\le \sum_{\ell=0}^{n-3}\mathbb P(\sup_{A\in \hat{\cA}_Q}|\Delta_{A_{\ell}}(h)-\Delta_{A_{\ell+1}}(h)|\ge {\frac{C_2Q}{2^{(\ell+1)\delta}}})\nonumber\\\le~& \sum_{\ell=0}^{n-3}C_{ 5}\exp(-\frac{Q}{C_{ 5}\vareps^2\cdot 2^{\ell(1-\theta+2\delta)}})\le C_{ 6}\exp(-\frac{Q}{C_{ 6}\vareps^2\cdot 2^{(n-2)(1-\theta+2\delta)}}).\label{eq: all levels union bound}
    \end{align}
    Summing over $Q\ge c_2\cdot 2^{\theta n}$ in \eqref{eq: all levels union bound}, taking a union bound and choosing $\varepsilon>0$ small enough gives the desired result (recall $\theta > 1- \theta + 2\delta$).
\end{proof}
 
\section{Phase Transition in 2d}\label{sec: 2d phase transition}
Throughout this section, we will always consider the long-range Ising model in $d=2$ with $2<\alpha\leq 3$.

Our proof of phase transition in two dimensions builds upon the ideas of \cite{Affonso_Bissacot_Maia_23}. Their argument uses an improved version of the long-range contour system of \cite{affonso_24}, together with the approach of \cite{Ding2021} and the coarse-graining technique of \cite{FFS84}. For $2<\alpha<3$, phase transition follows after we make a simple modification: we replace the coarse-graining estimations depending on the boundary of a contour with estimations depending on the interactions between the flipped region and the outside. This, however, is not enough to reach the critical value $\alpha=3$, for which we must also improve the counting argument in \cite{FFS84}. For the convenience of comparing notes with \cite{Affonso_Bissacot_Maia_23}, we use the same notations of \cite{Affonso_Bissacot_Maia_23} whenever possible.

Let us first introduce the long-range contour system. Given $\sigma\in\Omega^+$, the \textit{incorrect} points of $\sigma$ are defined to be $\partial\sigma \coloneqq \{x\in\Z^2 : \sigma_x=-\sigma_y \text{ for some }y\in\Z^2 \text{ with } |x-y|=1\}$. For a finite $A\Subset\Z^2$, we denote by $V(A)$ the points that are separated by $A$  from infinite (when viewing $\Z^2$ as a nearest-neighbor graph), including $A$ itself. We use the following definition to partition $\partial \sigma$.

    \begin{definition}\label{Def: delta-partiton}
    Let $M, a>0$. For each $A\Subset\Z^2$, a set $\Gamma \coloneqq \{\overline{\gamma} : \overline{\gamma} \subset A\}$ is called an $(M,a)$-\emph{partition} when the following two conditions are satisfied.
	\begin{enumerate}[label=\textbf{(\Alph*)}, series=l_after] 
		\item They form a partition of $A$, i.e.,  $\bigcup_{\overline{\gamma} \in \Gamma}\overline{\gamma}=A$ and $\overline{\gamma} \cap \overline{\gamma}' = \emptyset$ for distinct elements of $\Gamma$.  
		
		\item For all $\overline{\gamma}, \overline{\gamma}^\prime \in \Gamma$, 
			\be\label{B_distance_2}
			\d(\overline{\gamma},\overline{\gamma}') > M\min\left \{|V(\overline{\gamma})|,|V(\overline{\gamma}')|\right\}^\frac{a}{2}.
			\ee
	\end{enumerate}
\end{definition}
We choose the constants $M$ and $a$, similar to \cite{Affonso_Bissacot_Maia_23}, that is,  $a \coloneqq a(\alpha) = \frac{6}{\alpha-2}$ and $M>0$ is a large constant to be determined later. Moreover, for any $A\Subset \Z^2$, we let $\Gamma(A)$ be the finest $(M,a)$-partition of $A$, in the following sense:
given any other $(M,a)$-partition  $\Gamma^\prime$ it holds that, for all $\overline{\gamma}\in\Gamma$ there exists $\overline{\gamma}^\prime\in\Gamma^\prime(A)$ with $\overline{\gamma}\subset \overline{\gamma}^\prime$. There is an explicit construction for the finest $(M,a)$-partition: given any two $(M,a)$-partitions $\Gamma$ and $\Gamma^\prime$, the intersection of $\Gamma$ and $\Gamma^\prime$, defined as
    $\Gamma\cap\Gamma^\prime\coloneqq \{\overline{\gamma}\cap\overline{\gamma}^\prime : \overline{\gamma}  \in \Gamma , \ \overline{\gamma} \in\Gamma^\prime, \ \overline{\gamma} \cap\overline{\gamma}^\prime\neq \emptyset\}$,
 is an $(M,a)$-partition (both conditions \textbf{(A)} and \textbf{(B)} are easy to check). As the number of $(M, a)$-partitions of a finite set is finite, the intersection of all $(M, a)$-partitions of $A$ is a well-defined $(M, a)$-partition and it is also the finest. 
We write $\Gamma(\sigma)=\Gamma(\partial\sigma)$ to simplify the notation. The contours are then pairs $\gamma = (\overline{\gamma}, \lab_{\overline{\gamma}})$, where $\overline{\gamma}\in\Gamma(\sigma)$ for some configuration $\sigma$, and $\lab_{\overline{\gamma}}$ is a function that attributes a sign to each connected component of $\Z^2\setminus \overline{\gamma}$ (for every connected component $C\in \Z^2\setminus \overline{\gamma}$, $\lab_{\overline{\gamma}}(C)$ is the sign of $\sigma$ on the inner boundary of $C$. The sign on the boundary is constant by \cite[Lemma 3.10]{affonso_24}, see Figure \ref{Fig_contour_2d} for an illustration). The \textit{support} of $\gamma$ is $\Sp(\gamma)\coloneqq \overline{\gamma}$, and its \textit{interior} is $\Int(\gamma) \coloneqq V(\gamma)\setminus \Sp(\gamma)$, where  $V(\gamma)\coloneqq V(\Sp(\gamma))$. We also denote $|\gamma|\coloneqq |\Sp(\gamma)|$. Denoting by $\Int^{(1)}(\gamma),\Int^{(2)}(\gamma),\dots \Int^{(k)}(\gamma)$ the connected components of $\Int{(\gamma)}$, we can split the interior into the \textit{plus interior} $\Int_+(\gamma)$ and the \textit{minus interior} $\Int_-(\gamma)$, defined as
\begin{equation*}
    \Int_\pm(\gamma) = \hspace{-0.7cm}\bigcup_{\substack{1\le j\le k, \\ \lab_{\overline{\gamma}}(\Int^{(j)}(\gamma))=\pm 1}}\hspace{-0.7cm}\Int^{(j)}(\gamma).
\end{equation*}

A contour $\gamma\in\Gamma(\sigma)$ is \textit{external} when $V(\gamma)\not\subset V(\gamma^\prime)$ for all $\gamma^\prime\in \Gamma(\sigma)\setminus \{\gamma\}$. See Figure \ref{Fig_contour_2d} for an illustration of contours and minus interior.
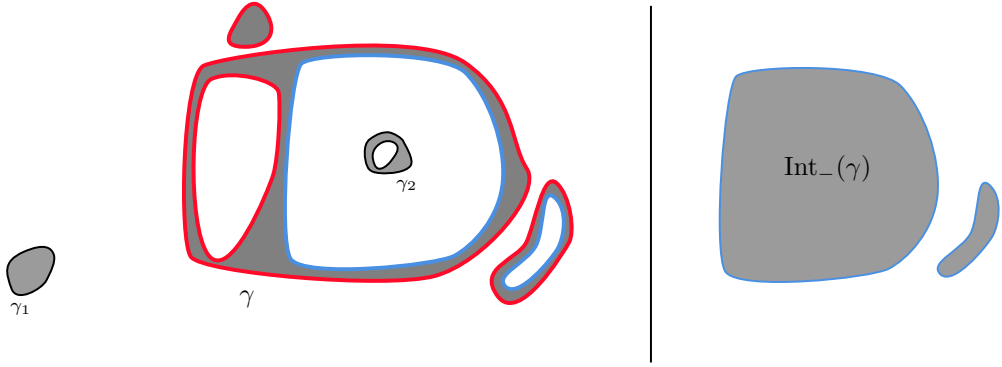
\begin{figure}[ht]
    \centering
    \input{Fig_contour_2d}
    \caption{The gray regions on the left, both light and dark gray, represent the set of incorrect points of a configuration $\sigma$. In this case, $\Gamma(\sigma) = \{\gamma, \gamma_1, \gamma_2\}$, and the dark gray region depicts $\Sp(\gamma)$, while the light gray regions delimits the supports $\Sp(\gamma_1)$ and $\Sp(\gamma_2)$. For each connected component of $\Sp(\gamma)^c$, a red boundary indicates that the label function attributes $+1$ to this region, while a blue boundary indicates the complementary case. On the right, we depict the minus interior $\Int_-(\gamma)$ of $\gamma$.}
    \label{Fig_contour_2d}
\end{figure}
The operation that erases a contour is $\tau_\gamma$, defined as 
\begin{equation*}
    \tau_\gamma(\sigma)_x = 
\begin{cases}
	\;\;\;\sigma_x &\text{ if } x \in \Int_+(\gamma)\cup V(\gamma)^c, \\
	-\sigma_x &\text{ if } x \in \Int_-(\gamma),\\
	+1 &\text{ if } x \in \Sp(\gamma).
\end{cases}
\end{equation*}

With these contours, in \cite{Affonso_Bissacot_Maia_23} they can control the energy gain of erasing a contour. Although their paper focuses only on the case $d\ge 3$, their proofs on the energy and entropy bounds apply to $d=2$ directly.
\begin{proposition}\cite[Proposition 2.10]{Affonso_Bissacot_Maia_23}\label{Prop: Cost_erasing_contour}
	For $M$ large enough, there exists a constant $b_1\coloneqq b_1(\alpha)>0$, such that for  any $\Lambda \Subset \Z^2$, $\sigma\in\Omega^+_\Lambda$ and $\gamma\in \Gamma(\sigma)$ external, it holds that
	\begin{equation*}
	  	H_{\Lambda; 0}^+(\sigma)- H_{\Lambda;0}^+(\tau_{\gamma}(\sigma))\geq b_1\Big(|\gamma| + J\big(\Int_-(\gamma)\big) \Big),  
	\end{equation*}
    with $J(A)\coloneqq J(A,A^c)$ for every $A\subset \Z^2$.
\end{proposition}

Defining $\Omega^+ \coloneqq \cup_{\Lambda\Subset\Z^2}\Omega^+_\Lambda$, let  
\begin{equation*}
    \mathcal{C}_0 \coloneqq \{\gamma \subset\Z^2: 0 \in V(\gamma), \gamma \in \Gamma(\sigma) \text{ for some }\sigma\in\Omega^+, \text{ and } \gamma \text{ external}\}
\end{equation*} be the collection of external contours surrounding the origin and $\mathcal{C}_0(n) \coloneqq \{\gamma \in \mathcal{C}_0: |\gamma|=n\}$ for $n\geq 0$. These long-range contours are in general not connected, but the control of the mutual distances between connected components of a contour (ensured by \eqref{B_distance_2} and the finest condition) is enough to yield an upper bound on the enumeration of these contours.

\begin{proposition}\cite[Corollary 3.28]{Affonso_Bissacot_Maia_23}\label{Cor: Bound_on_C_0_n}
	Let  $\Lambda\Subset \mathbb{Z}^2$. There exists a constant $b_2\coloneqq b_2(\alpha) >0$ such that, for all $n\geq 1$, $|\mathcal{C}_0(n)| \leq e^{b_2 n}$.  
\end{proposition}

As the approach in \cite{Affonso_Bissacot_Maia_23} relies on the ideas of \cite{Ding2021}, we define the good event
$$\mathcal{E}\coloneqq \left\{\Delta_{\Int_-(\gamma)}(h){\le} \frac{b_1|\gamma| + b_1 J({\Int_-(\gamma)})}{4}, \ \forall \gamma\in\mathcal{C}_0 \right\},$$ where $\Delta_A$ is the error function defined as in \eqref{Eq: Delta}, but in two dimensions. We remark that Lemma \ref{Lemma: Concentration.for.Delta.General} also holds when $d=2$. Then, the proof of phase transition in two dimensions follows the same steps of \cite[Theorem 4.1]{Affonso_Bissacot_Maia_23} once we prove the following result.
\begin{proposition}\label{prop: goal in 2d}
    For any $2<\alpha\le 3$, there exists a constant $C\coloneqq C(\alpha)$ such that $\mathbb{P}(\mathcal{E}^c)\leq e^{-\frac{C}{\varepsilon^2}}$ for all $\epsilon \leq C$.
\end{proposition}

As in \cite{Affonso_Bissacot_Maia_23}, we use a coarse-graining argument to control the external field. The argument in \cite{Affonso_Bissacot_Maia_23} uses Dudley's entropy bound \cite{Dudley67}, but we opt to follow the original work of \cite{FFS84} for two main reasons: it is the same argument used in the one-dimensional case (when proving Proposition \ref{prop: free energy control}), and it would be clearer to state the adaptations needed to reach the critical exponent $\alpha=3$.

The idea of the coarse-graining argument is to, at each scale $\ell\geq 0$, approximate each interior $\Int_-(\gamma)$ by a region $B_\ell(\Int_-(\gamma))$ close to the region $B_{\ell-1}(\Int_-(\gamma))$, in the sense that $|B_\ell(\Int_-(\gamma))\Delta\\ B_{\ell-1}(\Int_-(\gamma))|$ is small, and we should also control the number of possible approximations
at each step. We take the same approximations as in \cite{FFS84} and $\cite{Affonso_Bissacot_Maia_23}$, described next.

Fixing a parameter $r>4$, at each scale $\ell\geq 0$ we partition $\Z^2$ into cubes $$\mathrm{C}_{\ell}(x)\coloneqq \left(\prod_{i=1}^2{\left[2^{r\ell}x_i , \ 2^{r\ell}(x_i+1) \right)}\right)\cap \Z^2,$$ with $x\in\Z^{2}$. This gives a collection of disjoint cubes partitioning $\Z^2$. As in the one-dimensional case, each such cube will be called an $\ell$-cube, and $\mathrm{C}_{\ell}$ will denote an arbitrary $\ell$-cube. Given any $A\Subset\Z^2$, let $\C_{\ell}(A)$ be the collection of $\ell$-cubes which intersect $A$. We write simply $\C_{\ell}$ to denote an arbitrary collection of $\ell$-cubes. For any $A\subset\Z^2$, we define the collection of \textit{admissible} $\ell$-cubes to be 
\begin{equation}\label{eq: def of admissible cubes}
    \mathfrak{C}_\ell(A) \coloneqq \left\{\mathrm{C}_{\ell} : |\mathrm{C}_{\ell}\cap A| \geq \frac{1}{2}|\mathrm{C}_{\ell}|\right\}.
\end{equation}
Then, at each scale $\ell$, we approximate $\Int_-(\gamma)$ by $B_\ell(\Int_-(\gamma))\coloneqq \bigcup\limits_{\mathrm{C}\in\mathfrak{C}_{\ell}(\Int_-(\gamma))}\hspace{-0.5cm}\mathrm{C}$, the union of admissible $\ell$-cubes. For any $n, Q>0$, let \begin{equation}\label{eq: def of Int n,Q}
    \fInt(n,Q) \coloneqq \{\Int_-(\gamma): \gamma\in\mathcal{C}_0, \ n\leq |\gamma| < 2n \text{ and } Q\leq J(\Int_-(\gamma)) < 2Q\}.
\end{equation} We first give an upper bound on $|B_\ell\left(\fInt(n,Q)\right)| \coloneqq |\{B_\ell(\Int): \Int\in \fInt(n,Q)\}|$. To do so, we estimate the \textit{edge boundary} of admissible cubes, where the edge boundary  between two disjoint collections of $\ell$-cubes $\C_{\ell}$  and $\C_{\ell}^\prime$ is
\begin{equation}\label{eq: def of edge boundary}
    \partial(\C_{\ell},\C_{\ell}^\prime) \coloneqq \{ \{\mathrm{C}_{\ell}, \mathrm{C}^\prime_{\ell}\} : \mathrm{C}_{\ell} \in \C_{\ell}, \ \mathrm{C}_{\ell}^\prime \in \C_{\ell}^\prime \textrm{ and} \  \mathrm{C}_{\ell}^\prime \text{ shares an edge with }\mathrm{C}_{\ell}\}.
\end{equation}
To ease the notation, we will write $\partial\C_{\ell} =\partial(\C_{\ell},\C_{\ell}^c)$, where $\C_{\ell}^c = \C_{\ell}(\Z^2)\setminus \C_{\ell}$. We will also write $(\mathrm{C},\mathrm{C}^\prime)\in\partial \mathfrak{C}_{\ell}(\Int)$ to denote a pair in $\partial \mathfrak{C}_{\ell}(\Int)$ with $\mathrm{C}\in\mathfrak{C}_{\ell}(\Int)$ and $\mathrm{C}^\prime\notin \mathfrak{C}_{\ell}(\Int)$, for any $\Int = \Int_-(\gamma)$ for some $\gamma\in \mathcal C_0$. It is clear that $B_\ell(\Int_-(\gamma))$ is uniquely determined by $\partial \mathfrak{C}_\ell(\Int_-(\gamma))$. The next proposition controls the entropy of the admissible regions. 

\begin{proposition}\cite[Proposition 3.30]{Affonso_Bissacot_Maia_23}\label{Prop: Entropy_admissible_sets}
    Given $n, Q, \ell \geq 0$ and $M_{n, Q, \ell}>0$, let $\widehat{\fInt}(n,Q)$ be the collection of $\Int\in \fInt(n,Q)$ such that 
    \begin{equation}\label{Eq: general_bound_admissible_cubes}
        |\partial \mathfrak{C}_\ell(\Int)| \leq M_{n, Q, \ell}. 
    \end{equation}
    Then, denoting $B_\ell(\widehat{\fInt}(n,Q))\coloneqq \{B_\ell(\Int) : \Int \in \widehat{\fInt}(n,Q)\}$, there exist constants $b_3\coloneqq b_3(\alpha)$ and $\kappa = \kappa(\alpha)$ such that,
    \begin{equation}\label{Eq: Bound_on_boundary_of_admissible_sets}
        |B_\ell\left(\widehat{\fInt}(n, Q)\right)|\leq  \exp{\left\{b_3 \frac{\ell^{\kappa + 1} n}{2^{2r\ell}} + b_3 M_{n,Q,\ell} \right\}}. 
    \end{equation}
\end{proposition}

The statement of  \cite[Proposition 3.30]{Affonso_Bissacot_Maia_23} is slightly less general since they only consider $M_{n,Q,\ell} = C_1n/2^{r\ell}$ for a suitable constant $C_1>0$, but we can prove Proposition \ref{Prop: Entropy_admissible_sets} by adapting verbatim the arguments in \cite[Proposition 3.30]{Affonso_Bissacot_Maia_23}. In light of Proposition \ref{Prop: Entropy_admissible_sets}, the key remaining task is to bound $|\partial \mathfrak C_\ell(\Int)|$ for $\Int \in \fInt(n, Q)$. For $2<\alpha < 3$, a uniform bound suffices, but the same does not hold in the critical case of $\alpha = 3$, for which a more delicate estimate is required.
The next subsections are dedicated to proving Proposition \ref{prop: goal in 2d}, first for $2<\alpha<3$, then for the critical $\alpha=3$. 

\subsection{Phase transition for $2<\alpha< 3$}\label{sec: 2d not crtitical}

In previous works, the upper bound for $|\partial \mathfrak{C}_\ell(\Int_-(\gamma))|$ is given in terms of $|\gamma|$, see \cite[Proposition 2]{FFS84} and \cite[Proposition 3.17]{Affonso_Bissacot_Maia_23}. Our idea is to give an upper bound for $|\partial \mathfrak{C}_\ell(\Int_-(\gamma))|$ in terms of $J(\Int_-(\gamma))$, taking advantage of the large interaction between pairs in edge boundary $\partial \mathfrak{C}_\ell(\Int_-(\gamma))$.

\begin{proposition}\label{Prop: Prop.1.FFS_alpha_small}
    There exist constants $b_4,b_5$ depending only on $r$ such that, for every $\Int\in\fInt(n,Q)$ 
\begin{equation}\label{Eq: Prop.1.FFS.i_alpha_small}
    |{\partial}\mathfrak{C}_\ell(\Int)| \leq b_4 \frac{Q}{2^{r\ell(4-\alpha)}}, 
\end{equation}
    and 
\begin{equation}\label{Eq: Prop.1.FFS.ii_alpha_small}
    |B_\ell(\Int)\Delta B_{\ell+1}(\Int)| \leq b_5 2^{r\ell(\alpha-2)}Q, 
\end{equation}
for every $\ell\geq 0$.
\end{proposition}
The proof of this proposition is very similar to that of \cite[Proposition 2]{FFS84} and \cite[Proposition 3.17]{Affonso_Bissacot_Maia_23}. Moreover, in the next subsection, we will prove Proposition \ref{Prop: Prop.1.FFS}, which immediately implies Proposition \ref{Prop: Prop.1.FFS_alpha_small} (see Remark \ref{rmk: proof_of_weaker_prop}). As a result, we omit the proof of Proposition \ref{Prop: Prop.1.FFS_alpha_small}.

Using Propositions \ref{Prop: Entropy_admissible_sets} and \ref{Prop: Prop.1.FFS_alpha_small} one can prove Proposition \ref{prop: goal in 2d} following the arguments of \cite{Affonso_Bissacot_Maia_23}, with some minor modifications. We present a short alternative proof here, with the hope that an interested reader can see all the required adaptations to treat the case $\alpha = 3$.

\begin{proof}[\textbf{Proof of Proposition~\ref{prop: goal in 2d} for $2<\alpha<3$. }]
If $J(\Int_-(\gamma))<2$, then we get that $\Int_-(\gamma)=\emptyset$ and thus $\Delta_{\Int_-(\gamma)}=0$. Hence in order to upper-bound $\mathbb{P}\left(\mathcal{E}^c \right)$, it suffices  to consider $\Int\in\fInt(n,Q)$ for $n\ge1$ and $ Q\geq 2$.
A simple union bound yields

\begin{align}\label{eq: union_bound_bad_event}
    &\mathbb{P}\left(\mathcal{E}^c \right) = \mathbb{P}\left(\sup_{\substack{\gamma\in\mathcal{C}_0}} \frac{\Delta_{\Int_-(\gamma)}(h)}{|\gamma| +  J({\Int_-(\gamma)})} > \frac{b_1}{4}\right) 
    \leq \sum_{n\ge 1, Q\geq 2}\mathbb{P}\left(\sup_{\substack{\Int\in\fInt(n,Q)}} {\Delta_{\Int}(h)}> 
 {\frac{b_1}{4}(n + Q)} \right). 
\end{align}
For any $L\geq 0$ we can write  
\begin{equation}\label{Eq: Telescopic_Delta_1}
    \Delta_{\Int}(h) = \sum_{\ell = 0}^{L} \left( \Delta_{B_\ell(\Int)}(h) - \Delta_{B_{\ell+1}(\Int)}(h) \right) + \Delta_{B_{L+1}(\Int)}(h),
\end{equation}
with the convention that $\Delta_\emptyset(h)=0$. Write $\Delta(\ell, \Int,h)=\Delta_{B_\ell(\Int)}(h) - \Delta_{B_{\ell+1}(\Int)}(h)$ for notational convenience. By \eqref{Eq: Prop.1.FFS.i_alpha_small}, there is no admissible $\ell$-cube when $2^{r\ell(4-\alpha)}>2b_4Q$. Thus, taking  $L_Q \coloneqq \ceil{\frac{\log_{2}(2b_4Q)}{r(4-\alpha)}}$ we have $B_{L_Q+1}(\fInt(n,Q)) = \{\emptyset\}$, so the last term on the right hand side of \eqref{Eq: Telescopic_Delta_1} vanishes. Let $g_{\ell}=\min\{\ell+1,L_Q-\ell+1\}$. As $\frac{b_1 }{4}(n + Q)\geq \frac{b_1}{40}\sum_{\ell = 0}^{L_Q}(\frac{n+Q}{g_{\ell}^2})$, we can bound 
\begin{align}
\begin{split}\label{Eq: Union_coarse_graining}
    &\mathbb{P}\left(\sup_{\substack{\Int\in\fInt(n,Q)}} {\Delta_{\Int}(h)}> 
 {\frac{b_1}{4}(n + Q)}\right)  \leq \sum_{\ell = 0}^{L_Q}\mathbb{P}\left(\sup_{\substack{\Int\in\fInt(n,Q)}} |\Delta(\ell, \Int,h)|> 
 \frac{b_1(n+Q)}{40g_{\ell}^2} \right).
\end{split}
\end{align}

Using Lemma~\ref{Lemma: Concentration.for.Delta.General}, \eqref{Eq: Prop.1.FFS.ii_alpha_small} and a union bound, we have 
\begin{align*}
&\mathbb{P}\left(\sup_{\substack{\Int\in\fInt(n,Q)}}|\Delta(\ell, \Int,h)|> {\frac{b_1 }{40}\left(\frac{n+Q}{g_{\ell}^2}\right)}\right) \\
  \leq~& |B_{\ell}(\fInt(n, Q))|\cdot |B_{\ell+1}(\fInt(n, Q))| \cdot\exp{\left\{-\frac{C_1[(n + Q)g_\ell^{-2}]^2}{\varepsilon^2 2^{r\ell(\alpha-2)} Q}\right\}}\\
 \leq~& |B_{\ell}(\fInt(n, Q))|\cdot |B_{\ell+1}(\fInt(n, Q))|\cdot\exp{\left\{ -\frac{C_1}{\varepsilon^22^{r\ell(\alpha-2)}} \left(\frac{n+Q}{g_\ell^4}\right) \right\}}.
\end{align*}
Combined with Proposition \ref{Prop: Entropy_admissible_sets} applied to $M_{n, Q, \ell} = 2b_4\frac{Q}{2^{r\ell(4-\alpha)}}$ and $M_{n, Q, \ell+1} = 2b_4\frac{Q}{2^{r(\ell+1)(4-\alpha)}}$, it yields that 
\begin{align}
&\mathbb{P}\left(\sup_{\substack{\Int\in\fInt(n,Q)}}|\Delta(\ell, \Int,h)|> {\frac{b_1 }{40}\left(\frac{n+Q}{g_{\ell}^2}\right)}\right) \nonumber\\
\leq~& \exp{\left\{C_2\frac{(\ell+1)^{\kappa + 1}n}{2^{2r\ell}} + C_2\frac{Q}{2^{r\ell(4-\alpha)}}\right\}}\nonumber\cdot\exp{\left\{ -\frac{C_1}{\varepsilon^22^{r\ell(\alpha-2)}} \left(\frac{n+Q}{g_\ell^4}\right) \right\}}\nonumber \\
\leq~& \exp{\left\{ -\frac{C_3n}{\varepsilon^22^{r\ell(\alpha - 2)}g_\ell^4} - \frac{C_3Q}{\varepsilon^22^{r\ell(\alpha - 2)}g_\ell^4}  \right\}}.\label{eq: ell level free energy diff bound}
\end{align}
In the last inequality, we used the fact that $\frac{2^{2r\ell}}{(\ell+1)^{\kappa+1}}\ge C 2^{r\ell(\alpha-2)}g_\ell^4$ and $2^{r\ell(4-\alpha)}\geq C 2^{r\ell(\alpha-2)}g_\ell^4$  for some constant $C>0$ since $\alpha<3$ and thus we could choose $\varepsilon>0$ small enough such that $$\frac{C_1}{\varepsilon^22^{r\ell(\alpha-2)}g_\ell^4}-\frac{C_2(\ell+1)^{\kappa + 1}}{2^{2r\ell}}\ge \frac{C_3}{\varepsilon^22^{r\ell(\alpha - 2)}g_\ell^4}~\mbox{ and }~\frac{C_1}{\varepsilon^22^{r\ell(\alpha-2)}g_\ell^4}-\frac{C_2}{2^{r\ell(4-\alpha)}}\ge \frac{C_3}{\varepsilon^22^{r\ell(\alpha - 2)}g_\ell^4}.$$

The result follows now from some straightforward computations. 
Since $\alpha>2$, we get that $2^{r(\alpha-2)(L_Q-\ell+1)}>Cg_\ell^4 $ for some $C>0$ depending only on $\alpha$. Thus, by \eqref{Eq: Union_coarse_graining} and \eqref{eq: ell level free energy diff bound}, we get that
\begin{align}
    \mathbb{P}\left(\mathcal{E}^c \right)&\leq \sum_{Q\geq 2}\sum_{n\geq 1}\sum_{\ell = 0}^{L_Q} \exp\left\{ -\frac{C_3n}{\varepsilon^22^{rL_Q(\alpha - 2)}g_\ell^4}- C_3\frac{Q}{\varepsilon^22^{r\ell(\alpha - 2)}g_\ell^4}  \right\} \nonumber \\
    &\leq  \sum_{Q\geq 2}\sum_{n\geq 1}\exp{\left\{- \frac{C_3n}{\varepsilon^2(L_Q+1)^42^{rL_Q(\alpha - 2)}}\right\}}\sum_{\ell = 0}^{L_Q}\exp\left\{- {C_4}\frac{Q}{\varepsilon^22^{rL_Q(\alpha - 2)}}  \right\}.
    \label{eq: bad event final calculation small}
\end{align} Recalling that  $L_Q = \left\lceil\frac{\log_{2}(2b_4Q)}{r(4-\alpha)}\right\rceil$, we get that $2^{rL_Q(4-\alpha)}\leq 2^{r(4-\alpha)}\cdot 2b_4Q$. Thus, we have 
\begin{align}
    &\sum_{\ell=0}^{L_Q}\exp\left\{- {C_4}\frac{Q}{\varepsilon^22^{rL_Q(\alpha - 2)}}  \right\} \leq  (L_Q + 1)\exp{\left\{- \frac{C_5}{\varepsilon^2}Q^{1 - \frac{\alpha - 2}{4-\alpha}}\right\}}\nonumber \\ 
    &\hspace{3cm}\leq  2\left\lceil\frac{\log_{2}(2b_4Q)}{r(4-\alpha)}\right\rceil\exp{\left\{- \frac{C_5}{\varepsilon^2}Q^{\frac{6-2\alpha}{4-\alpha}}\right\}} \leq C_6^{-1}\exp{\left\{- \frac{C_6}{\varepsilon^2}Q^{\frac{6-2\alpha}{4-\alpha}}\right\}}.\label{eq: bad event final calculation small 1}
\end{align}
In addition, since $2<\alpha<3$ and using again that $2^{rL_Q}\leq 2^r(2b_4)^{\frac{1}{4-\alpha}} Q^{\frac{1}{4-\alpha}}$, we can upper-bound $(L_Q+1)^42^{r(\alpha - 2)L_Q}\leq C_7 Q$ for some constant $C_7>0$ depending only on $r$ and $\alpha$. Thus we get that \begin{equation}\label{eq: bad event final calculation small 3}
    \sum_{n\geq 1}\exp{\left\{- \frac{C_3n}{\varepsilon^2(L_Q+1)^42^{rL_Q(\alpha - 2)}}\right\}}\le C_8(L_Q+1)^42^{rL_Q(\alpha - 2)}\le C_9 Q.
\end{equation}  Plugging \eqref{eq: bad event final calculation small 1} and \eqref{eq: bad event final calculation small 3} into \eqref{eq: bad event final calculation small}, we get that
\begin{align*}
\mathbb{P}\left(\mathcal{E}^c \right)\leq  \sum_{Q\geq 2}C_9Q\cdot C_6^{-1} \exp\left\{- \frac{C_6 Q^{\frac{6-2\alpha}{4-\alpha}}}{\varepsilon^2}\right\}.
\end{align*}
The result follows by taking $\varepsilon^2$ sufficiently small. 
\end{proof}

\subsection{Phase Transition for $\alpha = 3$}\label{sec: 2d crtitical}

To reach the critical $\alpha=3$ we want to take advantage of interactions in each level. For any edge boundary $(\mathrm{C},\mathrm{C}^\prime)$, we start by considering interactions restricted to smaller cubes inside $\mathrm{C}$ and $\mathrm{C}^\prime$. Given any $\mathrm{C}_{\ell}(x)$, let 
\begin{equation}\label{eq: C_hat}
    \widehat{\mathrm{C}}_{\ell}(x) \coloneqq \left(\prod_{i=1}^2{\left[2^{r\ell}x_i +2^{r(\ell-1)}, \ 2^{r\ell}(x_i+1) - 2^{r(\ell - 1)}\right)}\right)\cap \Z^2 
\end{equation} 
be obtained from the cube $\mathrm{C}_{\ell}(x)$ after we remove an outer layer with thickness $2^{r(\ell - 1)}$, see Figure \ref{Fig: hat_C}. For an arbitrary $\ell$-cube $\mathrm{C}_{\ell}$, we denote ${\widehat{\mathrm{C}}_{\ell}}$ the cube of the form \eqref{eq: C_hat} sharing the same center as $\mathrm{C}_{\ell}$.  

\begin{lemma}\label{Lemma: Large_int}
    There exists a constant $b_6\coloneqq b_6(\alpha)$ such that, for any $A\Subset\Z^2$ and $(\mathrm{C},\mathrm{C}^\prime)\in \partial \mathfrak{C}_{\ell}(A)$ with $\ell\geq 1$,
    \begin{equation*}
        J(A\cap \widehat{\mathrm{C}}, A^c\cap \widehat{\mathrm{C}}^\prime) \geq b_6 2^{r\ell(4-\alpha)}.
    \end{equation*}
\end{lemma}
\begin{proof}
    For any $x\in \widehat{\mathrm{C}}$ and $y\in {\widehat{\mathrm{C}}^\prime}$, we have $J_{xy}\ge (3\cdot 2^{r\ell})^{-\alpha}$. Thus we can bound
    \begin{equation}\label{Eq: J_hat_C_aux}
           J(A\cap \widehat{\mathrm{C}}, A^c\cap {\widehat{\mathrm{C}}^\prime}) \geq \frac{1}{3^{\alpha}}|A\cap \widehat{\mathrm{C}}||A^c\cap {\widehat{\mathrm{C}}^\prime}|2^{-r\ell\alpha}.
    \end{equation}
    Moreover, we have $|A\cap \widehat{\mathrm{C}}| \geq |A\cap \mathrm{C}| - |\mathrm{C}\setminus \widehat{\mathrm{C}}| \geq 2^{2r\ell}(\frac{1}{2} - \frac{4}{2^r})$, where the second inequality follows from the admissibility condition as in \eqref{eq: def of admissible cubes}. Recalling that $r> 4$, we have $\frac{1}{2} - \frac{4}{2^r} \geq \frac{1}{4}$ and thus $|A\cap \widehat{\mathrm{C}}|\geq 2^{2r\ell - 2}$. An analogous argument shows that $|A^c\cap {\widehat{\mathrm{C}}^\prime}|\geq 2^{2r\ell - 2}$, and therefore  \eqref{Eq: J_hat_C_aux} yields
    \begin{equation*}
            J(A\cap \widehat{\mathrm{C}}, A^c\cap {\widehat{\mathrm{C}}^\prime}) \geq \frac{1}{16\cdot 3^\alpha}2^{r\ell(4-\alpha)},
    \end{equation*}
    proving the lemma for $b_6 = \frac{1}{16\cdot 3^\alpha}$.
\end{proof}

The key advantage of considering smaller boxes $\widehat{\mathrm{C}}$ is that their interactions do not overlap in the coarse-graining approximations. This is demonstrated in the next lemma.

\begin{lemma}\label{Lemma: No_overlap}
    For every $\ell\geq 0$ and $A\Subset \Z^2$, defining 
    \begin{equation}\label{Eq: Def_Q_ell}
        Q_\ell(A) \coloneqq \sum_{(\mathrm{C},\mathrm{C}^\prime)\in \partial\mathfrak{C}_\ell(A)}J(A\cap \widehat{\mathrm{C}}, A^c\cap {\widehat{\mathrm{C}}^\prime}),
    \end{equation}
    we have
    \begin{equation*}
        \sum_{\ell\geq 0}Q_\ell(A)\leq J(A).
    \end{equation*}
\end{lemma}
\begin{proof}
    Given any $(\mathrm{C}_{\ell}, \mathrm{C}_{\ell}^\prime)\in \partial \mathfrak{C}_{\ell}(A)$, it is enough to show that no interaction between ${\widehat{\mathrm{C}}_{\ell}}$ and $\widehat{\mathrm{C}}^\prime_{\ell}$ appears in $Q_k(A)$ for all $k>\ell$. An interaction $J_{xy}$ with $x\in A\cap {\widehat{\mathrm{C}}_{\ell}}$ and $y\in A^c\cap {\widehat{\mathrm{C}}^\prime_{\ell}}$ appears in $Q_k(A)$ only if there exists $(\mathrm{C}_{k}, \mathrm{C}^\prime_{k})\in \partial \mathfrak{C}_k(A)$ with $ \mathrm{C}_{\ell}\subset \mathrm{C}_{k}$ and $ \mathrm{C}^\prime_{\ell}\subset \mathrm{C}^\prime_{k}$. As both pairs $\{\mathrm{C}_{\ell},  \mathrm{C}^\prime_{\ell}\}$ and $\{ \mathrm{C}_{k},  \mathrm{C}^\prime_{k}\}$ share an edge, one of the edges of $\mathrm{C}_{\ell}$ must be contained in an edge of $\mathrm{C}_{k}$, and the same holds for $\mathrm{C}^\prime_{\ell}$ and $\mathrm{C}^\prime_{k}$, see Figure~\ref{Fig: hat_C} for an illustration with $k = \ell+1$. As $2^{r(k-1)}\geq 2^{r\ell}$, we must have $\widehat{\mathrm{C}}_{k}\cap \mathrm{C}_{\ell}= \emptyset$ and similarly $\widehat{\mathrm{C}}^\prime_{k}\cap \mathrm{C}^\prime_{\ell}= \emptyset$.  
\end{proof}

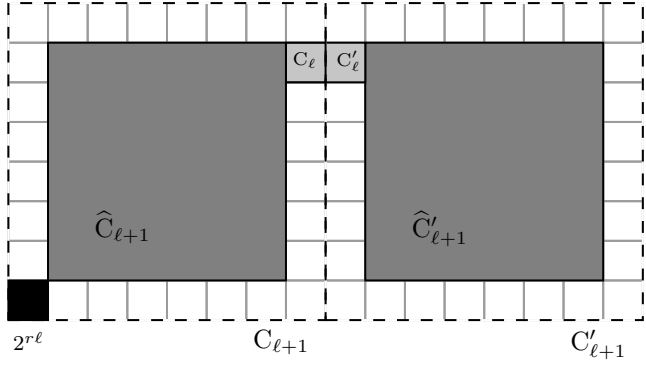
\begin{figure}
    \centering
    \input{Fig_C_hat}
    \caption{In this picture, we take $r=3$. The dashed lines delimit two $(\ell+1)$-cubes, $\mathrm{C}_{\ell+1}$ and $\mathrm{C}^\prime_{\ell+1}$, while the dark gray areas are $\widehat{\mathrm{C}}_{\ell+1}$ and $\widehat{\mathrm{C}}_{\ell+1}^\prime$. The black cube represents one of the $\ell$-cubes inside $\mathrm{C}_{\ell+1}$, and the light gray cubes are two neighboring $\ell$-cubes $\mathrm{C}_\ell\subset \mathrm{C}_{\ell+1}$ and $\mathrm{C}_\ell^\prime\subset\mathrm{C}_{\ell+1}^\prime$.}
    \label{Fig: hat_C}
\end{figure}

The next proposition is a modification of \cite[Proposition 1]{FFS84} and \cite[Proposition 3.16]{Affonso_Bissacot_Maia_23}, which is also a generalization of Proposition~\ref{Prop: Prop.1.FFS_alpha_small}. 
\begin{proposition}\label{Prop: Prop.1.FFS}
    There exists a  constant $b_7$ depending only on $r$ such that, for any $A\Subset \Z^2$ and $\ell \geq 0$, we have
\begin{equation}\label{Eq: Prop.1.FFS.i}
    |{\partial}\mathfrak{C}_\ell(A)| \leq  \frac{Q_\ell(A)}{b_62^{r\ell(4-\alpha)}}, 
\end{equation}
    and 
\begin{equation}\label{Eq: Prop.1.FFS.ii}
    |B_\ell(A)\Delta B_{\ell+1}(A)| \leq b_7 2^{r\ell(\alpha-2)} Q_\ell(A). 
\end{equation}
\end{proposition}
\begin{remark}\label{rmk: proof_of_weaker_prop}
    By the definition of $Q_\ell(A)$ we can directly bound $Q_\ell(A)\leq J(A)$, so Proposition \ref{Prop: Prop.1.FFS_alpha_small} follows immediately from Proposition \ref{Prop: Prop.1.FFS}.
\end{remark}

\begin{proof}[\textbf{Proof of Proposition~\ref{Prop: Prop.1.FFS}}]
   For \eqref{Eq: Prop.1.FFS.i}, combining Lemma \ref{Lemma: Large_int} with \eqref{Eq: Def_Q_ell} yields that
   \begin{equation*}
       Q_\ell(A)\geq b_62^{4-\alpha}|\partial\mathfrak{C}_{\ell}(A)|.
   \end{equation*}

    We now turn to the proof of \eqref{Eq: Prop.1.FFS.ii}. Given $\mathrm{C}\in \mathscr{C}_{\ell+1}(B_{\ell+1}(A)\setminus B_{\ell}(A))$, there exists a pair $\{\mathrm{C}_{\ell}, \mathrm{C}^\prime_{\ell}\}\in\partial \mathfrak{C}_\ell(A)$ with $\mathrm{C}_{\ell}\cup \mathrm{C}^\prime_{\ell}\subset \mathrm{C}$. By Lemma~\ref{Lemma: Large_int} we have
\begin{align*}\label{Eq: bound.on.c.bar}
    |\left(B_{\ell+1}(A)\setminus B_\ell(A)\right)\cap \mathrm{C}| &\leq |\mathrm{C}| = 2^{2r(\ell+1)} =2^{2r}2^{r\ell(\alpha - 2)}2^{r\ell(4-\alpha)} \\
    &\leq  b_6^{-1} 2^{2r}2^{r\ell(\alpha - 2)} J(A\cap \widehat{\mathrm{C}}_\ell, A^c\cap \widehat{\mathrm{C}}^\prime_\ell).
\end{align*}
Therefore, as all cubes  $\mathrm{C}\in \mathscr{C}_{\ell+1}(B_{\ell+1}(A)\setminus B_{\ell}(A))$ are disjoint, 
\begin{align*}
    |B_{\ell+1}(A)\setminus B_\ell(A)| &= \sum_{\mathrm{C}\in \mathscr{C}_{\ell+1}(B_{\ell+1}(A)\setminus B_\ell(A))} |(B_{\ell+1}(A)\setminus B_\ell(A))\cap \mathrm{C}| \\
    &\leq\sum_{\mathrm{C}\in \mathscr{C}_{\ell+1}(B_{\ell+1}(A)\setminus B_\ell(A))}  b_6^{-1} 2^{2r}2^{r\ell(\alpha - 2)} J(A\cap \widehat{\mathrm{C}}_\ell, A^c\cap \widehat{\mathrm{C}}_\ell^\prime)\\& \leq b_6^{-1}2^{2r}2^{r\ell(\alpha - 2)} Q_\ell(A). 
\end{align*}
To get the same bound for $|B_{\ell}(A)\setminus B_{\ell+1}(A)|$ we use a similar argument, covering $B_{\ell}(A)\setminus B_{\ell+1}(A)$ with ${(\ell+1)}$-cubes {and considering the edge boundaries in those $(\ell+1)$-cubes}. This concludes the proof taking $b_7 =  b_6^{-1} 2^{2r + 1}$.
\end{proof}

\begin{remark}
Notice that the proof of \eqref{Eq: Prop.1.FFS.ii} also works when $|B_{\ell+1}(A)\Delta B_\ell(A)| = |B_\ell(A)|$, that is, when $\ell$ is the last scale with non-empty admissible cubes. In this case, \eqref{Eq: Prop.1.FFS.ii} actually bounds $|B_\ell(A)|$.
\end{remark}

For $\alpha=3$, we can prove a stronger version of Lemma~\ref{Lemma: interacton_1} in two dimensions. Let us recall the following definition from the previous sections. Given $A\subset \Z^2$ and a configuration $\sigma\in\Omega$, $A^+(\sigma)\coloneqq \{x\in A : \sigma_x=1\}$ is the collection of plus spins in $A$, and similarly $A^-(\sigma)$ is the collection of minus spins in $A$. We start with an auxiliary lemma.
\begin{lemma}\label{lem: isoperimetric inequality for cubes}
    For any $0\leq\ell<L$, let $\C, \C^\prime$ be two disjoint collections of $\ell$-cubes such that $\cup_{\mathrm{C}\in \C\cup\C^\prime}\mathrm{C}$ is an $L$-cube. If $\min\{|\C|, |\C^\prime|\}\ge c{2^{2r(L-\ell)}}$ for some constant $0<c\le 0.5$, then
    \begin{equation*}
        |\partial(\C,\C^\prime) |\ge \sqrt{c}{2^{r(L-\ell)}}.
    \end{equation*}
\end{lemma}
\begin{proof}
    By the isoperimetric inequality, $|\partial\C|\ge 4\sqrt{|\C|}$ and $|\partial\C^\prime|\ge 4\sqrt{|\C^\prime|}$. Notice that 
$|\partial\C|+|\partial\C^\prime|=2|\partial(\C,\C^\prime) |+|\partial(\C\cup\C^\prime) |$. In addition, since the union of the $\ell$-cubes in $\C\cup\C^\prime$ is an $L$-cube, we get that $|\partial(\C\cup\C^\prime) |=4\cdot 2^{r(L-\ell)}$. Therefore,
    \begin{align}
        |\partial(\C,\C^\prime) |& \geq 2\sqrt{|\C|}+2\sqrt{|\C^\prime|}-2\cdot 2^{r(L-\ell)}\nonumber\\&=2\cdot\frac{(\sqrt{|\C|}+\sqrt{|\C^\prime|}- 2^{r(L-\ell)})\cdot (\sqrt{|\C|}+\sqrt{|\C^\prime|}+ 2^{r(L-\ell)})}{\sqrt{|\C|}+\sqrt{|\C^\prime|}+ 2^{r(L-\ell)}}.\label{eq: iso step 1}
    \end{align} Note that $|\C|+|\C^\prime|=2^{2r(L-\ell)}$, we get that  
    \begin{equation}\label{eq: iso step 2}
        (\sqrt{|\C|}+\sqrt{|\C^\prime|}- 2^{r(L-\ell)})\cdot (\sqrt{|\C|}+\sqrt{|\C^\prime|}+ 2^{r(L-\ell)})=2\sqrt{|\C|\cdot|\C^\prime|}\ge\sqrt{2c}\cdot2^{2r(L-\ell)}.
    \end{equation}In addition, we have \begin{equation}\label{eq: iso step 3}
        \sqrt{|\C|}+\sqrt{|\C^\prime|}+ 2^{r(L-\ell)}\le  (\sqrt{2}+1)\cdot2^{r(L-\ell)}.
    \end{equation}Plugging \eqref{eq: iso step 2} and \eqref{eq: iso step 3} into \eqref{eq: iso step 1}, we 
  complete the proof of Lemma~\ref{lem: isoperimetric inequality for cubes}.
\end{proof}

\begin{lemma}\label{Lemma: Int_boundary_cubes}
    There exists a constant $b_8>0$ {depending only on $r$} such that for any $L$-cube $\mathrm{C}$ and any configuration $\sigma$, if $\min\{|\mathrm{C}^+(\sigma)|,|\mathrm{C}^-(\sigma)|\}= m$, then we have \begin{equation}\label{eq: 2d critical energy bound}
       {J(\mathrm{C}^+(\sigma), \mathrm{C}^-(\sigma))}\ge b_8 \sqrt{m+1}\ln(m+1).
        \end{equation}
\end{lemma}
\begin{proof}
   Our proof is based on an induction on $L$. If $L=1$, the desired result follows by taking $b_8>0$ small enough. Now we assume \eqref{eq: 2d critical energy bound} holds for any $\ell\le L-1$ and then prove the case for $\ell=L$. Without loss of generality one can assume $|\mathrm{C}^+(\sigma)|\ge |\mathrm{C}^-(\sigma)|$. Recall the definition of $\mathfrak{C}_\ell\left(\mathrm{C}^-(\sigma)\right)$ from \eqref{eq: def of admissible cubes}. Let $a_\ell\coloneqq|\mathfrak{C}_\ell\left(\mathrm{C}^-(\sigma)\right)|$.  We consider four different cases. \\

    \textit{Case (i): $m\le 2^{2r(L-1)-1}$}. We first write $\mathrm{C}$ as the disjoint union of $(L-1)$-cubes $\mathrm{C}_1,\mathrm{C}_2,\cdots , \mathrm{C}_{2^{2r}}$. Let $x_i=|\mathrm{C}_i^-(\sigma)|$, for $1\le i\le 2^{2r}$, be the number of minus spins in $\mathrm{C}_i$. Since $\sum_{i=1}^kx_i=m\le 2^{2r(L-1)-1}$, we must have $\min\{|\mathrm{C}_i^+(\sigma)|,|\mathrm{C}_i^-(\sigma)|\}=x_i$ for all $1\leq i\leq 2^{2r}$. Thus, we can apply the induction hypothesis to get 
    \begin{equation}\label{eq: coarse-graining L-1}
        J( \mathrm{C}^+(\sigma),  \mathrm{C}^-(\sigma))\ge \sum_{i=1}^{2^{2r}} J( \mathrm{C}_i^+(\sigma),  \mathrm{C}_i^-(\sigma)) \ge \sum_{i=1}^{2^{2r}} b_8 \sqrt{x_i+1}\ln(x_i+1).
    \end{equation}
   The desired result follows once we notice that $f(x)=\sqrt{x+1}\ln(x+1)$ is a concave function for $x\ge 0$, and thus we get from Karamata's inequality that $\sum_{i=1}^{2^{2r}} \sqrt{x_i+1}\ln(x_i+1)\ge \sqrt{m+1}\ln(m+1)$. \\

   \textit{Case (ii):  $m > 2^{2r(L-1)-1}$ and $a_\ell < \frac{m}{2^{2r\ell+1}}$ for some $\ell\le \frac{L}{2}$.} In this case, there are at least $\frac{m}{2}$ minuses contained in non-admissible $\ell$-cubes with respect to $\mathrm{C}$ (recall \eqref{eq: def of admissible cubes}). Let $\mathrm{B}_1,\cdots,\mathrm{B}_k$ be the $\ell$-cubes in $\mathrm{C}$ that are not in $\mathfrak{C}_\ell( \mathrm{C}^-(\sigma))$, and as in Case (i), let $x_i = |B_i^-(\sigma)|$, for $1\leq i\leq k$, be the number of minuses in those cubes. Then we get from the induction hypothesis that \begin{equation}\label{eq: coarse-graining ell}
        J( \mathrm{C}^+(\sigma),  \mathrm{C}^-(\sigma))\ge \sum_{i=1}^k J( \B_i^+(\sigma),   \B_i^-(\sigma))\ge \sum_{i=1}^k b_8 \sqrt{x_i+1}\ln(x_i+1).
    \end{equation} 
    Again we can use that $f(x)=\sqrt{x+1}\ln(x+1)$ is concave, and thus by Karamata's inequality we get that $$\sum_{i=1}^k \sqrt{x_i+1}\ln(x_i+1),~0\le x_i\le 2^{2r\ell-1}$$ reaches its minimum if at most one $x_i$ takes value in $(0,2^{2r\ell-1})$ and others have values either 0 or $2^{2r\ell-1}$. Thus we get that \begin{equation}\label{eq: minimum for concave function}
        \sum_{i=1}^k \sqrt{x_i+1}\ln(x_i+1)\ge \left\lfloor\frac{m}{2^{2r\ell}}\right\rfloor\cdot \sqrt{2^{2r\ell-1}+1}\ln(2^{2r\ell-1}+1)\ge m\cdot 2^{-r\ell-0.5}.
    \end{equation}
    Combining \eqref{eq: coarse-graining ell}, \eqref{eq: minimum for concave function} and the fact that $2^{2r(L-1)-1}\le m\le 2^{2rL-1}$, we get that \begin{equation*}
         J( \mathrm{C}^+(\sigma),  \mathrm{C}^-(\sigma))\ge m\cdot 2^{-r\ell-0.5}\ge \sqrt{m}\cdot 2^{rL-r\ell-1}\ge C_2\sqrt{m+1}\cdot \ln(m+1),
    \end{equation*}
    for some small absolute constant $C_2>0$, where the last transition used the assumption that $\ell \leq L/2$. This then completes the verification of \eqref{eq: 2d critical energy bound} (we may choose $b_8$ smaller than the absolute constant $C_2$ here). \\ 

    \textit{Case (iii): $m > 2^{2r(L-1)-1}$, $a_j \geq \frac{m}{2^{2rj+1}}$ for all $j \leq \frac{L}{2}$ and $a_\ell > 3\cdot 2^{2r(L-\ell)-2}$ for some $\ell\le \frac{L}{2}$.}  Let $\mathrm{D}_1,\cdots,\mathrm{D}_k$ be the $\ell$-cubes in  $\mathfrak{C}_\ell( \mathrm{C}^-(\sigma))$ with $x_i\coloneqq|\mathrm{D}_i^-(\sigma)|\le 3\cdot2^{2r\ell-2}$. Then we get that $m\ge k\cdot 2^{2r\ell-1}+(a_\ell-k)\cdot 3\cdot2^{2r\ell-2}$. Combining with the fact that $m \le 2^{2rL-1}$ and $a_\ell > 3\cdot 2^{2r(L-\ell)-2}$, we get that $k\ge 3a_\ell-\frac{m}{2^{2r\ell-2}}\ge
    2^{2r(L-\ell)-2}$. Applying induction hypothesis to $\mathrm{D}_i$ we conclude that \begin{equation*}
         J( \mathrm{C}^+(\sigma),  \mathrm{C}^-(\sigma))\ge \sum_{i=1}^k J( \mathrm{D}_i^+(\sigma),   \mathrm{D}_i^-(\sigma))\ge 2^{2r(L-\ell)-2}\cdot b_8 \sqrt{2^{2r\ell-2}+1}\ln(2^{2r\ell-2}+1).
    \end{equation*} 
    Since $\ell\le L/2$ and  $m\le \frac{2^{2rL}}{2}$, the desired lower-bound holds.\\
    
    \textit{Case (iv): $m > 2^{2r(L-1)-1}$ and $\frac{m}{2^{2r\ell+1}}\le a_\ell\leq 3\cdot 2^{2r(L-\ell)-2}$ for all $\ell\le \frac{L}{2}$.} In this case, we recall \eqref{Eq: Def_Q_ell} and we next prove that for any $\ell\le\frac{L}{2}$, we have $Q_\ell(\mathrm{C}^-(\sigma))\ge C_3 \sqrt{m}$ for some absolute constant $C_3>0$. Having assumed $m > 2^{2r(L-1)-1}$ and $\frac{m}{2^{2r\ell+1}} \leq a_\ell$, we can apply Lemma~\ref{lem: isoperimetric inequality for cubes} (since $\min\{a_\ell, 2^{2r(L-\ell)}-a_\ell\}\ge \frac{m}{2^{2r\ell+1}}$) to get that $|\partial\left(\mathfrak{C}_\ell\left(\mathrm{C}^-(\sigma)\right),\mathfrak{C}_\ell\left(\mathrm{C}^+(\sigma)\right)\right)| \geq \sqrt{\frac{m}{2^{2r\ell+1}}}$. Then, recalling \eqref{Eq: Def_Q_ell}, we get from Lemma~\ref{Lemma: Large_int} that 
    $$Q_\ell(\mathrm{C}^-(\sigma))\ge |\partial\left(\mathfrak{C}_\ell\left(\mathrm{C}^-(\sigma)\right), \mathfrak{C}_\ell\left(\mathrm{C}^+(\sigma)\right)\right)| \cdot b_2 2^{r\ell}\ge \sqrt{\frac{m}{2^{2r\ell+1}}}\cdot b_2 2^{r\ell}\ge C_3 \sqrt{m+1}.$$
    Summing over $0\le \ell\le \frac{L}{2}$ and combining with Lemma~\ref{Lemma: No_overlap} we conclude that $$J( \mathrm{C}^+(\sigma),  \mathrm{C}^-(\sigma))\ge \frac{L}{2}\cdot C_3 \sqrt{m+1}.$$ The desired result follows by noticing that $\ln(m+1)\le 2rL$ (and we may choose $b_8 \leq \frac{C_3}{4r}$). 
\end{proof}

We are ready to prove the main result of this section.

\begin{proof}[\textbf{Proof of Proposition \ref{prop: goal in 2d} for $\alpha=3$}.]

Recall \eqref{eq: union_bound_bad_event} and \eqref{Eq: Telescopic_Delta_1}. Recall the definition of $\fInt(n, Q)$ from \eqref{eq: def of Int n,Q}. Consider $\Int\in \fInt(n,Q)$ for $n\ge 1$ and $Q\ge 2$.
Given any pair $(\mathrm{C},\mathrm{C}^\prime)\in\partial\mathfrak{C}_\ell(\Int)$ (recall the definition of edge boundary from \eqref{eq: def of edge boundary}), there exists a cube $\mathrm{C}_1\subset \mathrm{C}\cup \mathrm{C}^\prime$ with side length $2^{r\ell}$ and satisfying $\min{\{|\mathrm{C}_1\cap \Int|, |\mathrm{C}_1\cap \Int^c|\}}\geq 2^{2r\ell - 2}$. 
Considering $\eta = \mathbbm{1}_{\Int} - \mathbbm{1}_{\Int^c}$ the configuration that is plus in $\Int$ and minus in $\Int^c$, we have $\min{\{|\mathrm{C}_1\cap \Int|, |\mathrm{C}_1\cap \Int^c|\}} = \min{\{|\mathrm{C}_1^+(\eta)|, |\mathrm{C}_1^-(\eta)|\}}$, and by Lemma \ref{Lemma: Int_boundary_cubes} we get that
\begin{equation}\label{eq: edge pair interaction}
    J(\Int)\geq J(\Int\cap \mathrm{C}_1, \Int^c\cap \mathrm{C}_1) =  J(\mathrm{C}_1^{+}(\eta), \mathrm{C}_1^-(\eta)) \geq \frac{b_8}{4}\ell2^{r\ell}.
\end{equation} Fixing $L \coloneqq L_Q$, where $L_Q$ is the largest integer $\ell$ satisfying $\frac{b_8}{4}\ell2^{r\ell}\leq Q$, by \eqref{eq: edge pair interaction} we have $B_{L+1}(\Int) = \emptyset$ for every $\Int\in \fInt(n,Q)$.  Let $\mathcal{H}_\ell = \mathcal{H}_{\ell, n, Q}$ be the event such that for all $\Int\in \fInt(n, Q)$, $|\Delta(\ell, \Int,h)|\leq  C_1\left( Q_\ell(\Int) + \frac{n}{(\ell+1)^{2}}  + \frac{Q}{(L-\ell + 1)^{2}} \right)$
with $C_1<0.1b_1(2\pi)^{-2}$. Lemma~\ref{Lemma: No_overlap} guarantees that $C_1\sum_{\ell = 0}^{L}(Q_{\ell}(\Int) + \frac{n}{(\ell+1)^2} + \frac{Q}{(L-\ell + 1)^{2}} ) \leq \frac{b_1 }{4}(n + Q)$, so by 
 \eqref{Eq: Telescopic_Delta_1} we have \begin{equation}
      \mathcal{E}^c \subset \cup_{\ell = 0}^{L} \mathcal{H}^c_\ell.\label{eq: Telescopic_Delta_2}
 \end{equation}
 For $0\leq \ell \leq L$, we wish to split $\Int\in\fInt(n, Q)$ according to $Q_\ell(\Int)$ (recall \eqref{Eq: Def_Q_ell}). To this end, we note that Lemma \ref{Lemma: Large_int} assures that $Q_\ell(\Int)\geq b_62^{r\ell}$ as long as $\mathfrak{C}_\ell(\Int)\neq \emptyset$. Thus, defining $\fInt_{\ell, k}(n,Q)\coloneqq \{ \Int\in \fInt(n, Q) : Q_\ell(\Int)\in [b_62^{r\ell}2^k, b_62^{r\ell}2^{k+1})\}$, we can bound
 
\begin{align}\label{Eq: P_H_ell}
    \mathbb{P}(\mathcal{H}^c_\ell)\leq \sum_{k= 0}^{K_\ell}\mathbb{P}\Bigg(& \sup_{\Int\in \fInt_{\ell, k}(n,Q)} |\Delta(\ell, \Int,h)|\geq  z_{\ell,k,n,Q}\Bigg).
\end{align}
where $z_{\ell,k,n,Q}=C_1\left( b_62^{r\ell}2^k + \frac{n}{(\ell+1)^{2}}  + \frac{Q}{(L-\ell + 1)^{2}}  \right)$, and $K_\ell$ is the smallest integer $K$ satisfying $\frac{b_8}{4}L2^{rL}< b_62^{\ell}2^{K + 1}$. By our choice of $L=L_Q$, we have $L2^{rL}>\frac{4}{b_8}Q$, and thus  $\fInt_{\ell, k}(n,Q)=\{\emptyset\}$ for $k>K_\ell$.

By \eqref{Eq: Prop.1.FFS.i}, we can apply Proposition \ref{Prop: Entropy_admissible_sets} with $M_{n, Q, \ell} = 2^{k + 1}$,  getting that \begin{equation}
    |B_{\ell}(\fInt_{\ell, k}(n,Q))|\cdot |B_{\ell+1}(\fInt_{\ell, k}(n,Q))|\leq \exp{\left\{C_2 \frac{\ell^{\kappa + 1} n}{2^{2r\ell}} + C_22^{k} \right\}}.\label{eq: Prop.1.FFS.ii application}
\end{equation}  Using Lemma~\ref{Lemma: Concentration.for.Delta.General}, \eqref{Eq: Prop.1.FFS.ii} and a union bound, we get that

\begin{align*}
&\mathbb{P}\Bigg( \sup_{\Int\in \fInt_{\ell, k}(n,Q)} |\Delta(\ell, \Int,h)| \geq z_{\ell,k,n,Q} \Bigg)\\
\leq~& |B_{\ell}(\fInt_{\ell, k}(n,Q))|\cdot |B_{\ell+1}(\fInt_{\ell, k}(n,Q))|\cdot \exp{\left\{-\frac{C_3[ 2^{r\ell} 2^k + n(\ell+1)^{-2} + Q(L-\ell + 1)^{-2}]^2}{\varepsilon^2 2^{2r\ell} 2^k}\right\}} \\
\leq~& |B_{\ell}(\fInt_{\ell, k}(n,Q))|\cdot |B_{\ell+1}(\fInt_{\ell, k}(n,Q))|\cdot \exp{\left\{ -\frac{C_4}{\varepsilon^2}\left[\frac{n}{(\ell+1)^{2}2^{r\ell}} + \frac{Q}{2^{r\ell}(L-\ell+1)^{2}} + {2^k}\right] \right\}}.
\end{align*}
Combined with \eqref{eq: Prop.1.FFS.ii application}, it yields 
\begin{align}
    &\sum_{k=0}^{K_\ell} \mathbb{P}\left( \sup_{\Int\in \fInt_{\ell, k}(n,Q)} |\Delta(\ell, \Int,h)| \geq z_{\ell,k,n,Q} \right)\nonumber\\
 \leq~&\sum_{k=0}^{K_\ell} \exp{\left\{C_2\frac{(\ell+1)^{\kappa + 1}n}{2^{2r\ell}} + C_2{2^k}\right\}}\exp{\left\{ -\frac{C_4}{\varepsilon^2}\left[\frac{n}{(\ell+1)^{2}2^{r\ell}} + \frac{Q}{2^{r\ell}(L-\ell+1)^{2}} + {2^k}\right] \right\}}\nonumber \\
 \leq~& \exp{\left\{ -\frac{C_5n}{\varepsilon^2(\ell+1)^{2}2^{r\ell}} - C_5\frac{Q}{\varepsilon^22^{r\ell}(L-\ell+1)^{2}}  \right\}}\sum_{k\geq 0}\exp{\left\{{-\frac{C_52^{k}}{\varepsilon^2}}\right\}} \nonumber \\
 \leq~& \exp{\left\{ -\frac{C_6n}{\varepsilon^2(\ell+1)^{2}2^{r\ell}} - C_6\frac{Q}{\varepsilon^22^{r\ell}(L-\ell+1)^{2}}  \right\}},\label{eq: ell level free energy diff bound critical}
\end{align}
where the second and the last inequalities hold for $\varepsilon$ small enough. Combining\eqref{eq: union_bound_bad_event}, \eqref{eq: Telescopic_Delta_2}, \eqref{Eq: P_H_ell} and \eqref{eq: ell level free energy diff bound critical},
we have 
\begin{align*}
     &\mathbb{P}\left(\mathcal{E}^c \right)\leq \sum_{Q\geq 2}\sum_{n\geq 1}\sum_{\ell = 0}^{L}\exp{\left\{ -\frac{C_6n}{\varepsilon^2(\ell+1)^{2}2^{r\ell}} - C_6\frac{Q}{\varepsilon^22^{r\ell}(L-\ell+1)^{2}}  \right\}} \\
    \leq~&  \sum_{Q\geq 2}\sum_{n\geq 1}\exp{\left\{- \frac{C_6n}{\varepsilon^2(L+1)^{2}2^{rL}}\right\}}\sum_{\ell = 0}^{L}\exp{\left\{ - C_6\frac{Q}{\varepsilon^22^{r\ell}(L-\ell+1)^{2}}\right\}}.
\end{align*}
Again by our choice of $L=L_Q$ we have $\ln(Q)2^{rL}\leq C_{7}Q$, for $C_{7}$ large enough. Thus, 
\begin{align}
   \sum_{\ell = 0}^{L}\exp{\left\{- C_6\frac{Q}{\varepsilon^22^{r\ell}(L-\ell+1)^{2}}  \right\}} 
   &\leq  \sum_{\ell = 0}^{L}\exp{\left\{- C_6C_7\frac{\ln(Q)\cdot 2^{r(L - \ell)}}{\varepsilon^2(L-\ell+1)^{2}}  \right\}}.\label{eq: 2d critical final calculation 1}
\end{align} Recalling that $r>4$, we have $\frac{2^{r(L - \ell)}}{(L-\ell+1)^{2}}\ge (L-\ell+1)$. Combined with \eqref{eq: 2d critical final calculation 1}, it yields that
\begin{align}
    \sum_{\ell = 0}^{L}\exp{\left\{- C_6\frac{Q}{\varepsilon^22^{r\ell}(L-\ell+1)^{2}}  \right\}}
    &\leq  \sum_{\ell = 0}^{L}\exp{\left\{- C_6C_7\frac{\ln(Q)\cdot (L-\ell+1)}{\varepsilon^2}  \right\}}\nonumber\\
   &=  \sum_{\ell = 0}^{L} Q^{ -\frac{C_6C_7(L-\ell+1)}{\vareps^2}} \leq C_{8}Q^{ -\frac{C_{8}}{\vareps^2}}.\nonumber
\end{align}
Moreover, one can easily bound 
\begin{align*}
    \sum_{n\geq 1}\exp{\left\{- \frac{C_6n}{\varepsilon^2(L+1)^{2}2^{rL}}\right\}} \leq \sum_{n\geq 1}\exp{\left\{- \frac{C_{9}n}{\varepsilon^2\ln(Q)^{2}Q}\right\}} \leq C_{10} Q (\ln Q)^2,
\end{align*}
 and we conclude that $  \mathbb{P}\left(\mathcal{E}^c \right) \leq\sum_{Q\geq 2} C_{8}C_{10}Q^{1 - C_{8}/\varepsilon^2}(\ln Q)^2$, and the desired result follows by taking $\varepsilon$ large enough. 
\end{proof}

\section*{Acknowledgements}
J. Ding is supported by NSFC Tianyuan Key Program Project No.12226001, NSFC Key Program Project No.12231002, and by New Cornerstone Science Foundation through the XPLORER PRIZE. F. Huang is supported by the Beijing Natural Science Foundation Undergraduate Initiating Research Program (Grant No. QY23006). We thank Lucas Affonso for discussions during his visit to Peking University. 

\bibliographystyle{habbrv} 
\bibliography{bib.bib} 

\end{document}

%% file: Fig_plus-favored.tex
\tikzset{every picture/.style={line width=0.75pt}} 

\begin{tikzpicture}[x=0.75pt,y=0.75pt,yscale=-1,xscale=1]

\draw [color={rgb, 255:red, 255; green, 10; blue, 40 }  ,draw opacity=1 ][line width=4.5]    (19.7,35.3) -- (580.4,35.18) ;
\draw [color={rgb, 255:red, 74; green, 144; blue, 226 }  ,draw opacity=1 ][line width=4.5]    (259.7,35.3) -- (269.7,35.3) ;
\draw [color={rgb, 255:red, 74; green, 144; blue, 226 }  ,draw opacity=1 ][line width=4.5]    (279.7,35.3) -- (289.7,35.3) ;
\draw [color={rgb, 255:red, 74; green, 144; blue, 226 }  ,draw opacity=1 ][line width=4.5]    (299.7,35.3) -- (309.7,35.3) ;
\draw [color={rgb, 255:red, 74; green, 144; blue, 226 }  ,draw opacity=1 ][line width=4.5]    (199.7,35.3) -- (209.7,35.3) ;
\draw [color={rgb, 255:red, 74; green, 144; blue, 226 }  ,draw opacity=1 ][line width=4.5]    (189.7,35.3) -- (199.7,35.3) ;
\draw [color={rgb, 255:red, 74; green, 144; blue, 226 }  ,draw opacity=1 ][line width=4.5]    (109.7,35.3) -- (119.7,35.3) ;
\draw [color={rgb, 255:red, 74; green, 144; blue, 226 }  ,draw opacity=1 ][line width=4.5]    (49.7,35.3) -- (59.7,35.3) ;
\draw [color={rgb, 255:red, 74; green, 144; blue, 226 }  ,draw opacity=1 ][line width=4.5]    (379.7,35.3) -- (389.7,35.3) ;
\draw [color={rgb, 255:red, 74; green, 144; blue, 226 }  ,draw opacity=1 ][line width=4.5]    (389.7,35.3) -- (399.7,35.3) ;
\draw [color={rgb, 255:red, 74; green, 144; blue, 226 }  ,draw opacity=1 ][line width=4.5]    (449.7,35.3) -- (459.7,35.3) ;
\draw [color={rgb, 255:red, 74; green, 144; blue, 226 }  ,draw opacity=1 ][line width=4.5]    (509.7,35.3) -- (519.7,35.3) ;
\draw [color={rgb, 255:red, 74; green, 144; blue, 226 }  ,draw opacity=1 ][line width=4.5]    (149.7,35.3) -- (159.7,35.3) ;
\draw [color={rgb, 255:red, 74; green, 144; blue, 226 }  ,draw opacity=1 ][line width=4.5]    (329.7,35.3) -- (339.7,35.3) ;
\draw [color={rgb, 255:red, 74; green, 144; blue, 226 }  ,draw opacity=1 ][line width=4.5]    (539.7,35.3) -- (549.7,35.3) ;
\draw    (260,25) -- (260,45) ;
\draw    (340,25) -- (340,45) ;
\draw    (420,25) -- (420,45) ;
\draw    (500,25) -- (500,45) ;
\draw    (580,25) -- (580,45) ;
\draw    (180,25) -- (180,45) ;
\draw    (100,25) -- (100,45) ;
\draw    (20,25) -- (20,45) ;
\draw [color={rgb, 255:red, 74; green, 144; blue, 226 }  ,draw opacity=1 ][line width=4.5]    (549.7,35.3) -- (559.7,35.3) ;
\draw [color={rgb, 255:red, 74; green, 144; blue, 226 }  ,draw opacity=1 ][line width=4.5]    (559.7,35.3) -- (569.7,35.3) ;
\draw [color={rgb, 255:red, 74; green, 144; blue, 226 }  ,draw opacity=1 ][line width=4.5]    (59.7,35.3) -- (69.7,35.3) ;
\draw [color={rgb, 255:red, 74; green, 144; blue, 226 }  ,draw opacity=1 ][line width=4.5]    (69.7,35.3) -- (79.7,35.3) ;
\draw [color={rgb, 255:red, 74; green, 144; blue, 226 }  ,draw opacity=1 ][line width=4.5]    (529.7,35.3) -- (539.7,35.3) ;
\draw [color={rgb, 255:red, 0; green, 0; blue, 0 }  ,draw opacity=1 ]   (190,30) -- (190,40) ;
\draw [color={rgb, 255:red, 0; green, 0; blue, 0 }  ,draw opacity=1 ]   (90,30) -- (90,40) ;
\draw [color={rgb, 255:red, 0; green, 0; blue, 0 }  ,draw opacity=1 ]   (330,30) -- (330,40) ;
\draw [color={rgb, 255:red, 0; green, 0; blue, 0 }  ,draw opacity=1 ]   (320,30) -- (320,40) ;
\draw [color={rgb, 255:red, 0; green, 0; blue, 0 }  ,draw opacity=1 ]   (310,30) -- (310,40) ;
\draw [color={rgb, 255:red, 0; green, 0; blue, 0 }  ,draw opacity=1 ]   (300,30) -- (300,40) ;
\draw [color={rgb, 255:red, 0; green, 0; blue, 0 }  ,draw opacity=1 ]   (290,30) -- (290,40) ;
\draw [color={rgb, 255:red, 0; green, 0; blue, 0 }  ,draw opacity=1 ]   (280,30) -- (280,40) ;
\draw [color={rgb, 255:red, 0; green, 0; blue, 0 }  ,draw opacity=1 ]   (270,30) -- (270,40) ;
\draw [color={rgb, 255:red, 0; green, 0; blue, 0 }  ,draw opacity=1 ]   (110,30) -- (110,40) ;
\draw [color={rgb, 255:red, 0; green, 0; blue, 0 }  ,draw opacity=1 ]   (120,30) -- (120,40) ;
\draw [color={rgb, 255:red, 0; green, 0; blue, 0 }  ,draw opacity=1 ]   (130,30) -- (130,40) ;
\draw [color={rgb, 255:red, 0; green, 0; blue, 0 }  ,draw opacity=1 ]   (140,30) -- (140,40) ;
\draw [color={rgb, 255:red, 0; green, 0; blue, 0 }  ,draw opacity=1 ]   (150,30) -- (150,40) ;
\draw [color={rgb, 255:red, 0; green, 0; blue, 0 }  ,draw opacity=1 ]   (160,30) -- (160,40) ;
\draw [color={rgb, 255:red, 0; green, 0; blue, 0 }  ,draw opacity=1 ]   (170,30) -- (170,40) ;
\draw [color={rgb, 255:red, 0; green, 0; blue, 0 }  ,draw opacity=1 ]   (250,30) -- (250,40) ;
\draw [color={rgb, 255:red, 0; green, 0; blue, 0 }  ,draw opacity=1 ]   (240,30) -- (240,40) ;
\draw [color={rgb, 255:red, 0; green, 0; blue, 0 }  ,draw opacity=1 ]   (230,30) -- (230,40) ;
\draw [color={rgb, 255:red, 0; green, 0; blue, 0 }  ,draw opacity=1 ]   (220,30) -- (220,40) ;
\draw [color={rgb, 255:red, 0; green, 0; blue, 0 }  ,draw opacity=1 ]   (210,30) -- (210,40) ;
\draw [color={rgb, 255:red, 0; green, 0; blue, 0 }  ,draw opacity=1 ]   (200,30) -- (200,40) ;
\draw [color={rgb, 255:red, 0; green, 0; blue, 0 }  ,draw opacity=1 ]   (350,30) -- (350,40) ;
\draw [color={rgb, 255:red, 0; green, 0; blue, 0 }  ,draw opacity=1 ]   (30,30) -- (30,40) ;
\draw [color={rgb, 255:red, 0; green, 0; blue, 0 }  ,draw opacity=1 ]   (40,30) -- (40,40) ;
\draw [color={rgb, 255:red, 0; green, 0; blue, 0 }  ,draw opacity=1 ]   (50,30) -- (50,40) ;
\draw [color={rgb, 255:red, 0; green, 0; blue, 0 }  ,draw opacity=1 ]   (60,30) -- (60,40) ;
\draw [color={rgb, 255:red, 0; green, 0; blue, 0 }  ,draw opacity=1 ]   (70,30) -- (70,40) ;
\draw [color={rgb, 255:red, 0; green, 0; blue, 0 }  ,draw opacity=1 ]   (80,30) -- (80,40) ;
\draw [color={rgb, 255:red, 0; green, 0; blue, 0 }  ,draw opacity=1 ]   (490,30) -- (490,40) ;
\draw [color={rgb, 255:red, 0; green, 0; blue, 0 }  ,draw opacity=1 ]   (570,30) -- (570,40) ;
\draw [color={rgb, 255:red, 0; green, 0; blue, 0 }  ,draw opacity=1 ]   (370,30) -- (370,40) ;
\draw [color={rgb, 255:red, 0; green, 0; blue, 0 }  ,draw opacity=1 ]   (410,30) -- (410,40) ;
\draw [color={rgb, 255:red, 0; green, 0; blue, 0 }  ,draw opacity=1 ]   (400,30) -- (400,40) ;
\draw [color={rgb, 255:red, 0; green, 0; blue, 0 }  ,draw opacity=1 ]   (390,30) -- (390,40) ;
\draw [color={rgb, 255:red, 0; green, 0; blue, 0 }  ,draw opacity=1 ]   (380,30) -- (380,40) ;
\draw [color={rgb, 255:red, 0; green, 0; blue, 0 }  ,draw opacity=1 ]   (360,30) -- (360,40) ;
\draw [color={rgb, 255:red, 0; green, 0; blue, 0 }  ,draw opacity=1 ]   (440,30) -- (440,40) ;
\draw [color={rgb, 255:red, 0; green, 0; blue, 0 }  ,draw opacity=1 ]   (480,30) -- (480,40) ;
\draw [color={rgb, 255:red, 0; green, 0; blue, 0 }  ,draw opacity=1 ]   (470,30) -- (470,40) ;
\draw [color={rgb, 255:red, 0; green, 0; blue, 0 }  ,draw opacity=1 ]   (460,30) -- (460,40) ;
\draw [color={rgb, 255:red, 0; green, 0; blue, 0 }  ,draw opacity=1 ]   (450,30) -- (450,40) ;
\draw [color={rgb, 255:red, 0; green, 0; blue, 0 }  ,draw opacity=1 ]   (430,30) -- (430,40) ;
\draw [color={rgb, 255:red, 0; green, 0; blue, 0 }  ,draw opacity=1 ]   (520,30) -- (520,40) ;
\draw [color={rgb, 255:red, 0; green, 0; blue, 0 }  ,draw opacity=1 ]   (560,30) -- (560,40) ;
\draw [color={rgb, 255:red, 0; green, 0; blue, 0 }  ,draw opacity=1 ]   (550,30) -- (550,40) ;
\draw [color={rgb, 255:red, 0; green, 0; blue, 0 }  ,draw opacity=1 ]   (540,30) -- (540,40) ;
\draw [color={rgb, 255:red, 0; green, 0; blue, 0 }  ,draw opacity=1 ]   (530,30) -- (530,40) ;
\draw [color={rgb, 255:red, 0; green, 0; blue, 0 }  ,draw opacity=1 ]   (510,30) -- (510,40) ;

\draw (286.4,43.8) node [anchor=north west][inner sep=0.75pt]  [font=\footnotesize]  {$\mathrm{I}_{3}( x)$};
\draw (346,44.4) node [anchor=north west][inner sep=0.75pt]  [font=\footnotesize]  {$\mathrm{I}_{3}( x\ +\ 16)$};
\draw (426,44.4) node [anchor=north west][inner sep=0.75pt]  [font=\footnotesize]  {$\mathrm{I}_{3}( x\ +\ 32)$};
\draw (187,44.4) node [anchor=north west][inner sep=0.75pt]  [font=\footnotesize]  {$\mathrm{I}_{3}( x\ -\ 16)$};
\draw (107,44.4) node [anchor=north west][inner sep=0.75pt]  [font=\footnotesize]  {$\mathrm{I}_{3}( x\ -\ 32)$};

\end{tikzpicture}

%% file: Fig_balancing_procedure.tex
\tikzset{every picture/.style={line width=0.75pt}} 

\begin{tikzpicture}[x=0.75pt,y=0.75pt,yscale=-1,xscale=1]

\draw [color={rgb, 255:red, 255; green, 10; blue, 40 }  ,draw opacity=1 ][line width=4.5]    (10,34.64) -- (560.9,34.64) ;
\draw [color={rgb, 255:red, 74; green, 144; blue, 226 }  ,draw opacity=1 ][line width=4.5]    (156.64,34.55) -- (164.24,34.55) ;
\draw [color={rgb, 255:red, 74; green, 144; blue, 226 }  ,draw opacity=1 ][line width=4.5]    (232,34.64) -- (347.12,34.64) ;
\draw [color={rgb, 255:red, 74; green, 144; blue, 226 }  ,draw opacity=1 ][line width=4.5]    (365.62,34.64) -- (466.4,34.47) ;
\draw [color={rgb, 255:red, 74; green, 144; blue, 226 }  ,draw opacity=1 ][line width=4.5]    (34.67,34.64) -- (59.33,34.64) ;
\draw [color={rgb, 255:red, 255; green, 10; blue, 40 }  ,draw opacity=1 ][line width=4.5]    (10,116.15) -- (560.9,116.15) ;
\draw [color={rgb, 255:red, 74; green, 144; blue, 226 }  ,draw opacity=1 ][line width=4.5]    (166.23,122.86) ;
\draw [color={rgb, 255:red, 74; green, 144; blue, 226 }  ,draw opacity=1 ][line width=4.5]    (182.67,122.86) ;
\draw [color={rgb, 255:red, 74; green, 144; blue, 226 }  ,draw opacity=1 ][line width=4.5]    (125.11,122.86) ;
\draw [color={rgb, 255:red, 255; green, 10; blue, 40 }  ,draw opacity=1 ][line width=4.5]    (10,157.53) -- (560.9,157.53) ;
\draw [color={rgb, 255:red, 74; green, 144; blue, 226 }  ,draw opacity=1 ][line width=4.5]    (168.28,168.35) ;
\draw [color={rgb, 255:red, 74; green, 144; blue, 226 }  ,draw opacity=1 ][line width=4.5]    (184.73,168.35) ;
\draw [color={rgb, 255:red, 74; green, 144; blue, 226 }  ,draw opacity=1 ][line width=4.5]    (127.17,168.35) ;
\draw [color={rgb, 255:red, 74; green, 144; blue, 226 }  ,draw opacity=1 ][line width=4.5]    (363.62,116.13) -- (467.4,116.47) ;
\draw [color={rgb, 255:red, 74; green, 144; blue, 226 }  ,draw opacity=1 ][line width=4.5]    (233.56,157.53) -- (466.9,157.47) ;
\draw [color={rgb, 255:red, 74; green, 144; blue, 226 }  ,draw opacity=1 ][line width=4.5]    (34.67,116.15) -- (59.33,116.15) ;
\draw [color={rgb, 255:red, 74; green, 144; blue, 226 }  ,draw opacity=1 ][line width=4.5]    (34.67,157.53) -- (59.33,157.53) ;
\draw [color={rgb, 255:red, 74; green, 144; blue, 226 }  ,draw opacity=1 ][line width=4.5]    (232,116.15) -- (347.12,116.15) ;
\draw    (156.4,25.61) -- (156.4,42.06) ;
\draw    (164.62,25.61) -- (164.62,42.06) ;
\draw [color={rgb, 255:red, 155; green, 155; blue, 155 }  ,draw opacity=1 ]   (172.85,25.61) -- (172.85,42.06) ;
\draw [color={rgb, 255:red, 155; green, 155; blue, 155 }  ,draw opacity=1 ]   (148.18,25.61) -- (148.18,42.06) ;
\draw [color={rgb, 255:red, 155; green, 155; blue, 155 }  ,draw opacity=1 ]   (181.07,25.61) -- (181.07,42.06) ;
\draw [color={rgb, 255:red, 155; green, 155; blue, 155 }  ,draw opacity=1 ]   (139.96,25.61) -- (139.96,42.06) ;
\draw    (363.56,107.51) -- (363.56,123.96) ;
\draw [color={rgb, 255:red, 155; green, 155; blue, 155 }  ,draw opacity=1 ]   (380.01,107.51) -- (380.01,123.96) ;
\draw [color={rgb, 255:red, 155; green, 155; blue, 155 }  ,draw opacity=1 ]   (330.67,107.51) -- (330.67,123.96) ;
\draw [color={rgb, 255:red, 155; green, 155; blue, 155 }  ,draw opacity=1 ]   (314.23,107.51) -- (314.23,123.96) ;
\draw [color={rgb, 255:red, 155; green, 155; blue, 155 }  ,draw opacity=1 ]   (297.78,107.51) -- (297.78,123.96) ;
\draw [color={rgb, 255:red, 155; green, 155; blue, 155 }  ,draw opacity=1 ]   (281.34,107.51) -- (281.34,123.96) ;
\draw [color={rgb, 255:red, 155; green, 155; blue, 155 }  ,draw opacity=1 ]   (462.23,107.51) -- (462.23,123.96) ;
\draw [color={rgb, 255:red, 155; green, 155; blue, 155 }  ,draw opacity=1 ]   (445.79,107.51) -- (445.79,123.96) ;
\draw [color={rgb, 255:red, 155; green, 155; blue, 155 }  ,draw opacity=1 ]   (429.34,107.51) -- (429.34,123.96) ;
\draw [color={rgb, 255:red, 155; green, 155; blue, 155 }  ,draw opacity=1 ]   (412.9,107.51) -- (412.9,123.96) ;
\draw [color={rgb, 255:red, 155; green, 155; blue, 155 }  ,draw opacity=1 ]   (396.45,107.51) -- (396.45,123.96) ;
\draw    (347.12,107.51) -- (347.12,123.96) ;
\draw    (26.44,149.3) -- (26.44,165.75) ;
\draw [color={rgb, 255:red, 74; green, 144; blue, 226 }  ,draw opacity=1 ][line width=4.5]    (299.84,168.35) ;
\draw [color={rgb, 255:red, 74; green, 144; blue, 226 }  ,draw opacity=1 ][line width=4.5]    (258.73,168.35) ;
\draw [color={rgb, 255:red, 74; green, 144; blue, 226 }  ,draw opacity=1 ][line width=4.5]    (431.4,165.75) ;
\draw [color={rgb, 255:red, 74; green, 144; blue, 226 }  ,draw opacity=1 ][line width=4.5]    (447.84,165.75) ;
\draw [color={rgb, 255:red, 74; green, 144; blue, 226 }  ,draw opacity=1 ][line width=4.5]    (390.29,165.75) ;
\draw [color={rgb, 255:red, 74; green, 144; blue, 226 }  ,draw opacity=1 ][line width=4.5]    (692.5,211.33) ;
\draw [color={rgb, 255:red, 74; green, 144; blue, 226 }  ,draw opacity=1 ][line width=4.5]    (521.84,165.75) ;
\draw    (552.68,149.3) -- (552.68,165.75) ;
\draw [color={rgb, 255:red, 155; green, 155; blue, 155 }  ,draw opacity=1 ]   (32.8,149.3) -- (32.8,165.75) ;
\draw [color={rgb, 255:red, 255; green, 10; blue, 40 }  ,draw opacity=1 ][line width=4.5]    (10,75.49) -- (560.9,75.49) ;
\draw [color={rgb, 255:red, 74; green, 144; blue, 226 }  ,draw opacity=1 ][line width=4.5]    (232,75.49) -- (347.12,75.49) ;
\draw [color={rgb, 255:red, 74; green, 144; blue, 226 }  ,draw opacity=1 ][line width=4.5]    (365.62,75.49) -- (467.4,75.97) ;
\draw [color={rgb, 255:red, 74; green, 144; blue, 226 }  ,draw opacity=1 ][line width=4.5]    (34.67,75.49) -- (59.33,75.49) ;
\draw    (278.23,67.27) -- (278.23,83.71) ;
\draw    (286.45,67.27) -- (286.45,83.71) ;
\draw [color={rgb, 255:red, 155; green, 155; blue, 155 }  ,draw opacity=1 ]   (294.68,67.27) -- (294.68,83.71) ;
\draw [color={rgb, 255:red, 155; green, 155; blue, 155 }  ,draw opacity=1 ]   (270.01,67.27) -- (270.01,83.71) ;
\draw [color={rgb, 255:red, 155; green, 155; blue, 155 }  ,draw opacity=1 ]   (302.9,67.27) -- (302.9,83.71) ;
\draw [color={rgb, 255:red, 155; green, 155; blue, 155 }  ,draw opacity=1 ]   (261.79,67.27) -- (261.79,83.71) ;
\draw [color={rgb, 255:red, 255; green, 10; blue, 40 }  ,draw opacity=1 ][line width=4.5]    (279.51,75.4) -- (285.45,75.49) ;
\draw [color={rgb, 255:red, 255; green, 10; blue, 40 }  ,draw opacity=1 ][line width=4.5]    (279.51,34.55) -- (285.45,34.64) ;
\draw [color={rgb, 255:red, 155; green, 155; blue, 155 }  ,draw opacity=1 ]   (264.23,107.51) -- (264.23,123.96) ;
\draw [color={rgb, 255:red, 155; green, 155; blue, 155 }  ,draw opacity=1 ]   (247.78,107.51) -- (247.78,123.96) ;
\draw [color={rgb, 255:red, 155; green, 155; blue, 155 }  ,draw opacity=1 ]   (189.29,25.61) -- (189.29,42.06) ;
\draw [color={rgb, 255:red, 155; green, 155; blue, 155 }  ,draw opacity=1 ]   (197.51,25.61) -- (197.51,42.06) ;
\draw [color={rgb, 255:red, 155; green, 155; blue, 155 }  ,draw opacity=1 ]   (131.73,25.61) -- (131.73,42.06) ;
\draw [color={rgb, 255:red, 155; green, 155; blue, 155 }  ,draw opacity=1 ]   (123.51,25.61) -- (123.51,42.06) ;
\draw [color={rgb, 255:red, 155; green, 155; blue, 155 }  ,draw opacity=1 ]   (311.12,67.27) -- (311.12,83.71) ;
\draw [color={rgb, 255:red, 155; green, 155; blue, 155 }  ,draw opacity=1 ]   (319.34,67.27) -- (319.34,83.71) ;
\draw [color={rgb, 255:red, 155; green, 155; blue, 155 }  ,draw opacity=1 ]   (253.57,67.27) -- (253.57,83.71) ;
\draw [color={rgb, 255:red, 155; green, 155; blue, 155 }  ,draw opacity=1 ]   (245.34,67.27) -- (245.34,83.71) ;
\draw [color={rgb, 255:red, 155; green, 155; blue, 155 }  ,draw opacity=1 ]   (62.8,149.3) -- (62.8,165.75) ;

\draw (153.92,3.55) node [anchor=north west][inner sep=0.75pt]    {$\mathrm{I}_{3}$};
\draw (241.73,134.56) node [anchor=north west][inner sep=0.75pt]    {$\mathrm{I}_{9}$};
\draw (348.62,91.79) node [anchor=north west][inner sep=0.75pt]    {$\mathrm{I}_{4}$};
\draw (530.8,13.98) node [anchor=north west][inner sep=0.75pt]    {$\sigma ^{0}$};
\draw (531.89,95.62) node [anchor=north west][inner sep=0.75pt]    {$\sigma ^{2}$};
\draw (530.39,137.09) node [anchor=north west][inner sep=0.75pt]    {$\sigma ^{3}$};
\draw (41.2,165.91) node [anchor=north west][inner sep=0.75pt]  [color={rgb, 255:red, 128; green, 128; blue, 128 }  ,opacity=1 ]  {$\mathrm{I}_{5}$};
\draw (275.68,47.97) node [anchor=north west][inner sep=0.75pt]    {$\mathrm{I}_{3}^{\prime }$};
\draw (530.39,54.95) node [anchor=north west][inner sep=0.75pt]    {$\sigma ^{1}$};
\draw (233.31,37.34) node [anchor=north west][inner sep=0.75pt]  [font=\scriptsize]  {$0$};
\draw (233.31,78.19) node [anchor=north west][inner sep=0.75pt]  [font=\scriptsize]  {$0$};
\draw (235.36,160.23) node [anchor=north west][inner sep=0.75pt]  [font=\scriptsize]  {$0$};
\draw (233.31,118.85) node [anchor=north west][inner sep=0.75pt]  [font=\scriptsize]  {$0$};

\end{tikzpicture}

%% file: Fig_I_prime.tex
\tikzset{every picture/.style={line width=0.75pt}} 

\begin{tikzpicture}[x=0.75pt,y=0.75pt,yscale=-1,xscale=1]

\draw [line width=0.75]    (133,96) -- (387.83,96.92) (149.01,92.06) -- (148.99,100.06)(165.01,92.12) -- (164.99,100.12)(181.01,92.17) -- (180.99,100.17)(197.01,92.23) -- (196.99,100.23)(213.01,92.29) -- (212.99,100.29)(229.01,92.35) -- (228.98,100.35)(245.01,92.4) -- (244.98,100.4)(261.01,92.46) -- (260.98,100.46)(277.01,92.52) -- (276.98,100.52)(293.01,92.58) -- (292.98,100.58)(309.01,92.63) -- (308.98,100.63)(325.01,92.69) -- (324.98,100.69)(341.01,92.75) -- (340.98,100.75)(357.01,92.81) -- (356.98,100.81)(373.01,92.87) -- (372.98,100.87) ;
\draw [shift={(387.83,96.92)}, rotate = 180.21] [color={rgb, 255:red, 0; green, 0; blue, 0 }  ][line width=0.75]    (0,5.59) -- (0,-5.59)   ;
\draw [shift={(133,96)}, rotate = 180.21] [color={rgb, 255:red, 0; green, 0; blue, 0 }  ][line width=0.75]    (0,5.59) -- (0,-5.59)   ;
\draw  [color={rgb, 255:red, 0; green, 0; blue, 0 }  ,draw opacity=1 ] (143,64) -- (405,64) -- (405,128) -- (143,128) -- cycle ;
\draw [color={rgb, 255:red, 208; green, 2; blue, 27 }  ,draw opacity=1 ][line width=0.75]    (83.87,96) -- (133,96) ;
\draw [shift={(133,96)}, rotate = 180] [color={rgb, 255:red, 208; green, 2; blue, 27 }  ,draw opacity=1 ][line width=0.75]    (0,5.59) -- (0,-5.59)   ;
\draw [shift={(83.87,96)}, rotate = 180] [color={rgb, 255:red, 208; green, 2; blue, 27 }  ,draw opacity=1 ][line width=0.75]    (0,5.59) -- (0,-5.59)   ;
\draw  [color={rgb, 255:red, 155; green, 155; blue, 155 }  ,draw opacity=0.45 ] (10,64) -- (538,64) -- (538,128) -- (10,128) -- cycle ;
\draw [color={rgb, 255:red, 0; green, 0; blue, 0 }  ,draw opacity=1 ] [dash pattern={on 0.84pt off 2.51pt}]  (274,64) -- (274,128) ;

\draw (321,108.4) node [anchor=north west][inner sep=0.75pt]    {$\mathrm{C}_{i}$};
\draw (269.5,35.9) node [anchor=north west][inner sep=0.75pt]  [color={rgb, 255:red, 128; green, 128; blue, 128 }  ,opacity=1 ]  {$\mathrm{I}^{\prime }$};
\draw (101.17,103.89) node [anchor=north west][inner sep=0.75pt]  [color={rgb, 255:red, 208; green, 2; blue, 27 }  ,opacity=1 ]  {$\B_{i}$};
\draw (474.5,135.89) node [anchor=north west][inner sep=0.75pt]  [color={rgb, 255:red, 128; green, 128; blue, 128 }  ,opacity=1 ]  {$\rho _{\frac{3}{2}}\left(\mathrm{I}^{\prime }\right)$};

\end{tikzpicture}

%% file: Fig_Psi.tex
\tikzset{every picture/.style={line width=0.75pt}} 

\begin{tikzpicture}[x=0.75pt,y=0.75pt,yscale=-1,xscale=1]

\draw [color={rgb, 255:red, 0; green, 0; blue, 0 }  ,draw opacity=1 ][fill={rgb, 255:red, 0; green, 0; blue, 0 }  ,fill opacity=1 ][line width=6]    (90.34,188.03) -- (534.47,188.03) ;
\draw [color={rgb, 255:red, 255; green, 10; blue, 40 }  ,draw opacity=1 ][line width=6]    (407.58,188.03) -- (423.44,188.03) ;
\draw [color={rgb, 255:red, 255; green, 10; blue, 40 }  ,draw opacity=1 ][line width=6]    (359.99,188.03) -- (375.85,188.03) ;
\draw [color={rgb, 255:red, 255; green, 10; blue, 40 }  ,draw opacity=1 ][line width=6]    (336.2,188.03) -- (352.06,188.03) ;
\draw [color={rgb, 255:red, 255; green, 10; blue, 40 }  ,draw opacity=1 ][line width=6]    (272.75,188.03) -- (280.68,188.03) ;
\draw [color={rgb, 255:red, 255; green, 10; blue, 40 }  ,draw opacity=1 ][line width=6]    (288.61,188.03) -- (304.48,188.03) ;
\draw [color={rgb, 255:red, 255; green, 10; blue, 40 }  ,draw opacity=1 ][line width=6]    (320.34,188.03) -- (328.27,188.03) ;
\draw [color={rgb, 255:red, 255; green, 10; blue, 40 }  ,draw opacity=1 ][line width=6]    (114.14,188.03) -- (137.93,188.03) ;
\draw [color={rgb, 255:red, 255; green, 10; blue, 40 }  ,draw opacity=1 ][line width=6]    (177.58,188.03) -- (209.31,188.03) ;
\draw [color={rgb, 255:red, 255; green, 10; blue, 40 }  ,draw opacity=1 ][line width=6]    (217.24,188.03) -- (241.03,188.03) ;
\draw [color={rgb, 255:red, 255; green, 10; blue, 40 }  ,draw opacity=1 ][line width=6]    (439.3,188.03) -- (455.16,188.03) ;
\draw [color={rgb, 255:red, 255; green, 10; blue, 40 }  ,draw opacity=1 ][line width=6]    (463.09,188.03) -- (502.75,188.03) ;
\draw [color={rgb, 255:red, 255; green, 10; blue, 40 }  ,draw opacity=1 ][line width=6]    (518.61,188.03) -- (534.47,188.03) ;
\draw [color={rgb, 255:red, 0; green, 0; blue, 0 }  ,draw opacity=1 ][fill={rgb, 255:red, 0; green, 0; blue, 0 }  ,fill opacity=1 ][line width=6]    (90.34,160.28) -- (534.47,160.28) ;
\draw [color={rgb, 255:red, 0; green, 0; blue, 0 }  ,draw opacity=1 ][fill={rgb, 255:red, 0; green, 0; blue, 0 }  ,fill opacity=1 ][line width=6]    (90.34,132.52) -- (534.47,132.52) ;
\draw [color={rgb, 255:red, 255; green, 255; blue, 255 }  ,draw opacity=1 ][line width=6]    (455.16,188.03) -- (460.61,188.03) -- (463.09,188.03) ;
\draw [color={rgb, 255:red, 255; green, 255; blue, 255 }  ,draw opacity=1 ][line width=6]    (431.37,188.03) -- (436.82,188.03) -- (439.3,188.03) ;
\draw [color={rgb, 255:red, 255; green, 255; blue, 255 }  ,draw opacity=1 ][line width=6]    (423.44,188.03) -- (428.89,188.03) -- (431.37,188.03) ;
\draw [color={rgb, 255:red, 255; green, 255; blue, 255 }  ,draw opacity=1 ][line width=6]    (352.06,188.03) -- (357.51,188.03) -- (359.99,188.03) ;
\draw [color={rgb, 255:red, 255; green, 255; blue, 255 }  ,draw opacity=1 ][line width=6]    (328.27,188.03) -- (333.71,188.03) -- (336.2,188.03) ;
\draw [color={rgb, 255:red, 74; green, 144; blue, 226 }  ,draw opacity=1 ][line width=6]    (502.75,188.03) -- (518.61,188.03) ;
\draw [color={rgb, 255:red, 74; green, 144; blue, 226 }  ,draw opacity=1 ][line width=6]    (399.65,188.03) -- (407.58,188.03) ;
\draw [color={rgb, 255:red, 74; green, 144; blue, 226 }  ,draw opacity=1 ][line width=6]    (391.72,188.03) -- (399.65,188.03) ;
\draw [color={rgb, 255:red, 74; green, 144; blue, 226 }  ,draw opacity=1 ][line width=6]    (383.78,188.03) -- (391.72,188.03) ;
\draw [color={rgb, 255:red, 74; green, 144; blue, 226 }  ,draw opacity=1 ][line width=6]    (375.85,188.03) -- (383.78,188.03) ;
\draw [color={rgb, 255:red, 74; green, 144; blue, 226 }  ,draw opacity=1 ][line width=6]    (280.68,188.03) -- (288.61,188.03) ;
\draw [color={rgb, 255:red, 255; green, 255; blue, 255 }  ,draw opacity=1 ][line width=6]    (153.79,188.03) -- (159.24,188.03) -- (161.72,188.03) ;
\draw [color={rgb, 255:red, 255; green, 255; blue, 255 }  ,draw opacity=1 ][line width=6]    (137.93,188.03) -- (145.86,188.03) ;
\draw [color={rgb, 255:red, 255; green, 255; blue, 255 }  ,draw opacity=1 ][line width=6]    (209.31,188.03) -- (214.75,188.03) -- (217.24,188.03) ;
\draw [color={rgb, 255:red, 255; green, 255; blue, 255 }  ,draw opacity=1 ][line width=6]    (264.82,188.03) -- (270.27,188.03) -- (272.75,188.03) ;
\draw [color={rgb, 255:red, 74; green, 144; blue, 226 }  ,draw opacity=1 ][line width=6]    (256.89,188.03) -- (264.82,188.03) ;
\draw [color={rgb, 255:red, 74; green, 144; blue, 226 }  ,draw opacity=1 ][line width=6]    (248.96,188.03) -- (256.89,188.03) ;
\draw [color={rgb, 255:red, 74; green, 144; blue, 226 }  ,draw opacity=1 ][line width=6]    (241.03,188.03) -- (248.96,188.03) ;
\draw [color={rgb, 255:red, 255; green, 255; blue, 255 }  ,draw opacity=1 ][line width=6]    (169.65,188.03) -- (175.1,188.03) -- (177.58,188.03) ;
\draw [color={rgb, 255:red, 74; green, 144; blue, 226 }  ,draw opacity=1 ][line width=6]    (106.2,188.03) -- (114.14,188.03) ;
\draw [color={rgb, 255:red, 74; green, 144; blue, 226 }  ,draw opacity=1 ][line width=6]    (98.27,188.03) -- (106.2,188.03) ;
\draw [color={rgb, 255:red, 74; green, 144; blue, 226 }  ,draw opacity=1 ][line width=6]    (90.34,188.03) -- (98.27,188.03) ;
\draw    (288.61,184.07) -- (288.61,192) ;
\draw    (415.51,184.07) -- (415.51,192) ;
\draw    (98.27,184.07) -- (98.27,192) ;
\draw    (225.17,184.07) -- (225.17,192) ;
\draw    (352.06,184.07) -- (352.06,192) ;
\draw    (478.96,184.07) -- (478.96,192) ;
\draw    (296.55,184.07) -- (296.55,192) ;
\draw    (423.44,184.07) -- (423.44,192) ;
\draw    (106.2,184.07) -- (106.2,192) ;
\draw    (233.1,184.07) -- (233.1,192) ;
\draw    (359.99,184.07) -- (359.99,192) ;
\draw    (486.89,184.07) -- (486.89,192) ;
\draw    (153.79,184.07) -- (153.79,192) ;
\draw    (280.68,184.07) -- (280.68,192) ;
\draw    (407.58,184.07) -- (407.58,192) ;
\draw    (534.47,184.07) -- (534.47,192) ;
\draw    (90.34,184.07) -- (90.34,192) ;
\draw    (217.24,184.07) -- (217.24,192) ;
\draw    (344.13,184.07) -- (344.13,192) ;
\draw    (471.02,184.07) -- (471.02,192) ;
\draw    (114.14,184.07) -- (114.14,192) ;
\draw    (122.07,184.07) -- (122.07,192) ;
\draw    (137.93,184.07) -- (137.93,192) ;
\draw    (130,184.07) -- (130,192) ;
\draw    (177.58,184.07) -- (177.58,192) ;
\draw    (185.51,184.07) -- (185.51,192) ;
\draw    (201.38,184.07) -- (201.38,192) ;
\draw    (209.31,184.07) -- (209.31,192) ;
\draw    (193.44,184.07) -- (193.44,192) ;
\draw    (217.24,184.07) -- (217.24,192) ;
\draw    (241.03,184.07) -- (241.03,192) ;
\draw    (248.96,184.07) -- (248.96,192) ;
\draw    (264.82,184.07) -- (264.82,192) ;
\draw    (272.75,184.07) -- (272.75,192) ;
\draw    (256.89,184.07) -- (256.89,192) ;
\draw    (280.68,184.07) -- (280.68,192) ;
\draw    (328.27,184.07) -- (328.27,192) ;
\draw    (336.2,184.07) -- (336.2,192) ;
\draw    (344.13,184.07) -- (344.13,192) ;
\draw    (367.92,184.07) -- (367.92,192) ;
\draw    (375.85,184.07) -- (375.85,192) ;
\draw    (391.72,184.07) -- (391.72,192) ;
\draw    (399.65,184.07) -- (399.65,192) ;
\draw    (383.78,184.07) -- (383.78,192) ;
\draw    (407.58,184.07) -- (407.58,192) ;
\draw    (431.37,184.07) -- (431.37,192) ;
\draw    (439.3,184.07) -- (439.3,192) ;
\draw    (455.16,184.07) -- (455.16,192) ;
\draw    (463.09,184.07) -- (463.09,192) ;
\draw    (447.23,184.07) -- (447.23,192) ;
\draw    (471.02,184.07) -- (471.02,192) ;
\draw    (494.82,184.07) -- (494.82,192) ;
\draw    (502.75,184.07) -- (502.75,192) ;
\draw    (518.61,184.07) -- (518.61,192) ;
\draw    (526.54,184.07) -- (526.54,192) ;
\draw    (510.68,184.07) -- (510.68,192) ;
\draw    (534.47,184.07) -- (534.47,192) ;
\draw [color={rgb, 255:red, 255; green, 255; blue, 255 }  ,draw opacity=1 ][line width=6]    (455.16,160.28) -- (471.02,160.28) ;
\draw [color={rgb, 255:red, 255; green, 255; blue, 255 }  ,draw opacity=1 ][line width=6]    (423.44,160.28) -- (439.3,160.28) ;
\draw [color={rgb, 255:red, 255; green, 255; blue, 255 }  ,draw opacity=1 ][line width=4.5]    (352.06,160.28) -- (357.51,160.28) -- (359.99,160.28) ;
\draw [color={rgb, 255:red, 255; green, 255; blue, 255 }  ,draw opacity=1 ][line width=6]    (328.27,160.28) -- (333.71,160.28) -- (359.99,160.28) ;
\draw [color={rgb, 255:red, 74; green, 144; blue, 226 }  ,draw opacity=1 ][line width=6]    (502.75,160.28) -- (518.61,160.28) ;
\draw [color={rgb, 255:red, 74; green, 144; blue, 226 }  ,draw opacity=1 ][line width=6]    (391.72,160.28) -- (407.58,160.28) ;
\draw [color={rgb, 255:red, 74; green, 144; blue, 226 }  ,draw opacity=1 ][line width=6]    (375.85,160.28) -- (391.72,160.28) ;
\draw [color={rgb, 255:red, 74; green, 144; blue, 226 }  ,draw opacity=1 ][line width=4.5]    (280.68,160.28) -- (288.61,160.28) ;
\draw [color={rgb, 255:red, 255; green, 255; blue, 255 }  ,draw opacity=1 ][line width=6]    (137.93,160.28) -- (185.51,160.28) ;
\draw [color={rgb, 255:red, 255; green, 255; blue, 255 }  ,draw opacity=1 ][line width=6]    (106.2,160.28) -- (111.65,160.28) -- (122.07,160.28) ;
\draw [color={rgb, 255:red, 255; green, 255; blue, 255 }  ,draw opacity=1 ][line width=6]    (201.38,160.28) -- (214.75,160.28) -- (217.24,160.28) ;
\draw [color={rgb, 255:red, 255; green, 255; blue, 255 }  ,draw opacity=1 ][line width=6]    (264.82,160.28) -- (270.27,160.28) -- (328.27,160.28) ;
\draw [color={rgb, 255:red, 74; green, 144; blue, 226 }  ,draw opacity=1 ][line width=6]    (248.96,160.28) -- (264.82,160.28) ;
\draw [color={rgb, 255:red, 74; green, 144; blue, 226 }  ,draw opacity=1 ][line width=4.5]    (241.03,160.28) -- (248.96,160.28) ;
\draw [color={rgb, 255:red, 74; green, 144; blue, 226 }  ,draw opacity=1 ][line width=6]    (90.34,160.28) -- (106.2,160.28) ;
\draw    (90.34,156.31) -- (90.34,164.24) ;
\draw    (169.65,156.31) -- (169.65,164.24) ;
\draw    (296.55,156.31) -- (296.55,164.24) ;
\draw    (153.79,156.31) -- (153.79,164.24) ;
\draw    (280.68,156.31) -- (280.68,164.24) ;
\draw    (407.58,156.31) -- (407.58,164.24) ;
\draw    (534.47,156.31) -- (534.47,164.24) ;
\draw    (217.24,156.31) -- (217.24,164.24) ;
\draw    (344.13,156.31) -- (344.13,164.24) ;
\draw    (471.02,156.31) -- (471.02,164.24) ;
\draw    (106.2,156.31) -- (106.2,164.24) ;
\draw    (153.79,156.31) -- (153.79,164.24) ;
\draw    (264.82,156.31) -- (264.82,164.24) ;
\draw    (280.68,156.31) -- (280.68,164.24) ;
\draw    (312.41,156.31) -- (312.41,164.24) ;
\draw    (328.27,156.31) -- (328.27,164.24) ;
\draw    (344.13,156.31) -- (344.13,164.24) ;
\draw    (391.72,156.31) -- (391.72,164.24) ;
\draw [color={rgb, 255:red, 255; green, 255; blue, 255 }  ,draw opacity=1 ][line width=6]    (233.1,160.28) -- (246.48,160.28) -- (248.96,160.28) ;
\draw [color={rgb, 255:red, 255; green, 10; blue, 40 }  ,draw opacity=1 ][line width=6]    (471.02,132.52) -- (502.75,132.52) ;
\draw [color={rgb, 255:red, 255; green, 255; blue, 255 }  ,draw opacity=1 ][line width=6]    (439.3,132.52) -- (460.61,132.52) -- (471.02,132.52) ;
\draw [color={rgb, 255:red, 255; green, 255; blue, 255 }  ,draw opacity=1 ][line width=6]    (407.58,132.52) -- (439.3,132.52) ;
\draw [color={rgb, 255:red, 255; green, 255; blue, 255 }  ,draw opacity=1 ][line width=4.5]    (352.06,132.52) -- (357.51,132.52) -- (359.99,132.52) ;
\draw [color={rgb, 255:red, 255; green, 255; blue, 255 }  ,draw opacity=1 ][line width=6]    (312.41,132.52) -- (333.71,132.52) -- (375.85,132.52) ;
\draw [color={rgb, 255:red, 74; green, 144; blue, 226 }  ,draw opacity=1 ][line width=4.5]    (502.75,132.52) -- (518.61,132.52) ;
\draw [color={rgb, 255:red, 74; green, 144; blue, 226 }  ,draw opacity=1 ][line width=6]    (375.85,132.52) -- (407.58,132.52) ;
\draw [color={rgb, 255:red, 74; green, 144; blue, 226 }  ,draw opacity=1 ][line width=4.5]    (280.68,132.52) -- (288.61,132.52) ;
\draw [color={rgb, 255:red, 255; green, 255; blue, 255 }  ,draw opacity=1 ][line width=6]    (122.07,132.52) -- (185.51,132.52) ;
\draw [color={rgb, 255:red, 255; green, 255; blue, 255 }  ,draw opacity=1 ][line width=6]    (90.34,132.52) -- (122.07,132.52) ;
\draw [color={rgb, 255:red, 255; green, 255; blue, 255 }  ,draw opacity=1 ][line width=6]    (185.51,132.52) -- (214.75,132.52) -- (217.24,132.52) ;
\draw [color={rgb, 255:red, 255; green, 255; blue, 255 }  ,draw opacity=1 ][line width=6]    (248.96,132.52) -- (312.41,132.52) ;
\draw [color={rgb, 255:red, 74; green, 144; blue, 226 }  ,draw opacity=1 ][line width=4.5]    (241.03,132.52) -- (248.96,132.52) ;
\draw    (90.34,128.55) -- (90.34,136.48) ;
\draw    (153.79,128.55) -- (153.79,136.48) ;
\draw    (280.68,128.55) -- (280.68,136.48) ;
\draw    (407.58,128.55) -- (407.58,136.48) ;
\draw    (534.47,128.55) -- (534.47,136.48) ;
\draw    (217.24,128.55) -- (217.24,136.48) ;
\draw    (344.13,128.55) -- (344.13,136.48) ;
\draw    (471.02,128.55) -- (471.02,136.48) ;
\draw    (122.07,128.55) -- (122.07,136.48) ;
\draw    (153.79,128.55) -- (153.79,136.48) ;
\draw    (185.51,128.55) -- (185.51,136.48) ;
\draw    (280.68,128.55) -- (280.68,136.48) ;
\draw    (312.41,128.55) -- (312.41,136.48) ;
\draw    (344.13,128.55) -- (344.13,136.48) ;
\draw    (375.85,128.55) -- (375.85,136.48) ;
\draw    (407.58,128.55) -- (407.58,136.48) ;
\draw    (439.3,128.55) -- (439.3,136.48) ;
\draw    (471.02,128.55) -- (471.02,136.48) ;
\draw [color={rgb, 255:red, 255; green, 255; blue, 255 }  ,draw opacity=1 ][line width=6]    (217.24,132.52) -- (246.48,132.52) -- (248.96,132.52) ;
\draw    (248.96,128.55) -- (248.96,136.48) ;
\draw    (217.24,128.55) -- (217.24,136.48) ;
\draw [color={rgb, 255:red, 255; green, 255; blue, 255 }  ,draw opacity=1 ][line width=6]    (502.75,132.52) -- (534.47,132.52) ;
\draw    (534.47,128.55) -- (534.47,136.48) ;
\draw    (502.75,128.55) -- (502.75,136.48) ;
\draw    (248.96,156.31) -- (248.96,164.24) ;
\draw [color={rgb, 255:red, 74; green, 144; blue, 226 }  ,draw opacity=1 ][line width=6]    (304.48,188.03) -- (320.34,188.03) ;
\draw    (312.41,184.07) -- (312.41,192) ;
\draw    (304.48,184.07) -- (304.48,192) ;
\draw    (320.34,184.07) -- (320.34,192) ;
\draw [color={rgb, 255:red, 255; green, 10; blue, 40 }  ,draw opacity=1 ][line width=6]    (90.34,211.83) -- (535,211.83) ;
\draw [color={rgb, 255:red, 74; green, 144; blue, 226 }  ,draw opacity=1 ][line width=6]    (375.85,211.83) -- (407.58,211.83) ;
\draw [color={rgb, 255:red, 74; green, 144; blue, 226 }  ,draw opacity=1 ][line width=6]    (510.15,211.83) -- (518.61,211.83) ;
\draw [color={rgb, 255:red, 74; green, 144; blue, 226 }  ,draw opacity=1 ][line width=6]    (502.75,211.83) -- (510.15,211.83) ;
\draw [color={rgb, 255:red, 74; green, 144; blue, 226 }  ,draw opacity=1 ][line width=6]    (455.16,211.83) -- (459.13,211.83) ;
\draw [color={rgb, 255:red, 74; green, 144; blue, 226 }  ,draw opacity=1 ][line width=6]    (435.34,211.83) -- (439.3,211.83) ;
\draw [color={rgb, 255:red, 74; green, 144; blue, 226 }  ,draw opacity=1 ][line width=6]    (423.44,211.83) -- (427.4,211.83) ;
\draw [color={rgb, 255:red, 74; green, 144; blue, 226 }  ,draw opacity=1 ][line width=6]    (352.13,211.88) -- (356.2,211.83) ;
\draw [color={rgb, 255:red, 74; green, 144; blue, 226 }  ,draw opacity=1 ][line width=6]    (327.74,211.83) -- (332.13,211.88) ;
\draw [color={rgb, 255:red, 74; green, 144; blue, 226 }  ,draw opacity=1 ][line width=6]    (256.36,211.83) -- (268.79,211.83) ;
\draw [color={rgb, 255:red, 74; green, 144; blue, 226 }  ,draw opacity=1 ][line width=6]    (248.43,211.83) -- (256.36,211.83) ;
\draw [color={rgb, 255:red, 74; green, 144; blue, 226 }  ,draw opacity=1 ][line width=6]    (241.03,211.83) -- (248.96,211.83) ;
\draw [color={rgb, 255:red, 74; green, 144; blue, 226 }  ,draw opacity=1 ][line width=6]    (280.68,211.83) -- (288.61,211.83) ;
\draw [color={rgb, 255:red, 74; green, 144; blue, 226 }  ,draw opacity=1 ][line width=6]    (209.31,211.83) -- (217.24,211.83) ;
\draw [color={rgb, 255:red, 74; green, 144; blue, 226 }  ,draw opacity=1 ][line width=6]    (173.62,211.83) -- (177.58,211.83) ;
\draw [color={rgb, 255:red, 74; green, 144; blue, 226 }  ,draw opacity=1 ][line width=6]    (157.76,211.83) -- (161.72,211.83) ;
\draw [color={rgb, 255:red, 74; green, 144; blue, 226 }  ,draw opacity=1 ][line width=6]    (141.89,211.83) -- (145.86,211.83) ;
\draw [color={rgb, 255:red, 74; green, 144; blue, 226 }  ,draw opacity=1 ][line width=6]    (90.34,211.83) -- (114.14,211.83) ;
\draw    (153.79,207.86) -- (153.79,215.79) ;
\draw    (280.68,207.86) -- (280.68,215.79) ;
\draw    (407.58,207.86) -- (407.58,215.79) ;
\draw    (534.47,207.86) -- (534.47,215.79) ;
\draw    (90.34,207.86) -- (90.34,215.79) ;
\draw    (217.24,207.86) -- (217.24,215.79) ;
\draw    (344.13,207.86) -- (344.13,215.79) ;
\draw    (471.02,207.86) -- (471.02,215.79) ;
\draw    (161.72,207.86) -- (161.72,215.79) ;
\draw    (288.61,207.86) -- (288.61,215.79) ;
\draw    (415.51,207.86) -- (415.51,215.79) ;
\draw    (98.27,207.86) -- (98.27,215.79) ;
\draw    (225.17,207.86) -- (225.17,215.79) ;
\draw    (348.1,207.86) -- (348.1,215.79) ;
\draw    (482.92,207.86) -- (482.92,215.79) ;
\draw    (145.86,207.86) -- (145.86,215.79) ;
\draw    (272.75,207.86) -- (272.75,215.79) ;
\draw    (399.65,207.86) -- (399.65,215.79) ;
\draw    (526.54,207.86) -- (526.54,215.79) ;
\draw    (209.31,207.86) -- (209.31,215.79) ;
\draw    (336.2,207.86) -- (336.2,215.79) ;
\draw    (463.09,207.86) -- (463.09,215.79) ;
\draw    (106.2,207.86) -- (106.2,215.79) ;
\draw    (114.14,207.86) -- (114.14,215.79) ;
\draw    (130,207.86) -- (130,215.79) ;
\draw    (137.93,207.86) -- (137.93,215.79) ;
\draw    (122.07,207.86) -- (122.07,215.79) ;
\draw    (169.12,207.86) -- (169.12,215.79) ;
\draw    (177.58,207.86) -- (177.58,215.79) ;
\draw    (193.44,207.86) -- (193.44,215.79) ;
\draw    (201.38,207.86) -- (201.38,215.79) ;
\draw    (185.51,207.86) -- (185.51,215.79) ;
\draw    (233.1,207.86) -- (233.1,215.79) ;
\draw    (241.03,207.86) -- (241.03,215.79) ;
\draw    (256.89,207.86) -- (256.89,215.79) ;
\draw    (264.82,207.86) -- (264.82,215.79) ;
\draw    (248.96,207.86) -- (248.96,215.79) ;
\draw    (296.55,207.86) -- (296.55,215.79) ;
\draw    (328.27,207.86) -- (328.27,215.79) ;
\draw    (359.99,207.86) -- (359.99,215.79) ;
\draw    (367.92,207.86) -- (367.92,215.79) ;
\draw    (383.78,207.86) -- (383.78,215.79) ;
\draw    (391.72,207.86) -- (391.72,215.79) ;
\draw    (375.85,207.86) -- (375.85,215.79) ;
\draw    (423.44,208) -- (423.44,211.83) -- (423.44,215.93) ;
\draw    (431.37,207.86) -- (431.37,215.79) ;
\draw    (447.23,207.86) -- (447.23,215.79) ;
\draw    (459.13,207.86) -- (459.13,215.79) ;
\draw    (439.3,207.86) -- (439.3,215.79) ;
\draw    (486.89,207.86) -- (486.89,215.79) ;
\draw    (494.82,207.86) -- (494.82,215.79) ;
\draw    (510.68,207.86) -- (510.68,215.79) ;
\draw    (518.61,207.86) -- (518.61,215.79) ;
\draw    (502.75,207.86) -- (502.75,215.79) ;
\draw    (157.76,207.86) -- (157.76,215.79) ;
\draw    (284.65,207.86) -- (284.65,215.79) ;
\draw    (411.54,207.86) -- (411.54,215.79) ;
\draw    (94.31,207.86) -- (94.31,215.79) ;
\draw    (221.2,207.86) -- (221.2,215.79) ;
\draw    (356.03,207.86) -- (356.03,215.79) ;
\draw    (474.99,207.86) -- (474.99,215.79) ;
\draw    (165.16,207.86) -- (165.16,215.79) ;
\draw    (292.58,207.86) -- (292.58,215.79) ;
\draw    (419.47,207.86) -- (419.47,215.79) ;
\draw    (102.24,207.86) -- (102.24,215.79) ;
\draw    (229.13,207.86) -- (229.13,215.79) ;
\draw    (352.06,207.86) -- (352.06,215.79) ;
\draw    (478.96,207.86) -- (478.96,215.79) ;
\draw    (149.82,207.86) -- (149.82,215.79) ;
\draw    (276.72,207.86) -- (276.72,215.79) ;
\draw    (403.61,207.86) -- (403.61,215.79) ;
\draw    (530.51,207.86) -- (530.51,215.79) ;
\draw    (213.27,207.86) -- (213.27,215.79) ;
\draw    (340.17,207.86) -- (340.17,215.79) ;
\draw    (467.06,207.86) -- (467.06,215.79) ;
\draw    (110.17,207.86) -- (110.17,215.79) ;
\draw    (118.1,207.86) -- (118.1,215.79) ;
\draw    (133.96,207.86) -- (133.96,215.79) ;
\draw    (141.89,207.86) -- (141.89,215.79) ;
\draw    (126.03,207.86) -- (126.03,215.79) ;
\draw    (173.62,207.86) -- (173.62,215.79) ;
\draw    (181.55,207.86) -- (181.55,215.79) ;
\draw    (197.41,207.86) -- (197.41,215.79) ;
\draw    (205.34,207.86) -- (205.34,215.79) ;
\draw    (189.48,207.86) -- (189.48,215.79) ;
\draw    (237.06,207.86) -- (237.06,215.79) ;
\draw    (244.99,207.86) -- (244.99,215.79) ;
\draw    (260.86,207.86) -- (260.86,215.79) ;
\draw    (268.79,207.86) -- (268.79,215.79) ;
\draw    (252.93,207.86) -- (252.93,215.79) ;
\draw    (300.51,207.86) -- (300.51,215.79) ;
\draw    (324.3,207.86) -- (324.3,215.79) ;
\draw    (332.23,207.86) -- (332.23,215.79) ;
\draw    (363.96,207.86) -- (363.96,215.79) ;
\draw    (371.89,207.86) -- (371.89,215.79) ;
\draw    (387.75,207.86) -- (387.75,215.79) ;
\draw    (395.68,207.86) -- (395.68,215.79) ;
\draw    (379.82,207.86) -- (379.82,215.79) ;
\draw    (427.4,207.86) -- (427.4,215.79) ;
\draw    (435.34,207.86) -- (435.34,215.79) ;
\draw    (451.2,207.86) -- (451.2,215.79) ;
\draw    (455.16,207.86) -- (455.16,215.79) ;
\draw    (443.27,207.86) -- (443.27,215.79) ;
\draw    (490.85,207.86) -- (490.85,215.79) ;
\draw    (498.78,207.86) -- (498.78,215.79) ;
\draw    (514.64,207.86) -- (514.64,215.79) ;
\draw    (522.57,207.86) -- (522.57,215.79) ;
\draw    (506.71,207.86) -- (506.71,215.79) ;
\draw [color={rgb, 255:red, 74; green, 144; blue, 226 }  ,draw opacity=1 ][line width=6]    (304.48,211.83) -- (320.34,211.83) ;
\draw    (312.41,207.86) -- (312.41,215.79) ;
\draw    (316.37,207.86) -- (316.37,215.79) ;
\draw    (308.44,207.86) -- (308.44,215.79) ;
\draw    (304.48,207.86) -- (304.48,215.79) ;
\draw    (320.34,207.86) -- (320.34,215.79) ;
\draw [color={rgb, 255:red, 255; green, 10; blue, 40 }  ,draw opacity=1 ][line width=6]    (122.07,160.28) -- (137.93,160.28) ;
\draw [color={rgb, 255:red, 255; green, 10; blue, 40 }  ,draw opacity=1 ][line width=6]    (185.51,160.28) -- (201.38,160.28) ;
\draw [color={rgb, 255:red, 255; green, 10; blue, 40 }  ,draw opacity=1 ][line width=6]    (217.24,160.28) -- (233.1,160.28) ;
\draw [color={rgb, 255:red, 255; green, 10; blue, 40 }  ,draw opacity=1 ][line width=6]    (359.99,160.28) -- (375.85,160.28) ;
\draw [color={rgb, 255:red, 255; green, 10; blue, 40 }  ,draw opacity=1 ][line width=6]    (407.58,160.28) -- (423.44,160.28) ;
\draw [color={rgb, 255:red, 255; green, 10; blue, 40 }  ,draw opacity=1 ][line width=6]    (439.3,160.28) -- (455.16,160.28) ;
\draw [color={rgb, 255:red, 255; green, 10; blue, 40 }  ,draw opacity=1 ][line width=6]    (471.02,160.28) -- (502.75,160.28) ;
\draw [color={rgb, 255:red, 255; green, 10; blue, 40 }  ,draw opacity=1 ][line width=6]    (518.61,160.28) -- (534.47,160.28) ;
\draw    (137.93,156.31) -- (137.93,164.24) ;
\draw    (122.07,156.31) -- (122.07,164.24) ;
\draw    (185.51,156.31) -- (185.51,164.24) ;
\draw    (201.38,156.31) -- (201.38,164.24) ;
\draw    (217.24,156.31) -- (217.24,164.24) ;
\draw    (233.1,156.31) -- (233.1,164.24) ;
\draw    (423.44,156.31) -- (423.44,164.24) ;
\draw    (359.99,156.31) -- (359.99,164.24) ;
\draw    (486.89,156.31) -- (486.89,164.24) ;
\draw    (375.85,156.31) -- (375.85,164.24) ;
\draw    (407.58,156.31) -- (407.58,164.24) ;
\draw    (439.3,156.31) -- (439.3,164.24) ;
\draw    (455.16,156.31) -- (455.16,164.24) ;
\draw    (471.02,156.31) -- (471.02,164.24) ;
\draw    (502.75,156.31) -- (502.75,164.24) ;
\draw    (518.61,156.31) -- (518.61,164.24) ;
\draw    (534.47,156.31) -- (534.47,164.24) ;
\draw [color={rgb, 255:red, 255; green, 10; blue, 40 }  ,draw opacity=1 ][line width=6]    (145.86,188.03) -- (153.79,188.03) ;
\draw [color={rgb, 255:red, 255; green, 10; blue, 40 }  ,draw opacity=1 ][line width=6]    (161.72,188.03) -- (169.65,188.03) ;
\draw    (145.86,184.07) -- (145.86,192) ;
\draw    (153.79,184.07) -- (153.79,192) ;
\draw    (161.72,184.07) -- (161.72,192) ;
\draw    (169.65,184.07) -- (169.65,192) ;

\draw (541,202.4) node [anchor=north west][inner sep=0.75pt]    {$\Psi _{0}( A) \ =\ \sigma $};
\draw (541,177.4) node [anchor=north west][inner sep=0.75pt]    {$\Psi _{1}( A)$};
\draw (541,150.4) node [anchor=north west][inner sep=0.75pt]    {$\Psi _{2}( A)$};
\draw (541,122.4) node [anchor=north west][inner sep=0.75pt]    {$\Psi _{3}( A)$};
\draw (311,222.4) node [anchor=north west][inner sep=0.75pt]    {$0$};

\end{tikzpicture}

%% file: Fig_Lemma_3.3.tex
\tikzset{every picture/.style={line width=0.75pt}} 

\begin{tikzpicture}[x=0.75pt,y=0.75pt,yscale=-1,xscale=1]

\draw    (245,25) -- (245,45) ;
\draw  [color={rgb, 255:red, 255; green, 255; blue, 255 }  ,draw opacity=1 ][fill={rgb, 255:red, 155; green, 155; blue, 155 }  ,fill opacity=1 ] (245,50) -- (470,50) -- (470,80) -- (245,80) -- cycle ;
\draw    (245,50) -- (245,80) ;
\draw    (470,50) -- (470,80) ;
\draw    (470,80) -- (455,80) ;
\draw    (470,50) -- (455,50) ;
\draw    (405,25) -- (405,45) ;
\draw    (245,35) -- (405,35) ;
\draw [shift={(405,35)}, rotate = 180] [color={rgb, 255:red, 0; green, 0; blue, 0 }  ][line width=0.75]    (0,5.59) -- (0,-5.59)   ;
\draw [shift={(245,35)}, rotate = 180] [color={rgb, 255:red, 0; green, 0; blue, 0 }  ][line width=0.75]    (0,5.59) -- (0,-5.59)   ;
\draw    (245,90) -- (395,90) ;
\draw [shift={(395,90)}, rotate = 180] [color={rgb, 255:red, 0; green, 0; blue, 0 }  ][line width=0.75]    (0,5.59) -- (0,-5.59)   ;
\draw [shift={(245,90)}, rotate = 180] [color={rgb, 255:red, 0; green, 0; blue, 0 }  ][line width=0.75]    (0,5.59) -- (0,-5.59)   ;
\draw    (85,65) -- (615,65) (105,61) -- (105,69)(125,61) -- (125,69)(145,61) -- (145,69)(165,61) -- (165,69)(185,61) -- (185,69)(205,61) -- (205,69)(225,61) -- (225,69)(245,61) -- (245,69)(265,61) -- (265,69)(285,61) -- (285,69)(305,61) -- (305,69)(325,61) -- (325,69)(345,61) -- (345,69)(365,61) -- (365,69)(385,61) -- (385,69)(405,61) -- (405,69)(425,61) -- (425,69)(445,61) -- (445,69)(465,61) -- (465,69)(485,61) -- (485,69)(505,61) -- (505,69)(525,61) -- (525,69)(545,61) -- (545,69)(565,61) -- (565,69)(585,61) -- (585,69)(605,61) -- (605,69) ;
\draw [shift={(85,65)}, rotate = 180] [color={rgb, 255:red, 0; green, 0; blue, 0 }  ][line width=0.75]    (0,5.59) -- (0,-5.59)   ;
\draw    (260,50) -- (245,50) ;
\draw    (260,80) -- (245,80) ;
\draw    (405,25) -- (405,45) ;
\draw    (565,25) -- (565,45) ;
\draw    (405,35) -- (565,35) ;
\draw [shift={(565,35)}, rotate = 180] [color={rgb, 255:red, 0; green, 0; blue, 0 }  ][line width=0.75]    (0,5.59) -- (0,-5.59)   ;
\draw [shift={(405,35)}, rotate = 180] [color={rgb, 255:red, 0; green, 0; blue, 0 }  ][line width=0.75]    (0,5.59) -- (0,-5.59)   ;
\draw [color={rgb, 255:red, 255; green, 10; blue, 40 }  ,draw opacity=1 ][line width=2.25]    (505,65) -- (525,65) ;
\draw [color={rgb, 255:red, 255; green, 10; blue, 40 }  ,draw opacity=1 ][line width=2.25]    (185,65) -- (205,65) ;

\draw (316,12.4) node [anchor=north west][inner sep=0.75pt]    {$2^{\ell -1}$};
\draw (426,82.4) node [anchor=north west][inner sep=0.75pt]    {$\mathrm{I}$};
\draw (301,97.4) node [anchor=north west][inner sep=0.75pt]    {$\frac{15}{8} 2^{\ell -2}$};
\draw (486,12.4) node [anchor=north west][inner sep=0.75pt]    {$2^{\ell -1}$};
\draw (506,77.4) node [anchor=north west][inner sep=0.75pt]  [font=\scriptsize,color={rgb, 255:red, 255; green, 10; blue, 40 }  ,opacity=1 ]  {$\frac{2^{\ell }}{16}$};

\end{tikzpicture}

%% file: Fig_contour_2d.tex
\tikzset{every picture/.style={line width=0.75pt}} 

\begin{tikzpicture}[x=0.75pt,y=0.75pt,yscale=-1,xscale=1]

\draw  [color={rgb, 255:red, 255; green, 10; blue, 40 }  ,draw opacity=1 ][fill={rgb, 255:red, 128; green, 128; blue, 128 }  ,fill opacity=1 ][line width=1.5]  (114.55,40.13) .. controls (125.47,35.3) and (220.51,20.81) .. (246.18,38.68) .. controls (271.85,56.54) and (269.12,79.72) .. (277.31,92.04) .. controls (285.51,104.35) and (255.74,139.72) .. (230.89,146.36) .. controls (206.04,153) and (114.2,144.93) .. (108,136.7) .. controls (101.79,128.48) and (103.63,44.95) .. (114.55,40.13) -- cycle ;
\draw  [color={rgb, 255:red, 255; green, 10; blue, 40 }  ,draw opacity=1 ][fill={rgb, 255:red, 255; green, 255; blue, 255 }  ,fill opacity=1 ][line width=1.5]  (118.37,46.89) .. controls (129.3,42.06) and (151.15,47.61) .. (152.24,53.65) .. controls (153.33,59.68) and (152.24,86.24) .. (149.51,95.17) .. controls (146.78,104.11) and (127.66,150.22) .. (116.74,135.74) .. controls (105.81,121.25) and (107.45,51.71) .. (118.37,46.89) -- cycle ;
\draw  [color={rgb, 255:red, 74; green, 144; blue, 226 }  ,draw opacity=1 ][fill={rgb, 255:red, 255; green, 255; blue, 255 }  ,fill opacity=1 ][line width=1.5]  (163.71,39.16) .. controls (169.72,33.37) and (235.26,32.16) .. (246.73,43.99) .. controls (258.2,55.82) and (266.88,79.86) .. (265.3,98.55) .. controls (263.72,117.25) and (250.01,131.39) .. (240.72,136.22) .. controls (231.44,141.05) and (164.8,147.81) .. (158.25,137.67) .. controls (151.69,127.53) and (157.7,44.95) .. (163.71,39.16) -- cycle ;
\draw  [color={rgb, 255:red, 0; green, 0; blue, 0 }  ,draw opacity=1 ][fill={rgb, 255:red, 155; green, 155; blue, 155 }  ,fill opacity=0.3 ][line width=0.75]  (198.15,77.36) .. controls (202.03,74.49) and (204.52,73.8) .. (207.26,73.82) .. controls (210,73.84) and (214.61,76.94) .. (215.06,78.24) .. controls (215.5,79.55) and (221,87.69) .. (218.68,90.77) .. controls (216.35,93.85) and (208.99,94.88) .. (203.25,94.42) .. controls (197.51,93.96) and (197.25,91.67) .. (196.55,90.01) .. controls (195.84,88.35) and (194.28,80.23) .. (198.15,77.36) -- cycle ;
\draw  [color={rgb, 255:red, 0; green, 0; blue, 0 }  ,draw opacity=1 ][fill={rgb, 255:red, 255; green, 255; blue, 255 }  ,fill opacity=1 ][line width=0.75]  (203,79.8) .. controls (206.88,76.93) and (211.27,78.57) .. (212.05,80.83) .. controls (212.82,83.09) and (211.27,85.14) .. (208.95,88.22) .. controls (206.62,91.3) and (202.38,94.32) .. (200.67,90.28) .. controls (198.97,86.23) and (199.12,82.68) .. (203,79.8) -- cycle ;
\draw  [color={rgb, 255:red, 255; green, 10; blue, 40 }  ,draw opacity=1 ][fill={rgb, 255:red, 128; green, 128; blue, 128 }  ,fill opacity=1 ][line width=1.5]  (277.17,129.31) .. controls (281.54,120.29) and (285.4,92.99) .. (292.46,99.53) .. controls (299.53,106.06) and (301.71,124.09) .. (298.65,129.47) .. controls (295.59,134.84) and (273.67,171.09) .. (262.75,156.61) .. controls (251.83,142.12) and (272.8,138.32) .. (277.17,129.31) -- cycle ;
\draw  [color={rgb, 255:red, 74; green, 144; blue, 226 }  ,draw opacity=1 ][fill={rgb, 255:red, 255; green, 255; blue, 255 }  ,fill opacity=1 ][line width=1.5]  (282.08,130.77) .. controls (286.45,121.76) and (284.48,100.26) .. (291.54,106.79) .. controls (298.61,113.33) and (296.06,127.17) .. (293,132.54) .. controls (289.94,137.92) and (273.34,157.98) .. (266.78,151.86) .. controls (260.23,145.74) and (277.71,139.79) .. (282.08,130.77) -- cycle ;
\draw  [color={rgb, 255:red, 0; green, 0; blue, 0 }  ,draw opacity=1 ][fill={rgb, 255:red, 155; green, 155; blue, 155 }  ,fill opacity=0.3 ][line width=0.75]  (23.56,134.2) .. controls (34.48,129.37) and (37.76,132.26) .. (38.85,135.64) .. controls (39.94,139.02) and (37.17,144.58) .. (35.03,148.68) .. controls (32.88,152.79) and (20.21,160.27) .. (17.51,152.79) .. controls (14.82,145.3) and (12.63,139.02) .. (23.56,134.2) -- cycle ;
\draw    (340,10) -- (340,190) ;
\draw  [color={rgb, 255:red, 74; green, 144; blue, 226 }  ,draw opacity=1 ][fill={rgb, 255:red, 155; green, 155; blue, 155 }  ,fill opacity=1 ][line width=0.75]  (382.98,45.6) .. controls (388.98,39.8) and (454.53,38.6) .. (466,50.43) .. controls (477.47,62.26) and (486.15,86.3) .. (484.57,104.99) .. controls (482.99,123.69) and (469.27,137.83) .. (459.99,142.66) .. controls (450.7,147.49) and (384.07,154.25) .. (377.51,144.11) .. controls (370.96,133.97) and (376.97,51.39) .. (382.98,45.6) -- cycle ;
\draw  [color={rgb, 255:red, 74; green, 144; blue, 226 }  ,draw opacity=1 ][fill={rgb, 255:red, 155; green, 155; blue, 155 }  ,fill opacity=1 ][line width=0.75]  (501.34,124.47) .. controls (505.71,115.46) and (503.75,93.95) .. (510.81,100.49) .. controls (517.88,107.02) and (515.33,120.87) .. (512.27,126.24) .. controls (509.21,131.62) and (492.61,151.67) .. (486.05,145.56) .. controls (479.5,139.44) and (496.98,133.49) .. (501.34,124.47) -- cycle ;
\draw  [color={rgb, 255:red, 255; green, 10; blue, 40 }  ,draw opacity=1 ][fill={rgb, 255:red, 128; green, 128; blue, 128 }  ,fill opacity=1 ][line width=1.5]  (146.4,14.41) .. controls (151.73,23.74) and (146.93,30.92) .. (134.67,30.67) .. controls (122.4,30.41) and (127.73,23.74) .. (131.73,16.41) .. controls (135.73,9.08) and (141.07,5.08) .. (146.4,14.41) -- cycle ;

\draw (131,152.4) node [anchor=north west][inner sep=0.75pt]    {$\gamma $};
\draw (405.67,83.73) node [anchor=north west][inner sep=0.75pt]    {$\Int_{-}( \gamma )$};
\draw (15.33,158.73) node [anchor=north west][inner sep=0.75pt]  [font=\scriptsize]  {$\gamma _{1}$};
\draw (210,96.19) node [anchor=north west][inner sep=0.75pt]  [font=\scriptsize]  {$\gamma _{2}$};

\end{tikzpicture}

%% file: Fig_C_hat.tex
\tikzset{every picture/.style={line width=0.75pt}} 

\begin{tikzpicture}[x=0.75pt,y=0.75pt,yscale=-1,xscale=1]

\draw  [color={rgb, 255:red, 0; green, 0; blue, 0 }  ,draw opacity=0 ][fill={rgb, 255:red, 198; green, 198; blue, 198 }  ,fill opacity=1 ] (250,110) -- (270,110) -- (270,130) -- (250,130) -- cycle ;
\draw  [color={rgb, 255:red, 155; green, 155; blue, 155 }  ,draw opacity=0.27 ] (110,90) -- (130,90) -- (130,110) -- (110,110) -- cycle ;
\draw  [color={rgb, 255:red, 155; green, 155; blue, 155 }  ,draw opacity=0.27 ] (110,110) -- (130,110) -- (130,130) -- (110,130) -- cycle ;
\draw  [color={rgb, 255:red, 155; green, 155; blue, 155 }  ,draw opacity=0.27 ] (110,130) -- (130,130) -- (130,150) -- (110,150) -- cycle ;
\draw  [color={rgb, 255:red, 155; green, 155; blue, 155 }  ,draw opacity=0.27 ] (110,150) -- (130,150) -- (130,170) -- (110,170) -- cycle ;
\draw  [color={rgb, 255:red, 155; green, 155; blue, 155 }  ,draw opacity=0.27 ] (110,170) -- (130,170) -- (130,190) -- (110,190) -- cycle ;
\draw  [color={rgb, 255:red, 155; green, 155; blue, 155 }  ,draw opacity=0.27 ] (110,190) -- (130,190) -- (130,210) -- (110,210) -- cycle ;
\draw  [color={rgb, 255:red, 155; green, 155; blue, 155 }  ,draw opacity=0.27 ] (110,210) -- (130,210) -- (130,230) -- (110,230) -- cycle ;
\draw  [color={rgb, 255:red, 155; green, 155; blue, 155 }  ,draw opacity=0.27 ] (250,90) -- (270,90) -- (270,110) -- (250,110) -- cycle ;
\draw  [color={rgb, 255:red, 155; green, 155; blue, 155 }  ,draw opacity=0.27 ] (250,130) -- (270,130) -- (270,150) -- (250,150) -- cycle ;
\draw  [color={rgb, 255:red, 155; green, 155; blue, 155 }  ,draw opacity=0.27 ] (250,150) -- (270,150) -- (270,170) -- (250,170) -- cycle ;
\draw  [color={rgb, 255:red, 155; green, 155; blue, 155 }  ,draw opacity=0.27 ] (250,170) -- (270,170) -- (270,190) -- (250,190) -- cycle ;
\draw  [color={rgb, 255:red, 155; green, 155; blue, 155 }  ,draw opacity=0.27 ] (250,190) -- (270,190) -- (270,210) -- (250,210) -- cycle ;
\draw  [color={rgb, 255:red, 155; green, 155; blue, 155 }  ,draw opacity=0.27 ] (250,210) -- (270,210) -- (270,230) -- (250,230) -- cycle ;
\draw  [color={rgb, 255:red, 155; green, 155; blue, 155 }  ,draw opacity=0.27 ] (250,230) -- (270,230) -- (270,250) -- (250,250) -- cycle ;
\draw  [color={rgb, 255:red, 155; green, 155; blue, 155 }  ,draw opacity=0.27 ] (230,90) -- (250,90) -- (250,110) -- (230,110) -- cycle ;
\draw  [color={rgb, 255:red, 155; green, 155; blue, 155 }  ,draw opacity=0.27 ] (230,110) -- (250,110) -- (250,130) -- (230,130) -- cycle ;
\draw  [color={rgb, 255:red, 155; green, 155; blue, 155 }  ,draw opacity=0.27 ] (230,130) -- (250,130) -- (250,150) -- (230,150) -- cycle ;
\draw  [color={rgb, 255:red, 155; green, 155; blue, 155 }  ,draw opacity=0.27 ] (230,150) -- (250,150) -- (250,170) -- (230,170) -- cycle ;
\draw  [color={rgb, 255:red, 155; green, 155; blue, 155 }  ,draw opacity=0.27 ] (230,170) -- (250,170) -- (250,190) -- (230,190) -- cycle ;
\draw  [color={rgb, 255:red, 155; green, 155; blue, 155 }  ,draw opacity=0.27 ] (230,190) -- (250,190) -- (250,210) -- (230,210) -- cycle ;
\draw  [color={rgb, 255:red, 155; green, 155; blue, 155 }  ,draw opacity=0.27 ] (230,210) -- (250,210) -- (250,230) -- (230,230) -- cycle ;
\draw  [color={rgb, 255:red, 155; green, 155; blue, 155 }  ,draw opacity=0.27 ] (230,230) -- (250,230) -- (250,250) -- (230,250) -- cycle ;
\draw  [color={rgb, 255:red, 155; green, 155; blue, 155 }  ,draw opacity=0.27 ] (210,90) -- (230,90) -- (230,110) -- (210,110) -- cycle ;
\draw  [color={rgb, 255:red, 155; green, 155; blue, 155 }  ,draw opacity=0.27 ] (210,110) -- (230,110) -- (230,130) -- (210,130) -- cycle ;
\draw  [color={rgb, 255:red, 155; green, 155; blue, 155 }  ,draw opacity=0.27 ] (210,130) -- (230,130) -- (230,150) -- (210,150) -- cycle ;
\draw  [color={rgb, 255:red, 155; green, 155; blue, 155 }  ,draw opacity=0.27 ] (210,150) -- (230,150) -- (230,170) -- (210,170) -- cycle ;
\draw  [color={rgb, 255:red, 155; green, 155; blue, 155 }  ,draw opacity=0.27 ] (210,170) -- (230,170) -- (230,190) -- (210,190) -- cycle ;
\draw  [color={rgb, 255:red, 155; green, 155; blue, 155 }  ,draw opacity=0.27 ] (210,190) -- (230,190) -- (230,210) -- (210,210) -- cycle ;
\draw  [color={rgb, 255:red, 155; green, 155; blue, 155 }  ,draw opacity=0.27 ] (210,230) -- (230,230) -- (230,250) -- (210,250) -- cycle ;
\draw  [color={rgb, 255:red, 155; green, 155; blue, 155 }  ,draw opacity=0.27 ] (190,90) -- (210,90) -- (210,110) -- (190,110) -- cycle ;
\draw  [color={rgb, 255:red, 155; green, 155; blue, 155 }  ,draw opacity=0.27 ] (190,110) -- (210,110) -- (210,130) -- (190,130) -- cycle ;
\draw  [color={rgb, 255:red, 155; green, 155; blue, 155 }  ,draw opacity=0.27 ] (190,130) -- (210,130) -- (210,150) -- (190,150) -- cycle ;
\draw  [color={rgb, 255:red, 155; green, 155; blue, 155 }  ,draw opacity=0.27 ] (190,150) -- (210,150) -- (210,170) -- (190,170) -- cycle ;
\draw  [color={rgb, 255:red, 155; green, 155; blue, 155 }  ,draw opacity=0.27 ] (190,170) -- (210,170) -- (210,190) -- (190,190) -- cycle ;
\draw  [color={rgb, 255:red, 155; green, 155; blue, 155 }  ,draw opacity=0.27 ] (190,190) -- (210,190) -- (210,210) -- (190,210) -- cycle ;
\draw  [color={rgb, 255:red, 155; green, 155; blue, 155 }  ,draw opacity=0.27 ] (190,210) -- (210,210) -- (210,230) -- (190,230) -- cycle ;
\draw  [color={rgb, 255:red, 155; green, 155; blue, 155 }  ,draw opacity=0.27 ] (190,230) -- (210,230) -- (210,250) -- (190,250) -- cycle ;
\draw  [color={rgb, 255:red, 155; green, 155; blue, 155 }  ,draw opacity=0.27 ] (170,90) -- (190,90) -- (190,110) -- (170,110) -- cycle ;
\draw  [color={rgb, 255:red, 155; green, 155; blue, 155 }  ,draw opacity=0.27 ] (170,110) -- (190,110) -- (190,130) -- (170,130) -- cycle ;
\draw  [color={rgb, 255:red, 155; green, 155; blue, 155 }  ,draw opacity=0.27 ] (170,130) -- (190,130) -- (190,150) -- (170,150) -- cycle ;
\draw  [color={rgb, 255:red, 155; green, 155; blue, 155 }  ,draw opacity=0.27 ] (170,150) -- (190,150) -- (190,170) -- (170,170) -- cycle ;
\draw  [color={rgb, 255:red, 155; green, 155; blue, 155 }  ,draw opacity=0.27 ] (170,170) -- (190,170) -- (190,190) -- (170,190) -- cycle ;
\draw  [color={rgb, 255:red, 155; green, 155; blue, 155 }  ,draw opacity=0.27 ] (170,190) -- (190,190) -- (190,210) -- (170,210) -- cycle ;
\draw  [color={rgb, 255:red, 155; green, 155; blue, 155 }  ,draw opacity=0.27 ] (170,210) -- (190,210) -- (190,230) -- (170,230) -- cycle ;
\draw  [color={rgb, 255:red, 155; green, 155; blue, 155 }  ,draw opacity=0.27 ] (170,230) -- (190,230) -- (190,250) -- (170,250) -- cycle ;
\draw  [color={rgb, 255:red, 155; green, 155; blue, 155 }  ,draw opacity=0.27 ] (150,90) -- (170,90) -- (170,110) -- (150,110) -- cycle ;
\draw  [color={rgb, 255:red, 155; green, 155; blue, 155 }  ,draw opacity=0.27 ] (150,110) -- (170,110) -- (170,130) -- (150,130) -- cycle ;
\draw  [color={rgb, 255:red, 155; green, 155; blue, 155 }  ,draw opacity=0.27 ] (150,130) -- (170,130) -- (170,150) -- (150,150) -- cycle ;
\draw  [color={rgb, 255:red, 155; green, 155; blue, 155 }  ,draw opacity=0.27 ] (150,150) -- (170,150) -- (170,170) -- (150,170) -- cycle ;
\draw  [color={rgb, 255:red, 155; green, 155; blue, 155 }  ,draw opacity=0.27 ] (150,170) -- (170,170) -- (170,190) -- (150,190) -- cycle ;
\draw  [color={rgb, 255:red, 155; green, 155; blue, 155 }  ,draw opacity=0.27 ] (150,190) -- (170,190) -- (170,210) -- (150,210) -- cycle ;
\draw  [color={rgb, 255:red, 155; green, 155; blue, 155 }  ,draw opacity=0.27 ] (150,210) -- (170,210) -- (170,230) -- (150,230) -- cycle ;
\draw  [color={rgb, 255:red, 155; green, 155; blue, 155 }  ,draw opacity=0.27 ] (150,230) -- (170,230) -- (170,250) -- (150,250) -- cycle ;
\draw  [color={rgb, 255:red, 155; green, 155; blue, 155 }  ,draw opacity=0.27 ] (130,90) -- (150,90) -- (150,110) -- (130,110) -- cycle ;
\draw  [color={rgb, 255:red, 155; green, 155; blue, 155 }  ,draw opacity=0.27 ] (130,110) -- (150,110) -- (150,130) -- (130,130) -- cycle ;
\draw  [color={rgb, 255:red, 155; green, 155; blue, 155 }  ,draw opacity=0.27 ] (130,130) -- (150,130) -- (150,150) -- (130,150) -- cycle ;
\draw  [color={rgb, 255:red, 155; green, 155; blue, 155 }  ,draw opacity=0.27 ] (130,150) -- (150,150) -- (150,170) -- (130,170) -- cycle ;
\draw  [color={rgb, 255:red, 155; green, 155; blue, 155 }  ,draw opacity=0.27 ] (130,170) -- (150,170) -- (150,190) -- (130,190) -- cycle ;
\draw  [color={rgb, 255:red, 155; green, 155; blue, 155 }  ,draw opacity=0.27 ] (130,190) -- (150,190) -- (150,210) -- (130,210) -- cycle ;
\draw  [color={rgb, 255:red, 155; green, 155; blue, 155 }  ,draw opacity=0.27 ] (130,210) -- (150,210) -- (150,230) -- (130,230) -- cycle ;
\draw  [color={rgb, 255:red, 155; green, 155; blue, 155 }  ,draw opacity=0.27 ] (130,230) -- (150,230) -- (150,250) -- (130,250) -- cycle ;
\draw  [color={rgb, 255:red, 155; green, 155; blue, 155 }  ,draw opacity=0.27 ] (270,90) -- (290,90) -- (290,110) -- (270,110) -- cycle ;
\draw  [color={rgb, 255:red, 155; green, 155; blue, 155 }  ,draw opacity=0.27 ] (270,110) -- (290,110) -- (290,130) -- (270,130) -- cycle ;
\draw  [color={rgb, 255:red, 155; green, 155; blue, 155 }  ,draw opacity=0.27 ] (270,130) -- (290,130) -- (290,150) -- (270,150) -- cycle ;
\draw  [color={rgb, 255:red, 155; green, 155; blue, 155 }  ,draw opacity=0.27 ] (270,150) -- (290,150) -- (290,170) -- (270,170) -- cycle ;
\draw  [color={rgb, 255:red, 155; green, 155; blue, 155 }  ,draw opacity=0.27 ] (270,170) -- (290,170) -- (290,190) -- (270,190) -- cycle ;
\draw  [color={rgb, 255:red, 155; green, 155; blue, 155 }  ,draw opacity=0.27 ] (270,190) -- (290,190) -- (290,210) -- (270,210) -- cycle ;
\draw  [color={rgb, 255:red, 155; green, 155; blue, 155 }  ,draw opacity=0.27 ] (270,210) -- (290,210) -- (290,230) -- (270,230) -- cycle ;
\draw  [color={rgb, 255:red, 155; green, 155; blue, 155 }  ,draw opacity=0.27 ] (270,230) -- (290,230) -- (290,250) -- (270,250) -- cycle ;
\draw  [color={rgb, 255:red, 155; green, 155; blue, 155 }  ,draw opacity=0.27 ] (290,90) -- (310,90) -- (310,110) -- (290,110) -- cycle ;
\draw  [color={rgb, 255:red, 155; green, 155; blue, 155 }  ,draw opacity=0.27 ] (290,110) -- (310,110) -- (310,130) -- (290,130) -- cycle ;
\draw  [color={rgb, 255:red, 155; green, 155; blue, 155 }  ,draw opacity=0.27 ] (290,130) -- (310,130) -- (310,150) -- (290,150) -- cycle ;
\draw  [color={rgb, 255:red, 155; green, 155; blue, 155 }  ,draw opacity=0.27 ] (290,150) -- (310,150) -- (310,170) -- (290,170) -- cycle ;
\draw  [color={rgb, 255:red, 155; green, 155; blue, 155 }  ,draw opacity=0.27 ] (290,170) -- (310,170) -- (310,190) -- (290,190) -- cycle ;
\draw  [color={rgb, 255:red, 155; green, 155; blue, 155 }  ,draw opacity=0.27 ] (290,190) -- (310,190) -- (310,210) -- (290,210) -- cycle ;
\draw  [color={rgb, 255:red, 155; green, 155; blue, 155 }  ,draw opacity=0.27 ] (290,210) -- (310,210) -- (310,230) -- (290,230) -- cycle ;
\draw  [color={rgb, 255:red, 155; green, 155; blue, 155 }  ,draw opacity=0.27 ] (290,230) -- (310,230) -- (310,250) -- (290,250) -- cycle ;
\draw  [color={rgb, 255:red, 155; green, 155; blue, 155 }  ,draw opacity=0.27 ] (310,90) -- (330,90) -- (330,110) -- (310,110) -- cycle ;
\draw  [color={rgb, 255:red, 155; green, 155; blue, 155 }  ,draw opacity=0.27 ] (310,110) -- (330,110) -- (330,130) -- (310,130) -- cycle ;
\draw  [color={rgb, 255:red, 155; green, 155; blue, 155 }  ,draw opacity=0.27 ] (310,130) -- (330,130) -- (330,150) -- (310,150) -- cycle ;
\draw  [color={rgb, 255:red, 155; green, 155; blue, 155 }  ,draw opacity=0.27 ] (310,150) -- (330,150) -- (330,170) -- (310,170) -- cycle ;
\draw  [color={rgb, 255:red, 155; green, 155; blue, 155 }  ,draw opacity=0.27 ] (310,170) -- (330,170) -- (330,190) -- (310,190) -- cycle ;
\draw  [color={rgb, 255:red, 155; green, 155; blue, 155 }  ,draw opacity=0.27 ] (310,190) -- (330,190) -- (330,210) -- (310,210) -- cycle ;
\draw  [color={rgb, 255:red, 155; green, 155; blue, 155 }  ,draw opacity=0.27 ] (310,210) -- (330,210) -- (330,230) -- (310,230) -- cycle ;
\draw  [color={rgb, 255:red, 155; green, 155; blue, 155 }  ,draw opacity=0.27 ] (310,230) -- (330,230) -- (330,250) -- (310,250) -- cycle ;
\draw  [color={rgb, 255:red, 155; green, 155; blue, 155 }  ,draw opacity=0.27 ] (330,90) -- (350,90) -- (350,110) -- (330,110) -- cycle ;
\draw  [color={rgb, 255:red, 155; green, 155; blue, 155 }  ,draw opacity=0.27 ] (330,110) -- (350,110) -- (350,130) -- (330,130) -- cycle ;
\draw  [color={rgb, 255:red, 155; green, 155; blue, 155 }  ,draw opacity=0.27 ] (330,130) -- (350,130) -- (350,150) -- (330,150) -- cycle ;
\draw  [color={rgb, 255:red, 155; green, 155; blue, 155 }  ,draw opacity=0.27 ] (330,150) -- (350,150) -- (350,170) -- (330,170) -- cycle ;
\draw  [color={rgb, 255:red, 155; green, 155; blue, 155 }  ,draw opacity=0.27 ] (330,170) -- (350,170) -- (350,190) -- (330,190) -- cycle ;
\draw  [color={rgb, 255:red, 155; green, 155; blue, 155 }  ,draw opacity=0.27 ] (330,190) -- (350,190) -- (350,210) -- (330,210) -- cycle ;
\draw  [color={rgb, 255:red, 155; green, 155; blue, 155 }  ,draw opacity=0.27 ] (330,210) -- (350,210) -- (350,230) -- (330,230) -- cycle ;
\draw  [color={rgb, 255:red, 155; green, 155; blue, 155 }  ,draw opacity=0.27 ] (330,230) -- (350,230) -- (350,250) -- (330,250) -- cycle ;
\draw  [color={rgb, 255:red, 155; green, 155; blue, 155 }  ,draw opacity=0.27 ] (350,90) -- (370,90) -- (370,110) -- (350,110) -- cycle ;
\draw  [color={rgb, 255:red, 155; green, 155; blue, 155 }  ,draw opacity=0.27 ] (350,110) -- (370,110) -- (370,130) -- (350,130) -- cycle ;
\draw  [color={rgb, 255:red, 155; green, 155; blue, 155 }  ,draw opacity=0.27 ] (350,130) -- (370,130) -- (370,150) -- (350,150) -- cycle ;
\draw  [color={rgb, 255:red, 155; green, 155; blue, 155 }  ,draw opacity=0.27 ] (350,150) -- (370,150) -- (370,170) -- (350,170) -- cycle ;
\draw  [color={rgb, 255:red, 155; green, 155; blue, 155 }  ,draw opacity=0.27 ] (350,170) -- (370,170) -- (370,190) -- (350,190) -- cycle ;
\draw  [color={rgb, 255:red, 155; green, 155; blue, 155 }  ,draw opacity=0.27 ] (350,190) -- (370,190) -- (370,210) -- (350,210) -- cycle ;
\draw  [color={rgb, 255:red, 155; green, 155; blue, 155 }  ,draw opacity=0.27 ] (350,210) -- (370,210) -- (370,230) -- (350,230) -- cycle ;
\draw  [color={rgb, 255:red, 155; green, 155; blue, 155 }  ,draw opacity=0.27 ] (350,230) -- (370,230) -- (370,250) -- (350,250) -- cycle ;
\draw  [color={rgb, 255:red, 155; green, 155; blue, 155 }  ,draw opacity=0.27 ] (370,90) -- (390,90) -- (390,110) -- (370,110) -- cycle ;
\draw  [color={rgb, 255:red, 155; green, 155; blue, 155 }  ,draw opacity=0.27 ] (370,110) -- (390,110) -- (390,130) -- (370,130) -- cycle ;
\draw  [color={rgb, 255:red, 155; green, 155; blue, 155 }  ,draw opacity=0.27 ] (370,130) -- (390,130) -- (390,150) -- (370,150) -- cycle ;
\draw  [color={rgb, 255:red, 155; green, 155; blue, 155 }  ,draw opacity=0.27 ] (370,150) -- (390,150) -- (390,170) -- (370,170) -- cycle ;
\draw  [color={rgb, 255:red, 155; green, 155; blue, 155 }  ,draw opacity=0.27 ] (370,170) -- (390,170) -- (390,190) -- (370,190) -- cycle ;
\draw  [color={rgb, 255:red, 155; green, 155; blue, 155 }  ,draw opacity=0.27 ] (370,190) -- (390,190) -- (390,210) -- (370,210) -- cycle ;
\draw  [color={rgb, 255:red, 155; green, 155; blue, 155 }  ,draw opacity=0.27 ] (370,210) -- (390,210) -- (390,230) -- (370,230) -- cycle ;
\draw  [color={rgb, 255:red, 155; green, 155; blue, 155 }  ,draw opacity=0.27 ] (370,230) -- (390,230) -- (390,250) -- (370,250) -- cycle ;
\draw  [color={rgb, 255:red, 155; green, 155; blue, 155 }  ,draw opacity=0.27 ] (390,90) -- (410,90) -- (410,110) -- (390,110) -- cycle ;
\draw  [color={rgb, 255:red, 155; green, 155; blue, 155 }  ,draw opacity=0.27 ] (390,110) -- (410,110) -- (410,130) -- (390,130) -- cycle ;
\draw  [color={rgb, 255:red, 155; green, 155; blue, 155 }  ,draw opacity=0.27 ] (390,130) -- (410,130) -- (410,150) -- (390,150) -- cycle ;
\draw  [color={rgb, 255:red, 155; green, 155; blue, 155 }  ,draw opacity=0.27 ] (390,150) -- (410,150) -- (410,170) -- (390,170) -- cycle ;
\draw  [color={rgb, 255:red, 155; green, 155; blue, 155 }  ,draw opacity=0.27 ] (390,170) -- (410,170) -- (410,190) -- (390,190) -- cycle ;
\draw  [color={rgb, 255:red, 155; green, 155; blue, 155 }  ,draw opacity=0.27 ] (390,190) -- (410,190) -- (410,210) -- (390,210) -- cycle ;
\draw  [color={rgb, 255:red, 155; green, 155; blue, 155 }  ,draw opacity=0.27 ] (390,210) -- (410,210) -- (410,230) -- (390,230) -- cycle ;
\draw  [color={rgb, 255:red, 155; green, 155; blue, 155 }  ,draw opacity=0.27 ] (390,230) -- (410,230) -- (410,250) -- (390,250) -- cycle ;
\draw  [color={rgb, 255:red, 155; green, 155; blue, 155 }  ,draw opacity=0.27 ] (410,90) -- (430,90) -- (430,110) -- (410,110) -- cycle ;
\draw  [color={rgb, 255:red, 155; green, 155; blue, 155 }  ,draw opacity=0.27 ] (410,110) -- (430,110) -- (430,130) -- (410,130) -- cycle ;
\draw  [color={rgb, 255:red, 155; green, 155; blue, 155 }  ,draw opacity=0.27 ] (410,130) -- (430,130) -- (430,150) -- (410,150) -- cycle ;
\draw  [color={rgb, 255:red, 155; green, 155; blue, 155 }  ,draw opacity=0.27 ] (410,150) -- (430,150) -- (430,170) -- (410,170) -- cycle ;
\draw  [color={rgb, 255:red, 155; green, 155; blue, 155 }  ,draw opacity=0.27 ] (410,170) -- (430,170) -- (430,190) -- (410,190) -- cycle ;
\draw  [color={rgb, 255:red, 155; green, 155; blue, 155 }  ,draw opacity=0.27 ] (410,190) -- (430,190) -- (430,210) -- (410,210) -- cycle ;
\draw  [color={rgb, 255:red, 155; green, 155; blue, 155 }  ,draw opacity=0.27 ] (410,210) -- (430,210) -- (430,230) -- (410,230) -- cycle ;
\draw  [color={rgb, 255:red, 155; green, 155; blue, 155 }  ,draw opacity=0.27 ] (410,230) -- (430,230) -- (430,250) -- (410,250) -- cycle ;
\draw  [fill={rgb, 255:red, 128; green, 128; blue, 128 }  ,fill opacity=1 ] (290,110) -- (410,110) -- (410,230) -- (290,230) -- cycle ;
\draw  [color={rgb, 255:red, 155; green, 155; blue, 155 }  ,draw opacity=0.27 ] (210,210) -- (230,210) -- (230,230) -- (210,230) -- cycle ;
\draw  [fill={rgb, 255:red, 128; green, 128; blue, 128 }  ,fill opacity=1 ] (130,110) -- (250,110) -- (250,230) -- (130,230) -- cycle ;
\draw  [color={rgb, 255:red, 0; green, 0; blue, 0 }  ,draw opacity=1 ][fill={rgb, 255:red, 198; green, 198; blue, 198 }  ,fill opacity=1 ] (270,110) -- (290,110) -- (290,130) -- (270,130) -- cycle ;
\draw  [color={rgb, 255:red, 255; green, 255; blue, 255 }  ,draw opacity=1 ] (270,90) -- (430,90) -- (430,250) -- (270,250) -- cycle ;
\draw  [color={rgb, 255:red, 255; green, 255; blue, 255 }  ,draw opacity=1 ] (110,90) -- (270,90) -- (270,250) -- (110,250) -- cycle ;
\draw  [dash pattern={on 4.5pt off 4.5pt}] (270,90) -- (430,90) -- (430,250) -- (270,250) -- cycle ;
\draw  [dash pattern={on 4.5pt off 4.5pt}] (110,90) -- (270,90) -- (270,250) -- (110,250) -- cycle ;
\draw  [color={rgb, 255:red, 0; green, 0; blue, 0 }  ,draw opacity=1 ] (250,110) -- (270,110) -- (270,130) -- (250,130) -- cycle ;
\draw  [fill={rgb, 255:red, 0; green, 0; blue, 0 }  ,fill opacity=1 ] (110,230) -- (130,230) -- (130,250) -- (110,250) -- cycle ;

\draw (111,254.4) node [anchor=north west][inner sep=0.75pt]  [font=\footnotesize]  {$2^{r\ell}$};
\draw (232,253.4) node [anchor=north west][inner sep=0.75pt]  [font=\normalsize]  {$\mathrm{C}_{\ell +1}$};
\draw (392,253.4) node [anchor=north west][inner sep=0.75pt]  [font=\normalsize]  {$\mathrm{C}_{\ell +1}^{\prime }$};
\draw (152,193.4) node [anchor=north west][inner sep=0.75pt]  [font=\normalsize]  {$\widehat{\mathrm{C}}_{\ell+1}$};
\draw (312,193.4) node [anchor=north west][inner sep=0.75pt]  [font=\normalsize]  {$\widehat{\mathrm{C}}^\prime_{\ell+1}$};
\draw (252,113.4) node [anchor=north west][inner sep=0.75pt]  [font=\scriptsize]  {$\mathrm{C}_{\ell}$};
\draw (274,112.4) node [anchor=north west][inner sep=0.75pt]  [font=\scriptsize]  {$\mathrm{C}_{\ell }^{\prime }$};

\end{tikzpicture}